\documentclass[11pt]{article}

\usepackage{amssymb}
\usepackage{amsmath}
\usepackage{bbm,bm}
\usepackage{multirow}	
\usepackage{enumerate}
\usepackage{graphicx}
\usepackage{color}
\usepackage[dvipsnames,table]{xcolor}
\usepackage{hyperref}
\usepackage{stmaryrd}
\usepackage{amsthm}
\usepackage{tikz}
\usepackage{enumitem}
\usepackage{array}
\newcolumntype{C}[1]{>{\centering\arraybackslash}m{#1}}
\usepackage{rotating}
\usepackage{mathrsfs}

\newcommand\qan{{\quad\hbox{and}\quad}}
\newcommand\qanf{{\quad\hbox{and}}}
\newcommand\qin{{\quad\hbox{in}\quad}}
\newcommand\qon{{\quad\hbox{on}\quad}}
\renewcommand\div{{\mathrm{div}}}
\newcommand\bdiv{{\mathbf{div}}}
\newcommand\bta{{\boldsymbol\tau}}
\newcommand\bze{{\boldsymbol\zeta}}
\newcommand\bet{{\boldsymbol\eta}}
\newcommand\bchi{{\boldsymbol\chi}}

\newcommand\bgam{{\boldsymbol\gamma}}
\newcommand\bom{{\boldsymbol\omega}}
\newcommand\bdel{{\boldsymbol\delta}}
\newcommand\tr{{\mathrm{tr}}}
\newcommand\bsi{{\boldsymbol\sigma}}
\newcommand\bxi{{\boldsymbol\xi}}
\newcommand\brho{{\boldsymbol\rho}}
\newcommand\bn{{\mathbf n}}
\newcommand\bu{{\mathbf u}}
\newcommand\bz{{\mathbf z}}
\newcommand\bv{{\mathbf v}}

\newcommand\bw{{\mathbf w}}
\newcommand\R{{\mathrm R}}
\newcommand\disp{\displaystyle}
\newcommand\Pcal{{\mathcal P}}
\newcommand{\Pcalbf}{{\boldsymbol{\mathcal P}}}
\newcommand{\Pcalbb}{\mathcal{P}\hspace{-1.5ex}\mathcal{P}}
\renewcommand\L{{\mathrm L}}

\newcommand\bg{{\mathbf g}}
\newcommand\by{{\mathbf y}}
\newcommand\dist{{\mathrm{dist}}}

		\usepackage{xcolor}
		\hypersetup{
			colorlinks,
			linkcolor={red!50!black},
			citecolor={blue!50!black},
			urlcolor={blue!80!black}
		}

		\numberwithin{equation}{section}

		\newtheorem{theorem}{Theorem}[section]
		
		\newtheorem{lemma}[theorem]{Lemma}
		
		\theoremstyle{remark}

\begin{document}
			
					\title{Mixed virtual element approximation for the five-ﬁeld formulation of the steady Boussinesq problem with temperature-dependent parameters			
							}
					
	         \author{
										{\sc Zeinab Gharibi}\thanks{CI$^2$MA and GIMNAP-Departamento de Matem\'atica, Universidad del B\'io-B\'io, Casilla 5-C, Concepci\'on, Chile, email: {\tt zgharibi@ubiobio.cl}.}
									}
					
					\date{ }
					
\maketitle
\begin{abstract}
						
\noindent
In this work we develop recent research on the fully mixed virtual element method (mixed-VEM) based on the Banach space for the stationary Boussinesq equation to suggest and analyze a new mixed-VEM for the stationary two-dimensional Boussinesq equation with temperature-dependent parameters in terms of the pseudostress, vorticity, velocity, pseudoheat vector and temperature fields. The well-posedness of continuous formulation is analyzed utilizing a fixed-point strategy, a smallness assumption on the data, and some additional regularities on the solution.  The discretization for the mentioned variables is based on the coupling $ \mathbb{H}(\bdiv_{6/5}) $- and $ \mathbf{H}(\div_{6/5}) $-conforming virtual element techniques.
The proposed scheme is rewritten as an equivalent fixed point operator equation, so that its existence and stability estimates have been proven. In addition, an a priori convergence analysis is established by utilizing the Céa estimate and a suitable assumption on data for all variables in their natural norms showing an optimal rate of convergence. Finally, several numerical examples are presented to illustrate the performance of the proposed method.

\end{abstract}
					
					

\section{Introduction}\label{sec.1}
Various types of free convection are present in nature and in industry, including mantle convection, stratified oceanic flows, and onboard cooling systems for electronic devices. These processes can be modeled by coupling the continuity and momentum equations (Navier--Stokes) with the energy equations using the Boussinesq approximation, assuming that the fluid density is constant in all terms of the equations except for the buoyancy term in the momentum equation, where linearity dependency is taken into account. Despite this, other properties of the fluid may also be affected by temperature, such as viscosity and thermal conductivity, which are significant factors in the flow of oils and nanofluids. In this regard, a short list of a variety of numerical approximations for the generalizations of Boussinesq equations with temperature-dependent parameters, include finite element method (FEM) \cite{bmp-MAN-1995,da-JSC-2016,sl-JCP-2017,tt-NM-2005,wsf-JCP-2017,hylh-CMA-2022,oz-CAM-2017}, mixed and augmented-mixed FEM approaches \cite{agor-NHM-2020,ag-CMAM-2020,ago-JSC-2019,agr-ESAIM-2015,oqs-IMA-2014}, and virtual element method \cite{avv-IMA-2022}.

Specifically, the authors in \cite{oz-CAM-2017,oqs-IMA-2014} studied finite element methods for solving the generalized Boussinesq equation based on primal formulations. First, the problem is addressed in its primitive variables, while in the second, to achieve conformity, an extra variable called normal heat flux is introduced through the boundary.
A recent study in \cite{ago-JSC-2019} discusses primal and mixed formulations of the energy (with a space-dependent thermal conductivity) and the momentum and continuity (with temperature-dependent viscosity) equations, respectively. This study introduces pseudostress and vorticity tensors as auxiliary variables. In this way, the authors showed the optimally convergent method, using the fixed-point strategy proposed in \cite{agr-ESAIM-2015,cgo-NMPDE-2016}, and employing Raviart--Thomas, Lagrange, and discontinuous elements to approximate the pseudostress, the velocity, and temperature, and discontinuous elements for the vorticity and normal heat flux, respectively, whenever the exact solution is smooth enough, and the data is sufficiently small.
Nonetheless, the analysis of the method is restricted to the two-dimensional case due to the variable viscosity. This drawback has been overcome in the recent works \cite{agor-NHM-2020,ag-CMAM-2020}, by adding the rate of strain tensor to the previous variables of fluid, i.e., pseudostress, the velocity and vorticity. The first one deals with the primal formulation of the energy equation, while the second one introduces an augmented fully-mixed finite element method for the problem. In both cases, the analysis is based on the introduction of a pseudostress tensor relating the diffusive and convective terms with the pressure and it is proven optimal convergence. However, this approach is not without its drawbacks. One such drawback is that it can be computationally expensive.

In recent years, a significant number of researchers have focused on extending the virtual element method (VEM) \cite{bfm-M2AN-2014,bbcmmr-MMMAS-2013,cg-IMA.NA-2017,cgs-SIAM.NA-2018,gms-C-2018} and mixed VEMs to solve linear and nonlinear problems in fluid mechanics using velocity-pressure or primal formulations \cite{blv-MMNA-2017,blv-SIAM.NA-2018} and pseudostress-velocity-based formulations \cite{cg-IMA.NA-2017,cgs-SIAM.NA-2018,gs-M3AS-2021,gms-JCM-2021,gms-C-2018}.
Here, we specifically focus on the second category. Especially, a mixed-VEM based on the pseudostress-velocity formulation for the Stokes equation was developed in \cite{cg-IMANUM-2017} as a natural extension of previously proposed dual-mixed-FEMs for elliptic equations in divergence form. The discussion in \cite{cg-IMANUM-2017} includes the virtual finite element subspaces to be employed, the associated interpolation operators, and their respective approximation properties. Moreover, the uniform boundedness of the resulting family of projectors and their corresponding approximation properties are established there.
In this manner, the classical Babuska--Brezzi theory is employed to establish the solvability of the discrete scheme and obtain the corresponding a priori error estimates for both the virtual solution and its fully computable projection.
 Later on, this approach was extended to other models in fluid mechanics, such as quasi-Newtonian Stokes \cite{cgs-SIAM.NA-2018}, linear and nonlinear Brinkman \cite{cgs-MMMAC-2017,gms-C-2018}, Navier--Stokes \cite{gms-M3AS-2018,gs-M3AS-2021}, and Boussinesq \cite{gms-JCM-2021}, where different formulations were considered. 
In particular, the Banach spaces-based approach, usually employed for solving various nonlinear problems in continuum mechanics via primal and mixed finite element methods, is extended for the first time in \cite{gs-M3AS-2021} to the VEM framework and its respective applications. More precisely, an $\mathrm L^p$ spaces-based mixed VEM for a pseudostress--velocity formulation of the two-dimensional Navier-Stokes equations with Dirichlet boundary conditions is proposed and analyzed there.
The simplicity of the resulting mixed-VEM scheme is reflected by the absence of augmented terms, in contrast to previous work on this model \cite{gms-M3AS-2018}. Additionally, only a virtual element space for the pseudostress tensor is required since the non-virtual but explicit subspace, given by the piecewise polynomial vectors of degree $\le k$, is employed to approximate the velocity.
Later on,
as an extension of \cite{gs-M3AS-2021}, the authors \cite{gg-Pre-2023} introduced and analyzed a fully mixed VEM for the Boussinesq problem.
Specifically, they elected the equations and the associated variational formulation from \cite{gs-M3AS-2021}, adapting the approach to propose a novel $\mathrm L^p$ spaces-based fully mixed-VEM for the Boussinesq problem, which, to the best of our knowledge, is introduced for the first time.

According to the above discussion and aiming to broaden the scope of mixed-VEM applicability to nonlinear models in fluid mechanics featuring variable coefficients, we extend the approach outlined in \cite{gg-Pre-2023} to solve the Boussinesq problem with parameters dependent on temperature.
 To achieve this, we consider the mixed formulations of the Navier--Stokes and energy equations by defining the pseudostress and vorticity tensors, and pseudoheat vector as additional variables. These, together with the temperature and velocity, comprise the unknowns of our problem. We then adapt the approach from \cite{gg-Pre-2023} to propose a fully mixed-VEM for the Boussinesq problem with temperature-dependent parameters.
In this way, we first present the continuous formulation of the problem. Then, we discretize the corresponding problem by applying tools proposed in \cite{gg-Pre-2023}. 
More precisely, the pseudostress and pseudoheat vector are approximated by the virtual element subspaces of $ \mathbb{H}(\bdiv_{6/5};\Omega) $ and $ \mathbf{H}(\div_{6/5};\Omega) $, respectively, whereas the non-virtual element subspaces of $ \mathbf{L}^{6}(\Omega) $, $ \mathbb{L}_{\mathtt{skew}}^{2}(\Omega) $, and $ \mathrm{L}^{6}(\Omega) $ are employed to approximate the velocity, vorticity tensor and the temperature, respectively. Thus, we develop the corresponding solvability analysis for discrete problem and derive the convergence analysis with optimal error estimates.\\ 
	
\textbf{Outline.}
The rest of this work is organized as follows. 
In Sec. \ref{sec.2}, we set the model of interest, define the auxiliary
unknowns to be considered, and introduce the corresponding variational formulation. The solvability is analyzed in Sec. \ref{sec.3} within
a Banach framework. In fact, apart from establishing the boundedness properties of the bilinear (and trilinear) forms, we demonstrate the equivalence of the continuous formulation with a fixed-point equation and prove the well-definedness of the corresponding operator using Banach-Nečas-Babuška Theorem. Finally, we apply the Banach fixed-point theorem to confirm the existence of a unique solution.
Next, in Sec. \ref{sec.4} we present the VE discretization, introducing the mesh entities, the degrees
of freedom, the construction of mixed VE spaces, and introducing the discerte multilinear forms.
In Sec. \ref{sec.5}, we prove the existence of the discrete solution by utilizing a discrete version of the fixed point strategy developed in Section \ref{sec.3} for the continuous case, providing the usual estimates concerning the bilinear and trilinear forms involved, and establishing the discrete inf-sup condition. 
In addition, in Sec. \ref{sec.6} the Strang-type lemmas in Banach spaces are applied to derive a priori error estimates, and corresponding rates of convergence are established. Finally, the performance of the method is illustrated in Section \ref{sec.7} with several numerical examples.


\subsection{Notations}
					
For any vector fields $ \mathbf{v}=(v_1, v_2)^{\tt t} $ and $ \mathbf{w}=(w_1, w_2)^{\tt t} $, we set the gradient, divergence,
and tensor product operators as
\[
\boldsymbol{\nabla}\mathbf{v}:=(\nabla v_1, \nabla v_2), \quad \div(\mathbf{v}):=\partial_x v_1 + \partial_y v_2, 
\qan \mathbf{v}\otimes\mathbf{w}\,:=\, (v_i \, w_j)_{i,j=1,2} \,,
\]
respectively. In addition, denoting by $\mathbb{I}$ the identity matrix of $\R^2$, for any tensor fields 
$\bta=(\tau_{ij}),~\boldsymbol{\zeta}=(\zeta_{ij})\in \mathrm{R}^{2\times 2}$, we write as usual
\[
\bta^{\mathtt{t}}:=(\tau_{ji}), \quad \operatorname{tr}(\bta):=\tau_{11}+\tau_{22},\quad 
\bta^{\mathtt{d}}:= \bta-\dfrac{1}{2}\operatorname{tr}(\bta)\mathbb{I}\,,
\qan \bta: \boldsymbol{\zeta}:= \sum_{i,j=1}^{2}\tau_{ij}\zeta_{ij} \,,
\]
which corresponds, respectively, to the transpose, the trace, and the deviator tensor of $ \bta $, and to the
tensorial product between $ \bta $ and $ \boldsymbol{\zeta} $. Next, given a bounded domain $ \mathcal{D}\subset \mathrm{R}^{2} $ 
with boundary $ \partial \mathcal{D} $, we let $ \mathbf{n} $ be the outward unit normal vector on $ \partial \mathcal{D} $. Also, 
given $ r\geq 0 $ and $ 1< p\leq \infty $, we let $\mathrm W^{r,p}(\mathcal{D}) $ be
the standard Sobolev space with norm $ \Vert\cdot\Vert_{r,p;\mathcal{D}} $ and seminorm $ \vert\cdot\vert_{r,p;\mathcal{D}} $. In particular, for $ r=0 $ we let
$\mathrm L^{p}(\mathcal{D}) := \mathrm W^{0,p}(\mathcal{D}) $ be the usual Lebesgue space, and for $ p=2 $ we let 
$\mathrm H^{s}(\mathcal{D}) := \mathrm W^{r,2}(\mathcal{D})$ be the classical
Hilbertian Sobolev space with norm $ \Vert\cdot\Vert_{s,\mathcal{D}} $ and seminorm $ \vert\cdot\vert_{s,\mathcal{D}} $. 
Furthermore, given a generic scalar functional space M, we let $\mathbf{M}$ and $ \mathbb{M} $ be its vector and 
tensorial counterparts, respectively, whose norms and seminorms are denoted exactly as those of M. 
On the other hand, given $t\in (1,+\infty)$, and letting $ \bdiv $ (resp. $ \operatorname{\mathbf{rot}} $) 
be the usual divergence operator $ \div $ (resp. rotational operator $ \operatorname{rot} $) acting along the 
rows of a given tensor, we introduce the non-standard Banach spaces
\[
\mathbf{H}(\Diamond_t;\Omega)~:=~\Big\{\mathbf{v}\in \mathbf{L}^{2}(\Omega):\quad  
\Diamond(\mathbf{v})\in \mathrm{L}^t(\Omega)\Big\}\qquad\forall\,\Diamond\in \{\div,\operatorname{rot}\}\,,
\]
and
\[
\mathbb{H}(\blacklozenge_t;\Omega)~:=~\Big\{\bta\in \mathbb{L}^{2}(\Omega):\quad  
\blacklozenge(\bta)\in \mathbf{L}^t(\Omega)\Big\}\qquad\forall\, \blacklozenge\in \{\bdiv,\operatorname{\textbf{rot}}\}\,,
\]
equipped with the usual norms
\[
\|\circ\|_{\div_t; \Omega} ~:=~\|\circ\|_{0, \Omega} + 
\|\div(\circ)\|_{0,t; \Omega}, \quad\quad \forall\, \circ\in \mathbf{H}(\div_t; \Omega),
\]
and
\[
\|\bullet\|_{\bdiv_t ; \Omega}~:=~\|\bullet\|_{0,\Omega} +
\|\bdiv(\bullet)\|_{0,t; \Omega}, \quad\quad 
\forall\,\bullet\in \mathbb{H}(\bdiv_t ; \Omega),
\]
respectively. 
					
Then, proceeding as in \cite[eq. (1.43), Section 1.3.4]{g-Springer-2014}, it is easy to show that for each $ s\geq \frac{2n}{n+2} $ there holds 
\begin{subequations}
\begin{align}
\langle \bta\cdot\bn , v\rangle \, &=\, \int_{\Omega}\Big\{\bta\cdot\nabla v\, +\, v \,\div(\bta) \Big\}\qquad \forall\, (\bta , v)\in \mathbf{H}(\div_s;\Omega)\times \mathrm{H}^{1}(\Omega)\,,\label{eq:int.part_1}\\[1mm]
\langle \bta\bn , \bv\rangle \, &=\, \int_{\Omega}\Big\{\bta\cdot\boldsymbol{\nabla} \bv\, +\, \bv\cdot\bdiv(\bta) \Big\}\qquad \forall\, (\bta , \bv)\in \mathbb{H}(\bdiv_s;\Omega)\times \mathbf{H}^{1}(\Omega)\,,\label{eq:int.part_2}
\end{align}
\end{subequations}
where $\langle \cdot,\cdot\rangle$ stands for both the duality pairings
$(\mathrm H^{-1/2}(\Gamma),\mathrm H^{1/2}(\Gamma)\big)$ and $(\mathbf H^{-1/2}(\Gamma),\mathbf H^{1/2}(\Gamma)\big)$.
\section{The continuous formulation}\label{sec.2}
In this section we introduce the model problem and derive its corresponding weak formulation.
\subsection{The model problem}						
In this paper, we consider and study the stationary Boussinesq
problem with varying viscosity and
thermal conductivity. This problem involves a system of equations where the incompressible Navier--Stokes equation is coupled with the transport equation through a convective term and a buoyancy term, usually exerting force in the opposite direction to gravity. More precisely, given an external force per unit mass $ \bg\in \textbf{L}^{\infty}(\Omega) $, we focus on finding a velocity field $ \bu:\Omega\rightarrow\mathbf{R} $, a pressure field $ p:\Omega\rightarrow\mathrm{R}  $ and the temperature field $ \varphi:\Omega\rightarrow\mathrm{R}  $, such that
\begin{subequations}
\begin{align}
-\bdiv\big(\mu(\varphi)\,\mathbf{e}(\bu)\big) + (\boldsymbol{\nabla}\bu)\bu+\nabla p - \varphi\, \bg&\,=\,\mathbf{0}\quad\qin\Omega\,,\label{eq:prob1}\\[1mm]
\div(\bu)&\,=\,0\quad\qin\Omega\,,\label{eq:prob2}\\[1mm]
-\div(\kappa(\varphi)\nabla\varphi) +\bu\cdot\nabla\varphi &\,=\,0\quad \qin\Omega\,,\label{eq:prob3}\\[1mm]
\bu&\,=\,\mathbf{0}\quad \qon\Gamma\,,\label{eq:prob4}\\[1mm]
\varphi&\,=\,\varphi_{D}\quad \qon\Gamma_{D}\,,\label{eq:prob5}\\[1mm]
\kappa(\varphi)\nabla\varphi\cdot \mathbf{n}&\,=\,0\quad \qon\Gamma_{N}\,,\label{eq:prob5a}
\end{align}
\end{subequations}
					
where $ \Omega $ is a bounded polygonal domain in $ \mathrm{R}^{2} $, with boundary $ \Gamma $  which split as $ \Gamma = \Gamma_{D}\,\cup\,\Gamma_{N}  $, $ \varphi_{D}\in H^{1/2}(\Gamma_{D}) $ is a prescribed temperature and $ \mathbf{e}(\bu) $ is the symmetric part of the velocity gradient tensor, which defined as $ \mathbf{e}(\bu)=(\boldsymbol{\nabla}(\bu)+\boldsymbol{\nabla}(\bu)^{\mathtt{t}})/2 $.
In addition, $ \mu , \kappa :\mathrm{R}\rightarrow\mathrm{R}^{+} $ are the temperature-dependent viscosity and thermal conductivity functions, respectively, which are assumed to be bounded above and below by positive constants, that is, there exist the positive constants $ \mu_1 $, $ \mu_2 $ and $ \kappa_1 $, $ \kappa_2 $ 	such that
\begin{equation}\label{eq:bound.m.k}
	  \mu_1 \,\leq\, \mu(w)\,\leq\, \mu_2 \qan \kappa_1 \,\leq\, \kappa(w)\,\leq\, \kappa_2
	 \qquad\forall\, w \in \mathrm{R}\,.	
\end{equation}
In what follows, we assume that these functions are Lipschitz continuous, that is, there exist the positive constants $ \mathcal{L}_{\mu} $	and $ \mathcal{L}_{\kappa} $ such that
\begin{equation}\label{eq:lip.mu.kap}
\big|\mu(w)-\mu(v)\big|\,\leq\, \mathcal{L}_{\mu}\,\big| w-v\big|\qan
\big|\kappa(w)-\kappa(v)\big|\,\leq\, \mathcal{L}_{\kappa}\,\big|w-v\big|\qquad\forall\, w \,, v\in \mathrm{R}\,.	
\end{equation}
					
Finally, considering \eqref{eq:prob1} and to ensure the uniqueness of the pressure, we will seek this unknown within the space 
\[
L_{0}^{2}(\Omega)\,:=\, \Big\{ q\in L^{2}(\Omega):\quad \int_{\Omega} q \,=\, 0  \Big\}\,.
\]

As our focus lies in utilizing a fully-mixed variational formulation for the coupled model \eqref{eq:prob1}-\eqref{eq:prob5a}, we initially take the approach presented in \cite{ago-Calcolo-2018} for the fluid.
  More precisely, we now introduce as further unknown the pseudostress of fluid equation, by 
\begin{equation}\label{eq:defNewV2}
\bsi\,:=\,\mu(\varphi)(\boldsymbol{\nabla}\bu-\bgam)-\bu\otimes\bu-p \,\mathbb{I}
\qin \Omega	\quad \text{where}\quad \bgam\, =\,\dfrac{1}{2}\left(\boldsymbol{\nabla}\bu-(\boldsymbol{\nabla}\bu)^{\mathtt{t}}\right)\,.
\end{equation}

In this way, applying the matrix trace to the above equation, and utilizing \eqref{eq:prob2} as $$\tr\big(\mu(\varphi)\boldsymbol{\nabla}  \bu\big)\,=\,0\,,$$
one arrives at
\begin{equation}\label{eq:postpr}
	p \,=\,-\dfrac{1}{2}\,\Big\{ \tr(\bsi)+\tr(\bu\otimes\bu)\Big\}
	  \qin \Omega \,.
\end{equation}
					
Subsequently, the momentum and continuity equations \eqref{eq:prob1} and \eqref{eq:prob2}, along with boundary condition \eqref{eq:prob4}, can be reformulated as follows:
\begin{equation}\label{sys1}
	\begin{array}{rcl}
	\frac{1}{\mu(\varphi)}\, \bsi^{\mathtt{d}}+\frac{1}{\mu(\varphi)}(\bu\otimes\bu)^{\mathtt{d}}&=&\boldsymbol{\nabla}\bu-\bgam\qin \Omega\,,\\[1.8ex]
	\bdiv (\bsi)+\varphi\, \mathbf{g}&=&0\qin \Omega\,,\\[1.8ex]
	\mathbf{u}\,=\,\mathbf{0}\qon \Gamma\quad &\text{and}& \quad 
	\displaystyle \int_\Omega \tr\big(\bsi+\bu\otimes\bu\big) \,=\, 0\,.
   \end{array}
\end{equation}
Notice that the first row of \eqref{sys1} suggests us the vorticity tensor $ \bgam $ must be sought in space with the following condition
\begin{equation}\label{eq:cond1}
\bgam\,=\,-\bgam^{\mathtt{t}}\qin\Omega\,.
\end{equation}
					
On the other hand, in order to construct a mixed formulation for the energy equation, we  define unknow the pseudoheat as
\begin{equation*}
\brho\,:=\,\kappa(\varphi)\nabla \varphi-\bu\, \varphi  \qin\Omega\,.	
\end{equation*}
Therefore, thanks to the no-slip condition of velocity, the energy equation \eqref{eq:prob3}, and the boundary condition \eqref{eq:prob5}-\eqref{eq:prob5a}, can be rewritten in the following terms
\begin{equation}\label{sys2}
	\begin{array}{rcl}
	\frac{1}{\kappa(\varphi)}\,\brho + 	\frac{1}{\kappa(\varphi)}\,\bu\,\varphi &=& \nabla \varphi \qin \Omega\,,\\[1.5ex]
		\div(\brho)&=& 0 \qin \Omega\,,\\[1.5ex]
	\varphi|_{\Gamma_D }\,=\,\varphi_{D}\quad &\text{and}& \quad 
	\brho\cdot\mathbf{n}|_{\Gamma_N } \,=\, 0\,.
	\end{array}
\end{equation}
					
As a result, Boussinesq problem \eqref{eq:prob1}-\eqref{eq:prob5a} becomes equations \eqref{sys1} and \eqref{sys2}.
\subsection{The variational formulation}
In this section we use a Banach framework for the continuous weak formulation of \eqref{sys1} and \eqref{sys2}.
Initially, to derive a five-field mixed formulation for \eqref{sys1}, we directly apply \eqref{eq:int.part_2} with $ r\geq \frac{2n}{n+2} $ and $ \bta \in \mathbb{H}(\bdiv_r;\Omega)$ by testing the first equation of \eqref{sys1} to obtain:
\begin{equation}\label{EE0}
		\int_{\Omega} \dfrac{1}{\mu(\varphi)}\, \bsi^{\mathtt{d}}:\bta^{\mathtt{d}}+\int_{\Omega} \dfrac{1}{\mu(\varphi)}\,(\bu\otimes\bu)^{\mathtt{d}}:\bta+\int_{\Omega}\bgam:\bta+\int_{\Omega}\bu\cdot\bdiv(\bta)\,=\,\textbf{0}\qquad \forall~ \bta\in \mathbb{H}(\bdiv_r;\Omega)\,.
\end{equation}
It is easy to see that the first term of \eqref{EE0} makes sense for $ \bsi, \bta\in\mathbb{L}^{2}(\Omega) $ due to Cauchy--Schwarz's inequality and the boundedness of $ \mu $ (cf. first column of \eqref{eq:bound.m.k}), whereas thanks to \eqref{eq:cond1} and Cauchy--Schwarz’s inequality again, the third term of \eqref{EE0} makes sense for $ \bgam\in \mathbb{L}_{\mathtt{skew}}^{2}(\Omega) $ where 
\[
\mathbb{L}_{\mathtt{skew}}^{2}(\Omega)\,:=\,\Big\{\bta\in \mathbb{L}^{2}(\Omega):\quad \bta+\bta^{\mathtt{t}}=\mathbf{0}\Big\}\,.
\]

In turn, given the knowledge of $ \bdiv(\bta)\in \mathbf{L}^{r}(\Omega) $ and employing H\"older's inequality, we infer from the fourth term of \eqref{EE0} that we seek $ \bu\in \mathbf{L}^{r^{'}}(\Omega) $, where $ r^{'} $ is the conjugate of $ r $. 
Alternatively, when using the vector functions in $ \mathbf{L}^{r^{'}}(\Omega) $ as a test function for the second row of \eqref{sys1}, we arrive at	
\begin{equation}\label{EE1}
		\int_{\Omega}\bdiv (\bsi)\cdot \bv +\int_{\Omega}\varphi\,\bg\cdot\bv\,=\,0 \qquad \forall\, \bv\in \mathbf{L}^{r^{'}}(\Omega)\,,
\end{equation}		
Observe from the earlier deduction and the defined spaces for $ \bsi $ and $ \bv $ that the first term of \eqref{EE1} is well-defined. Regarding the second term, which relies on the choice of $ \varphi $, we will address it subsequently.
Lastly, according to \eqref{eq:defNewV2}, the symmetry of $ \bsi $ is weakly imposed as 
\begin{equation*}
	\disp\int_{\Omega}\bsi:\bom \,=\,0\qquad\forall\,\bom\in\mathbb{L}_{\mathtt{skew}}^{2}(\Omega)\,.
\end{equation*}
								
Similarly, letting now $\mathbf{X}$ and $\mathrm{Y}$ be corresponding test spaces for the energy equation \eqref{sys2}, and 
multiplying the first and second equations of \eqref{sys2} by $\bet \in \mathbf{X}$ and
$\psi \in \mathrm{Y}$ respectively, we obtain
\begin{equation}\label{EE0a}
	\int_{\Omega}	\frac{1}{\kappa(\varphi)}\,\brho\cdot\bet + \int_{\Omega}	\frac{1}{\kappa(\varphi)}\,\bu\,\varphi\cdot\bet \,=\,\int_{\Omega}\nabla \varphi\cdot\bet
						\qquad \forall\, \bet\in \mathbf{X}\,,
\end{equation}
and
\begin{equation}\label{EE2a}
\int_{\Omega} \psi\, \div(\brho) \,=\,
		0 \qquad\forall\, \psi\in \mathrm{Y}\,.
\end{equation}
					
Now, concerning the particular selection of the mentioned spaces, it is noteworthy that due to the boundedness of $ \kappa $ (cf. second columns of \eqref{eq:bound.m.k}), the first term of \eqref{EE0a} is well defined when the space $ \mathbf{L}^{2}(\Omega)$ is chosen for $ \brho, \, \bet $. Assuming the initially that $ \varphi \in \mathrm{H}^{1}(\Omega) $ and $ s,s'\in (1,\infty) $ be conjugate to each other, and utilizing \eqref{eq:int.part_1} with $\bet \in \mathbf{H}_{N} (\div_s;\Omega)$, we apply the boundary condition (cf. the third row of \eqref{sys2}), to get
 \begin{equation}\label{EE3-otra}
\int_{\Omega}\nabla \varphi\cdot \bet \,=\,-\int_{\Omega}\varphi\, \div(\bet)+
\langle \bet \cdot\bn,\varphi_D\rangle_{\Gamma_D} \quad\forall\, \bet\in \mathbf{H}_{N} (\div_s;\Omega) \,,
\end{equation}
where $ \mathbf{H}_{N}(\div_s;\Omega) $ is defined by
\[
\mathbf{H}_{N}(\div_s;\Omega)\, :=\, \Big\{  \bet \in \mathbf{H}(\div_s;\Omega):\quad \bet\cdot\bn = 0 \qon\Gamma_{N} \Big \}\,.
\]
Furthermore, we remark that to study the continuity property of functions $ \mu $ and $ \kappa $ within the
definition of the second bilinear forms existing in \eqref{EE0} and \eqref{EE0a}, which will be required for the solvability analysis of the fixed-point operator equation to be proposed afterwards, we need to be able to control the expressions
\begin{equation*}
\int_{\Omega}\left(\dfrac{1}{\mu(\varpi)}-\dfrac{1}{\mu(\phi)}\right)(\bw\otimes\bu)^{\mathtt{d}}:\bta \qan 
\int_{\Omega}\left(\dfrac{1}{\kappa(\varpi)}-\dfrac{1}{\kappa(\phi)}\right)\bu\, \varphi\cdot\bet\,,
\end{equation*}
where $ \varpi $ and $ \phi $ are arbitrary scalars that belong to the identical space where we will search for the unknown $ \varphi $.
Now, utilizing the boundedness and Lipschitz-continuity properties of $ \mu $ and $ \kappa $ (cf. \eqref{eq:bound.m.k} and \eqref{eq:lip.mu.kap}), straightforward applications of the Cauchy--Schwarz and H\"older inequalities yield
\begin{equation}\label{eq:hh}
\begin{split}
\Big|\int_{\Omega}\left(\dfrac{1}{\mu(\varpi)}-\dfrac{1}{\mu(\phi)}\right)(\bw\otimes\bu)^{\mathtt{d}}:\bta\Big|\, \leq \, \dfrac{\mathcal{L}_{\mu}}{\mu_1^{2}}\,\Vert \varpi-\phi\Vert_{0,6;\Omega}\, \Vert\bw\Vert_{0,6;\Omega}\, \Vert\bu\Vert_{0,6;\Omega}\,  \Vert\bta\Vert_{0,\Omega}\,,\\[2ex]
\Big|	\int_{\Omega}\left(\dfrac{1}{\kappa(\varpi)}-\dfrac{1}{\kappa(\phi)}\right)\bu\, \varphi\cdot\bet\Big|\, \leq \, \dfrac{\mathcal{L}_{\kappa}}{\kappa_1^{2}}\,\Vert \varpi-\phi\Vert_{0,6;\Omega}\, \Vert\bu\Vert_{0,6;\Omega}\, \Vert\varphi\Vert_{0,6;\Omega}\,  \Vert\bet\Vert_{0,\Omega}\,,
\end{split}
\end{equation}
					
In this way, the above leads us to look for $ \bu $ and $ \varphi $ in the spaces $ \mathbf{L}^{6}(\Omega) $ and $ \mathrm{L}^{6}(\Omega) $, respectively, and well-definedness of the second terms of \eqref{EE0} and \eqref{EE0a} by setting $ r^{'}:=6 $, $ s^{'}:=6 $ and therefore $ r=6/5 $, $ s=6/5 $.
According to the foregoing discussion, we now define 			
\begin{equation*}
\begin{array}{c}
\mathbb{X}\,:=\,\mathbb{H}_0(\bdiv_{6/5};\Omega)\,, \quad \mathbf{Y}\,:=\,\mathbf{L}^{6}(\Omega)\,,\quad \mathbb{Z}\,:=\, \mathbb{L}_{\mathtt{skew}}^{2}(\Omega)\,,\qan\\[2ex]
\mathbf{X}\,:=\,\mathbf{H}_N(\bdiv_{6/5};\Omega)\,, \quad \mathrm{Y}\,:=\,\mathrm{L}^{6}(\Omega)\,.
\end{array}
\end{equation*}
Next, to rewrite the aforementioned formulation in a more suitable manner for the forthcoming analysis, we proceed by defining product space 
\[
\mathcal{Z} \, :=\, \mathbf{Y} \times \mathbb{Z}\,, 
\]			
and notations
\[
\vec{\bu}\,:=\,(\bu,\bgam)\,,\,	\vec{\bv}\,:=\,(\bv,\bom)\,,\, \vec{\bw}\,:=\,(\bw,\bdel)\in \mathcal{Z}\,,
\]
with equipping $ \mathcal{Z} $ with corresponding norm given by
\[
\Vert \vec{\bw}\Vert_{\mathcal{Z}}\, =\, \Vert(\bw,\bdel)\Vert\,:=\, \Vert \bw\Vert_{0,6;\Omega}+\Vert \bdel\Vert_{0,\Omega}\,.
\]

Consequently, replacing \eqref{EE3-otra} back into 
\eqref{EE0a}, and gathering the resulting equation with \eqref{EE2a}, \eqref{EE0} and \eqref{EE1}, we deduce 
that the variational formulation of \eqref{sys1} and \eqref{sys2} becomes: 
Find $ (\bsi, \vec{\bu})\in\mathbb{X}\times\mathcal{Z} $ and $ (\brho, \varphi)\in \mathbf{X}\times \mathrm{Y}$  such that 
\begin{subequations}
\begin{align}
\mathscr{A}_{\varphi}^{S}(\bsi, \bta)+\mathscr{O}_{\varphi}^{S}(\bu;\bu,\bta)+\mathscr{B}^{S}(\bta,\vec{\bu})&\,=\,0  
\qquad\qquad\quad\,\,\,\, \forall\,\bta\in \mathbb{X} \,,\label{eq:v1a-eq:v4-a}\\[1ex]
\mathscr{B}^{S}(\bsi,\vec{\bv})&\,=\, -\disp\int_{\Omega}\varphi\,\bg\cdot\bv
 \quad\,\, \forall\, \bv\in \mathbf{Y} \,,\label{eq:v1a-eq:v4-b}
\\[1ex]
\mathscr{A}^{T}_{\varphi}(\boldsymbol{\rho},\boldsymbol{\eta})+\mathscr{O}^{T}_{\varphi}(\bu;\varphi,\boldsymbol{\eta})+\mathscr{B}^{T}(\varphi,\boldsymbol{\eta})&\,=\ \mathscr{F}^{T}(\bet)  \qquad\quad\,\, \forall\, \bet\in\mathbf{X} \,,\label{eq:v1a-eq:v4-c}
\\[1ex]
\mathscr{B}^{T}(\psi,\brho)&\,=\ 0  \qquad\qquad\qquad\,\forall\, \psi\in \mathrm{Y} \,,\label{eq:v1a-eq:v4-d}
\end{align}
\end{subequations}
					
where the bilinear forms 
$ \mathscr{B}^S: \mathbb{X}\times \mathcal{Z}\rightarrow \mathrm{R} $  and
$\mathscr{B}^{T}: \mathbf{X}\times \mathrm{Y}\rightarrow \mathrm{R} $ are defined as
\begin{equation}\label{def-a-b-tilde-a-tilde-b}
\begin{array}{rcll}
\mathscr{B}^{S}(\bta ,\vec{\bv})&:=& \disp  \int_{\Omega}\bv\cdot\bdiv(\bta)+\int_{\Omega}\bom:\bta & \quad \forall\,(\bta,\vec{\bv}) \in \mathbb{X} \times \mathcal{Z}\,, 	 \\[2ex]
\mathscr{B}^{T}(\bet ,\psi)&:=& \disp  \int_{\Omega}\psi\,\div(\bet) & \quad \forall\,(\bet,\psi) \in \mathrm{Y} \times \mathbf{X}\,,
\end{array}
\end{equation}
whereas the linear functional  $\mathscr{F}^{T} : \mathbf X \rightarrow \mathrm R$
are given by
\begin{equation}\label{def-f-tilde-f}
\begin{array}{rcll}
	\mathscr{F}^{T}(\bet) &:=& \langle \bet \cdot\bn,\varphi_D\rangle_{\Gamma_D} & \quad\forall\,\bet \in \mathbf X\,.
\end{array}
\end{equation}
					
In turn, for each $ \mathbf{z}\in\mathbf{Y} $ and $\phi\in \mathrm{Y} $ we set the bilinear forms $\mathscr{A}_{\phi}^{S}: \mathbb{X} \times \mathbb{X} \rightarrow \mathrm R$, $\mathscr{O}_{\phi}^{S}(\bz;\cdot,\cdot) : \mathbf{Y}\times\mathbb{X}
\rightarrow \R $, and $\mathscr{A}^{T}_{\phi}: \mathbf{X} \times \mathbf{X} \rightarrow \mathrm R$, $\mathscr{O}^{T}_{\phi}(\bz;\cdot,\cdot) : \mathrm{Y}\times\mathbf{X}
\rightarrow \R $ as
					
\begin{equation}\label{def-c-tilde-c}
\begin{array}{rcll}
\mathscr{A}_{\phi}^{S}(\bze,\bta)&:=& \disp \int_{\Omega}\dfrac{1}{\mu(\phi)}\,\bze^{\mathtt{d}}:\bta^{\mathtt{d}}&\quad \forall\,\bze,\bta\in \mathbb{X}\,,\\[2ex]
	\mathscr{O}_{\phi}^{S}(\bz;\bw,\bta) &:=& \disp \int_{\Omega}\dfrac{1}{\mu(\phi)}\,(\bz\otimes\bw)^{\mathtt{d}}: \bta&
		\quad \forall\, (\bw,\bta) \in \mathbf Y \times \mathbb X\,, \qan \\[2ex]
		\mathscr{A}^{T}_{\phi}(\bxi,\bet)&:=& \disp \int_{\Omega}\dfrac{1}{\kappa(\phi)}\,\bxi\cdot\bet&\quad \forall\,\bxi,\bet\in \mathbf{X}\,,\\[2ex]
	 \mathscr{O}^{T}_{\phi}(\bz;\omega,\bet) &:=& \disp \int_{\Omega}\dfrac{1}{\kappa(\phi)}\,\mathbf{z}\,\omega\cdot \boldsymbol{\eta} &
	\quad \forall\, (\omega,\bet) \in \mathrm Y \times \mathbf X\,.
\end{array}
\end{equation}
\section{The continuous solvability analysis}\label{sec.3}
In this section, we investigate the solvability of \eqref{eq:v1a-eq:v4-a}-\eqref{eq:v1a-eq:v4-d} by utilizing abstract findings pertaining to the stability characteristics of the associated bilinear forms, as compiled below. Then, we proceed by introducing resolvent operators corresponding to each of the individual problems that constitute \eqref{eq:v1a-eq:v4-a}-\eqref{eq:v1a-eq:v4-d}, subsequently transforming the latter into an equivalent fixed-point equation. Ultimately, we demonstrate the well-defined nature of the aforementioned operators and use the classical Banach theorem to prove the solvability of \eqref{eq:v1a-eq:v4-a}-\eqref{eq:v1a-eq:v4-d}.
					
\subsection{Preliminaries}\label{sub.sec.3.3.1}
We begin by stating a slight adaptation of the abstract result established in \cite[Theorem 3.4]{cg-CMA-2022}.
					
\begin{theorem}\label{t1:wel}
Let $ X $ and $ Y $ be reflexive Banach spaces, and let $ a: X\times X \rightarrow\mathrm{R} $ and $ b: X\times Q \rightarrow\mathrm{R} $ be given bounded bilinear forms. Moreover, let $ \mathbf{B}: X\rightarrow Y' $ be the bounded linear operator induced by $ b $, and let $ \mathcal{H}:=N(\mathbf{B}) $ be the respective null space. Assume that:
\begin{itemize}[leftmargin=.6in]
\item[i)] $ a $ is positive semi-definite, that is
\begin{equation*}
a(v,v)\, \geq\, 0\quad\forall\, v\in X \,,
\end{equation*}
\item[ii)] there exists a constant $ \alpha $ such that
\begin{equation*}
\sup_{0\neq v\in \mathcal{H}}\dfrac{a(w,v)}{\Vert v\Vert_{X}}\,\geq\,\alpha\,\Vert w\Vert_{X}\qquad\forall\, w\in \mathcal{H} \,,
\end{equation*}
and
\begin{equation*}
\sup_{0\neq w\in \mathcal{H}}\dfrac{a(w,v)}{\Vert w\Vert_{X}}\,\geq\,\alpha\,\Vert v\Vert_{X}\qquad\forall\, v\in \mathcal{H} \,,
\end{equation*}
\item[iii)] and there exists a constant $ \beta $ such that
\begin{equation*}
\sup_{0\neq v\in X}\dfrac{b(v,q)}{\Vert v\Vert_{X}}\,\geq\,\beta\,\Vert q\Vert_{Y}\qquad\forall\, q\in Y \,,
\end{equation*}
\end{itemize}
Then, for each pair $ (f , g)\in X' \times Y' $ there exists a unique $ (u,p)\in X\times Y $ such that
\begin{equation*}
\begin{array}{rcll}
a(u,v)+b(v,p) &=& f(v) &\quad\forall\, v\in X \\[1ex]
b(u,q)&=& g(q) &\quad\forall\, q\in Y
\end{array}
\end{equation*}
Moreover, there exists a constant $ C_{st} $, depending only on $ \| a\| $, $ \| c\| $, $ \alpha $, and $ \beta $, such that
\begin{equation*}
\Vert (u,p)\Vert_{X\times Y}\, \leq \, C_{st}\,\Big\{\Vert f\Vert_{X'}+\Vert g\Vert_{Y'} \Big\}\,.
\end{equation*}
\end{theorem}
Next, we establish the stability properties of some of the forms involved in \eqref{eq:v1a-eq:v4-a}-\eqref{eq:v1a-eq:v4-d}.
We begin by state the boundedness of all the variational forms involved \eqref{eq:v1a-eq:v4-a}-\eqref{eq:v1a-eq:v4-d}. Direct applications of the Cauchy--Schwarz and H\"older inequalities, the Sobolev embedding $  \mathrm{H}^{1}(\Omega)\subset\mathrm{L}^{6}(\Omega) $
and the continuity of the normal trace operator in $ \mathbf{H}(\div_{6/5};\Omega) $ implies that there exist constants, denoted and given by
\[
\|\mathscr{A}^{S}\| = 1/\mu_1,\quad \|\mathscr{B}^S\| =1, \quad \| \mathscr{O}^{S} \| = |\Omega|^{1/6}/\mu_1\,,
\]
\[
\|\mathscr{A}^{T}\| = 1/\kappa_1,\quad \|\mathscr{B}^{T}\|  = 1, \quad \| \mathscr{O}^{T}_\phi \| = |\Omega|^{1/6}/\kappa_1,,\quad \| \mathscr{F}^{T}\| = \| i_6\|\,, 
\]
					
such that for each $ \bz\in \mathbf{Y} $ and $ \phi\in \mathrm{Y} $ there hold 
\begin{equation}\label{eq:bound}
\begin{array}{rcll}
\big|\mathscr{A}_{\phi}^{S}(\bze , \bta)\big|&\leq &  \|\mathscr{A}^S\|\, \Vert \bze\Vert_{\div_{6/5},\Omega}\, \Vert \bta\Vert_{\div_{6/5},\Omega} &\quad\forall \, \bze , \bta\in \mathbb{X}\,, \\[1ex]
\big|\mathscr{B}^S(\bta,\vec{\bv})\big|&\leq &  \|\mathscr{B}^S\|\, \Vert \bta\Vert_{\div_{6/5},\Omega}\, \Vert
\vec{\bv}\Vert &\quad\forall\, (\bta,\vec{\bv})\in \mathbb{X}\times\mathcal{Z}\,,\\[1ex]
\big| \mathscr{O}_{\phi}^{S}(\bz; \bw, \bta)\big| &\leq &\|\mathscr{O}^{S}\|\,\Vert \bz\Vert_{0,6;\Omega}\, \Vert\bw\Vert_{0,6;\Omega}\, \Vert\bta\Vert_{0,\Omega}&\quad\forall\, (\bw, \bta)\in \mathbf{Y}\times\mathbb{X}\,, \\[1ex]
\big|\mathscr{A}^{T}_\phi(\bxi , \bet)\big|&\leq & \|\mathscr{A}^{T}\|\, \Vert \bxi\Vert_{\div_{6/5},\Omega}\,\Vert \bet\Vert_{\div_{6/5},\Omega} &\quad\forall\,\bxi , \bet \in\mathbf{X}\,,\\[1ex]
\big|\mathscr{B}^{T} (\bet ,\psi)\big| &\leq & \|\mathscr{B}^{T}\|\, \Vert \bet\Vert_{\div_{6/5},\Omega}\, \Vert\psi\Vert_{0,6;\Omega}&\quad\forall\, (\bet ,\psi)\in\mathbf{X}\times\mathrm{Y}\,, \\[1ex]
\big|\mathscr{O}^{T}_\phi(\bz; \omega, \bet)\big| &\leq & \| \mathscr{O}^{T}\| \, \Vert \bz\Vert_{0,6;\Omega}\, \Vert \omega\Vert_{0,6;\Omega}\, \Vert\bet\Vert_{0,\Omega} &\quad\forall\, (\omega, \bet)\in \mathrm{Y}\times\mathbf{X}\,.
\end{array}
\end{equation}
Moreover, given arbitrary $ \bz,\by\in \mathbf{Y} $ and $ \phi,\varpi\in \mathrm{Y} $ noting from the definitions of $ \mathscr{O}_{\phi}^{S} $ (cf. first row of \eqref{def-c-tilde-c}) and $ \mathscr{O}^{T}_\phi $ (cf. second row of \eqref{def-c-tilde-c}), we readily deduce from \eqref{eq:hh} that there exist constants
\[
\ell_{\mathtt{lip},\mathscr{O}^S}\,:=\, \dfrac{\mathcal{L}_\mu}{\mu_1 ^{2}}\,, \quad \ell_{\mathtt{lip},\mathscr{O}^{T}}\,:=\, \dfrac{\mathcal{L}_\kappa}{\kappa_1 ^{2}}\,,
\]
such that
\begin{equation}\label{eq:bound.c.ct}
\begin{split}
	\big| \mathscr{O}_{\phi}^{S}(\bz; \bw, \bta)-\mathscr{O}_\varpi^S(\bz; \bw, \bta)\big| &\,\leq\, \ell_{\mathtt{lip},\mathscr{O}^S}\,\Vert \phi-\varpi\Vert_{0,6;\Omega}\,\Vert \bz\Vert_{0,6;\Omega}\, \Vert\bw\Vert_{0,6;\Omega}\, \Vert\bta\Vert_{0,\Omega}\,,\\[1ex]	
	\big|\mathscr{O}^{T}_\phi(\bz; \omega, \bet)-\mathscr{O}^{T}_\varpi(\bz; \omega, \bet)\big|& \, \leq \, \ell_{\mathtt{lip},\mathscr{O}^{T}}\,\Vert \phi-\varpi\Vert_{0,6;\Omega}\,\Vert \bz\Vert_{0,6;\Omega}\, \Vert \omega\Vert_{0,6;\Omega}\, \Vert\bet\Vert_{0,\Omega}\,,\\[1ex]
		\big|\mathscr{O}_{\phi}^{S}(\bz;\bw,\bta)-\mathscr{O}_{\phi}^{S}(\by;\bw,\bta)\big|&\,\leq\,\| \mathscr{O}^{S}\| \, \Vert\bz-\by\Vert_{0,6;\Omega}\, \Vert\bw\Vert_{0,6;\Omega}\, \Vert \bta\Vert_{\bdiv_{6/5},\Omega}\,,\\[1ex]
		\big|\mathscr{O}^{T}_\phi(\bz;\omega,\bet)-\mathscr{O}^{T}_\phi(\by;\omega,\bet)\big|&\,\leq\,\| \mathscr{O}^{T}\| \, \Vert\bz-\by\Vert_{0,6;\Omega}\, \Vert\omega\Vert_{0,6;\Omega}\, \Vert \bet\Vert_{\div_{6/5},\Omega}\,,
	\end{split}
\end{equation}
for all $ (\bw, \bta)\in \mathbf{Y}\times\mathbb{X} $ and $ (\omega,\bet)\in \mathrm{Y}\times\mathbf{X} $.
\subsection{A fixed point strategy}\label{sub.sec.3.3.2}
We begin by rewriting the first part of variational formulation \eqref{eq:v1a-eq:v4-a}-\eqref{eq:v1a-eq:v4-d} as an equivalent fixed point equation. To this end, we first let $ \mathscr{L}^S:  \mathbf{Y} \times \mathrm{Y}\rightarrow \mathbb{X}\times \mathcal{Z}$ be the operator defined by
\begin{equation}\label{eq:opr1}
\mathscr{L}^S(\bz, \phi)\, =\,\big(\mathscr{L}^S_1(\bz, \phi), \mathscr{L}^S_2(\bz, \phi),\mathscr{L}^S_3(\bz, \phi)\big)\, :=\,
(\bsi_{\star}, \vec{\bu}_{\star})\,, 
\end{equation}
	for each $ (\bz, \phi)\in \mathbb{X} \times \mathbf{Y} $, where $ (\bsi_{\star}, \vec{\bu}_{\star})=(\bsi_{\star}, (\bu_{\star},\bgam_{\star}))\in \mathbb{X}\times \mathcal{Z} $ is the unique solution of \eqref{eq:v1a-eq:v4-a}-\eqref{eq:v1a-eq:v4-b} when $ \mathscr{A}_{\varphi}^{S} $ and $ \mathscr{O}_{\varphi}^{S}(\bu;\cdot,\cdot) $ are replaced by $ \mathscr{A}_{\phi}^{S} $ and $ \mathscr{O}_{\phi}^{S}(\bz;\cdot,\cdot) $, respectively, that is
\begin{equation}\label{eq:subp1}
\begin{array}{rcll}
\mathscr{A}_{\phi}^{S}(\bsi_\star, \bta)+\mathscr{O}_{\phi}^{S}(\bz;\bu_\star,\bta)+\mathscr{B}^{S}(\bta ,\vec{\bu}_\star )&=&0 & 
\quad \forall\,\bta\in \mathbb{X} \,,
\\[1.5ex]
\mathscr{B}^{S}(\bsi_\star ,\vec{\bv})&=& \disp\int_{\Omega}\phi\,\bg\cdot\bv
& \quad \forall\, \vec{\bv}\in \mathcal{Z} \,,
\end{array}
\end{equation}
equivalently, defining $ \mathcal{A}_{\phi}^S: (\mathbb{X}\times\mathcal{Z}) \times (\mathbb{X}\times\mathcal{Z})\rightarrow\mathrm{R}$, for each $ \phi\in \mathrm{Y} $, arising from \eqref{eq:subp1} after adding the terms involving $ \mathscr{A}_{\phi}^{S} $ and $ \mathscr{B}^S $, that is,
	\begin{equation}\label{def.A}
	\mathcal{A}_{\phi}^S\big((\bze,\vec{\bw}), (\bta,\vec{\bv})\big)\,:=\, \mathscr{A}_{\phi}^{S}(\bze,\bta)+\mathscr{B}^S (\bta,\vec{\bw})+\mathscr{B}^{S}(\bze ,\vec{\bv})\,,
	\end{equation}
	and functional $ \mathcal{F}_{\phi}^S\in (\mathbb{X}\times\mathcal{Z})^{'} $ as
	\begin{equation}\label{eq:def.F}
\mathcal{F}_{\phi}^S(\bta,\,\vec{\bv})\,=\, \disp\int_{\Omega}\phi\,\bg\cdot\bv\qquad\forall\,(\bta,\vec{\bv})\in \mathbb{X}\times\mathcal{Z}\,,
	\end{equation}
	we can restate \eqref{eq:subp1} as
\begin{equation}\label{eq:sub_1}
	\begin{array}{rcll}
	\mathcal{A}_{\phi}^S \big((\bsi_{\star}\,,\vec{\bu}_{\star}), (\bta,\, \vec{\bv})\big)\,+\,\mathscr{O}_{\phi}^{S}(\bz;\bu_\star,\bta)\,&=&\, \mathcal{F}_{\phi}^S(\bta,\,\vec{\bv})&\quad \forall\,(\bta,\, \vec{\bv})\in \mathbb{X}\times\mathcal{Z}\,.
	\end{array}	
\end{equation}
				
	In turn, we let $ \mathscr{L}^T: \mathbf{Y}\times\mathrm{Y}\rightarrow  \mathbf{X}\times\mathrm{Y} $ be the operator given by
\begin{equation}\label{eq:opr2}
	\mathscr{L}^T(\bz , \phi)\, =\, \big(\mathscr{L}^T_1(\bz , \phi), \mathscr{L}^T_2(\bz , \phi)\big)\,:=\,
(\brho_{\star}, \varphi_{\star})\,, 
\end{equation}
for  each $ (\bz, \phi)\in\mathbf{Y}\times\mathrm{Y} $	where $ (\brho_{\star}, \varphi_{\star})\in  \mathbf{X}\times\mathrm{Y} $ is the unique solution of  \eqref{eq:v1a-eq:v4-c}-\eqref{eq:v1a-eq:v4-d} when $\mathscr{A}^{T}_{\varphi}$ and $ \mathscr{O}^{T}_\varphi(\bu;\cdot,\cdot) $ are replaced by $\mathscr{A}^{T}_{\phi}$ and $ \mathscr{O}^{T}_\phi(\bz;\cdot,\cdot) $, respectively, that is
\begin{equation}\label{eq:subp2}
\begin{array}{rcll}
\mathscr{A}^{T}_{\phi}(\brho_\star,\bet)+\mathscr{O}^{T}_{\phi}(\bz;\varphi_\star,\bet)+\mathscr{B}^{T}(\bet ,\varphi_\star )&=& \mathscr{F}^{T}(\bet) & \quad \forall\, \bet\in\mathbf{X} \,,
\\[1.5ex]
\mathscr{B}^{T}(\brho_\star ,\psi )&=& 0 & \quad\forall\, \psi\in \mathrm{Y} \,,	
\end{array}
\end{equation}
Analogously, defining $ \mathcal{A}_{\phi}^T: (\mathbf{X}\times\mathrm{Y})\times (\mathbf{X}\times\mathrm{Y})\rightarrow \mathrm{R} $, involving bilinear forms $ \mathscr{A}^{T}_\phi $ and $ \mathscr{B}^{T} $, that is,
\begin{equation}\label{def.At}
\mathcal{A}_{\phi}^T\big((\bxi,\omega), (\bet,\psi)\big)\,:=\, \mathscr{A}^{T}_\phi(\bxi,\bet)+\mathscr{B}^{T} (\bet ,\omega)+\mathscr{B}^{T}(\bxi,\psi)\,,
\end{equation}
we realize that problem \eqref{eq:subp2} can be re-stated as:
\begin{equation}\label{eq:sub_2}
\begin{array}{rcll}
	\mathcal{A}_{\phi}^T \big((\brho_{\star}\,,\varphi_{\star}), (\bet,\, \psi)\big)+\mathscr{O}^{T}_\phi(\bz;\varphi_\star,\bet)\,&=&\, \mathscr{F}^{T}(\bet)&\quad \forall\, (\bet,\psi)\in \mathrm{Y} \,.
\end{array}	
\end{equation}
Thus, defining the operator $ \mathscr{L} :\mathbf{Y}\times \mathrm{Y}\rightarrow \mathbf{Y}\times \mathrm{Y} $ as
\begin{equation}\label{eq:def.opr.E}
\mathscr{L}\big(\bz,\phi\big)\,:=\, \big(\mathscr{L}_2^S(\bz,\mathscr{L}_2^T(\bz,\phi)), \mathscr{L}_2^T(\bz,\phi)\big)\quad\forall\,(\bz,\phi)\in \mathbf{Y}\times \mathrm{Y}\,,
\end{equation}

it follows that \eqref{eq:v1a-eq:v4-a}-\eqref{eq:v1a-eq:v4-d} can be rewritten as the fixed-point equation: Find $ (\bu , \varphi)\in  \mathbf{Y}\times \mathrm{Y}$ such that
\begin{equation}\label{eq:fixp}
\mathscr{L}(\bu , \varphi)\, =\,(\bu , \varphi)\,.  
\end{equation}
\subsection{Well-posedness of the uncoupled problems}
In what follows we show that the uncoupled problems \eqref{eq:sub_1} and \eqref{eq:sub_2} are well-posed, which means that $ \mathscr{L}^S $ and $ \mathscr{L}^T $ (cf. \eqref{eq:opr1} and \eqref{eq:opr2}) are indeed well-defined, respectively. We observe that, given $ \phi\in \mathrm{Y} $ the structures of $ \mathcal{A}_{\phi}^S $ and $  \mathcal{A}_{\phi}^T $ appeared in 
the
pair problems \eqref{eq:sub_1} and \eqref{eq:sub_2} are same as the one in Theorem \ref{t1:wel}.
To this end, first we will utilize in turn Theorem \ref{t1:wel} to establish the needed assumptions.
We readily observe from the definition of  $ \mathscr{A}_{\phi}^{S} $ (res. $ \mathscr{A}^{T}_\phi $) given by \eqref{def-c-tilde-c}, that the bilinear form $ \mathscr{A}_{\phi}^{S} $ (res. $ \mathscr{A}^{T}_\phi $) is elliptic with the constant $ \alpha_S := 1/\mu_2 $ (res. $ \alpha_T := 1/\kappa_2 $), that is, 
\begin{subequations}
\begin{align}
	\mathscr{A}_{\phi}^{S}(\bta , \bta)\, \geq\, \alpha_S \, \Vert \bta^{\mathtt{d}}\Vert_{0,\Omega}^{2}\quad\forall\,\bta\in \mathbb{X}\label{eq:ell.a}\qan\\[2ex]
\mathscr{A}^{T}_\phi(\bet , \bet)\, \geq\, \alpha_T \, \Vert \bet\Vert_{0,\Omega}^{2}\quad\forall\,\bet\in \mathbf{X}\,. \label{eq:ell.at}
\end{align}
\end{subequations}
Now, let us look at the kernel of kernels of the operators induced by $ \mathscr{B}^{S} |_{\mathbb{X}\times\mathcal{Z}}$ and $ \mathscr{B}^{T}|_{\mathbf{X}\times\mathrm{Y}}$ (cf. $\eqref{def-a-b-tilde-a-tilde-b}$), that is
\begin{equation*}
	\begin{array}{c}
	\mathrm{K}\, :=\, \Big\{\bta\in\mathbb{X}: \quad \mathscr{B}^{S}(\bta,\vec{\bv})\,=\, 0\quad\forall\,\bv\in \mathcal{Z}\Big\}\,,\qan\\[2ex]
	\widetilde{\mathrm{K}}\, :=\, \Big\{\bet\in\mathbf{X}: \quad  \mathscr{B}^{T}(\bet,\psi)=0\quad\forall\,\psi\in \mathrm{Y} \Big\}\,,
	\end{array}
\end{equation*}
which, proceeding similarly to \cite{gors-CMA-2021} reduce to					
\begin{subequations}
	\begin{align}
	\mathrm{K}&\, =\, \Big\{\bta\in\mathbb{X}: \quad\bta \,=\,\bta^{\mathtt{t}}\qan \bdiv(\bta)\,=\, \mathbf{0}\qin \Omega\Big\}\,,\qan\label{eq:def.Ki}\\[2ex]
\widetilde{\mathrm{K}}&\, =\, \Big\{\bet\in\mathbf{X}: \quad  \div (\bet)=0\qin\Omega  \Big\}\,,\label{eq:def.Kit}
	\end{align}
\end{subequations}
				
Thus,  using \eqref{eq:ell.a}, \eqref{eq:def.Ki} and that
there exists a constant $ c_{\Omega} $, depending only on $ \Omega $, such that (cf. \cite[Lemma 2.3]{cgo-NMPDE-2021})
\begin{equation}\label{eq:div.ta}
	c_{\Omega}\, \Vert\bta\Vert_{0,\Omega}^{2}\, \leq\, \Vert\bta^{\mathtt{d}}\Vert_{0,\Omega}^{2}+ \Vert\bdiv(\bta)\Vert_{0,6/5;\Omega}^{2}\qquad\forall\, \bta\in \mathbb{X}\,,
	\end{equation}
 it is seen that $ \mathscr{A}_{\phi}^{S} $ is $ \mathrm{K} $-elliptic with constant $ \alpha_S\, c_{\Omega} $, whereas thanks to \eqref{eq:ell.at} and \eqref{eq:def.Kit} $ \mathscr{A}^{T}_\phi $ is $ \widetilde{\mathrm{K}} $-elliptic with the same constant $ \alpha_T $ given by \eqref{eq:ell.at}, that is,
\begin{subequations}
	\begin{align}
\mathscr{A}_{\phi}^{S}(\bta , \bta)&\, \geq\, \alpha_S\, c_{\Omega} \, \Vert \bta\Vert_{\bdiv_{6/5},\Omega}^{2}\quad\forall\,\bta\in \mathrm{K}\,,\qan\label{eq:K.ell.a}\\[2ex]
	\mathscr{A}^{T}_\phi(\bet , \bet)&\, \geq\, \alpha_T \, \Vert \bet\Vert_{\div_{6/5},\Omega}^{2}\quad\forall\,\bet\in \widetilde{\mathrm{K}}\,.\label{eq:K.ell.at}
\end{align}
\end{subequations}
					
Having established that $ \mathscr{A}_{\phi}^{S} $ (res. $ \mathscr{A}^{T}_\phi $) satisfies hypotheses i) and ii) of Theorem \ref{t1:wel}, the next step is to demonstrate the corresponding assumption iii), namely the continuous inf-sup condition for $ \mathscr{B}^{S} $ (res. $ \mathscr{B}^{T} $). 
For this purpose, we recall from a slight
modification of \cite[Lemma 3.5]{gors-CMA-2021} and \cite[Lemma 3.1]{cov-CALCOLO-2020} that there exist constants $ \beta_S $ and $ \beta_T $ such that there hold
\begin{equation}\label{eq:inf.bi}
\sup_{\mathbf{0}\neq \bta\in \mathbb{X}}\dfrac{\mathscr{B}^{S}(\bta ,\vec{\bv})}{\Vert \bta\Vert_{\bdiv_{6/5},\Omega}}\,\geq\,\beta_S \,\Vert \vec{\bv}\Vert_{0,6;\Omega}\qquad\forall\, \vec{\bv}\in \mathcal{Z} \,,
\end{equation}
and
\begin{equation}\label{eq:inf.bit}
\sup_{\mathbf{0}\neq \bet\in \mathbf{X}}\dfrac{\mathscr{B}^{T}(  \bet ,\psi)}{\Vert \bet\Vert_{\div_{6/5},\Omega}}\,\geq\,\beta_T\,\Vert \psi\Vert_{0,6;\Omega}\qquad\forall\, \psi\in \mathrm{Y} \,,
\end{equation}

Consequently, according to \eqref{eq:K.ell.a}, \eqref{eq:inf.bi} and \eqref{eq:K.ell.at}, \eqref{eq:inf.bit}, the needed assumptions of Theorem \ref{t1:wel} are fulfilled, leading to a direct application of this abstract theorem, which guarantees the existence of positive constants $ \alpha_{\mathcal{A}^S} $ and $ \alpha_{\mathcal{A}^T} $, depending only on $ \alpha_S, c_{\Omega}, \beta_S, \| \mathscr{A}_{\phi}^{S}\| $ and $ \alpha_T, \beta_T, \| \mathscr{A}^{T}_\phi\| $, such that
\begin{equation}\label{eq:inf.A}
	\sup\limits_{\mathbf{0}\neq (\bta,\vec{\bv})\in \mathbb{X}\times\mathbf{Y}}\dfrac{\mathcal{A}_{\phi}^S 
	\big((\bze, \vec{\bw}), (\bta,\vec{\bv})\big)}{\Vert (\bta,\vec{\bv})\Vert}\,\geq\,\alpha_{\mathcal{A}^S}\,\Vert(\bze, \vec{\bw})\Vert\qquad\forall\, (\bze, \vec{\bw})\in \mathbb{X}\times\mathcal{Z} \,,
\end{equation}
and
\begin{equation}\label{eq:inf.At}
	\sup\limits_{\mathbf{0}\neq (\bet, \psi)\in \mathbf{X}\times\mathrm{Y}}\dfrac{\mathcal{A}_{\phi}^T\big((\bxi, \omega), (\bet, \psi)\big)}{\Vert (\bet, \psi)\Vert}\,\geq\,\alpha_{\mathcal{A}^T}\,\Vert(\bxi, \omega)\Vert\qquad\forall\, (\bxi, \omega)\in \mathbf{X}\times\mathrm{Y} \,.
\end{equation}
In turn, it follows from \eqref{eq:inf.A} and \eqref{eq:inf.At}, and the boundedness of $ \mathscr{O}_{\phi}^{S} $ and $ \mathscr{O}^{T}_\phi $ (cf. \eqref{eq:bound.c.ct}) that for each $ (\bze,\vec{\bw})\in \mathbb{X}\times\mathcal{Z} $ and $ (\bxi,\omega)\in \mathbf{X}\times\mathrm{Y} $ there holds
\begin{equation*}
	\sup\limits_{\mathbf{0}\neq (\bta,\vec{\bv})\in \mathbb{X}\times\mathcal{Z}}\dfrac{\mathcal{A}_{\phi}^S 
	\big((\bze, \vec{\bw}), (\bta,\vec{\bv})\big)+\mathscr{O}_{\phi}^{S}(\bz;\bw,\bv)}{\Vert (\bta,\vec{\bv})\Vert}\,\geq\,\left(\alpha_{\mathcal{A}^S}-\dfrac{|\Omega|^{1/6}}{\mu_1}\Vert\bz\Vert_{0,6;\Omega}\right)\,\Vert(\bze, \vec{\bw})\Vert\,,
\end{equation*}
and
\begin{equation*}
	\sup\limits_{\mathbf{0}\neq (\bet, \psi)\in \mathbf{X}\times\mathrm{Y}}\dfrac{\mathcal{A}_{\phi}^T\big((\bxi, \omega), (\bet, \psi)\big)+\mathscr{O}^{T}_\phi(\bz;\omega,\psi)}{\Vert (\bet, \psi)\Vert}\,\geq\,\left(\alpha_{\mathcal{A}^T}-\dfrac{|\Omega|^{1/6}}{\kappa_1}\Vert\bz\Vert_{0,6;\Omega}\right)\,\Vert(\bxi, \omega)\Vert\,,
\end{equation*}
					
and thus, under the assumptions that
\begin{equation}\label{eq:as.z1}
\Vert\bz\Vert_{0,6;\Omega}\, \leq \, r_1\quad \text{with}\quad r_1 \,:=\, \dfrac{\mu_1 \,\alpha_{\mathcal{A}^S}}{2|\Omega|^{1/6}}\,,
\end{equation}
and
\begin{equation}\label{eq:as.z2}
\Vert\bz\Vert_{0,6;\Omega}\, \leq \, r_2 \quad\text{with}\quad r_2\,:=\, \dfrac{\kappa_1 \,\alpha_{\mathcal{A}^T}}{2|\Omega|^{1/6}}\,,
\end{equation}
we arrive at
\begin{equation}\label{eq:inf.A.1}
\sup\limits_{\mathbf{0}\neq (\bta,\vec{\bv})\in \mathbb{X}\times\mathcal{Z}}\dfrac{\mathcal{A}_{\phi}^S 
\big((\bze, \vec{\bw}), (\bta,\vec{\bv})\big)+\mathscr{O}_{\phi}^{S}(\bz;\bw,\bv)}{\Vert (\bta,\vec{\bv})\Vert}\,\geq\,\dfrac{\alpha_{\mathcal{A}^S}}{2}\,\Vert(\bze, \bw)\Vert\quad \forall\, (\bze, \vec{\bw})\in \mathbb{X}\times\mathcal{Z}\,,
\end{equation}
and
\begin{equation}\label{eq:inf.At.1}
\sup\limits_{\mathbf{0}\neq (\bet, \psi)\in \mathbf{X}\times\mathrm{Y}}\dfrac{\mathcal{A}_{\phi}^T\big((\bxi, \omega), (\bet, \psi)\big)+\mathscr{O}^{T}_\phi(\bz;\omega,\psi)}{\Vert (\bet, \psi)\Vert}\,\geq\,\dfrac{\alpha_{\mathcal{A}^T}}{2}\,\Vert(\bxi, \omega)\Vert\quad\forall\, (\bxi, \omega)\in \mathbf{X}\times\mathrm{Y}\,.
\end{equation}
									
As a result of the above analysis, we are in the position of establishing the well-definedness of $ \mathscr{L}^S $ and $ \mathscr{L}^T $ (cf. \eqref{eq:opr1} and \eqref{eq:opr2}), or equivalently, the solvability of uncoupled problems \eqref{eq:subp1} and \eqref{eq:subp2}.
\begin{lemma}\label{l_wel.S.T}
For each $ \bz\in \mathbf{Y} $ satisfying \eqref{eq:as.z1} and $ \phi\in \mathrm{Y} $ there exists a unique $ (\bsi_{\star},\vec{\bu}_{\star})\in \mathbb{X}\times\mathcal{Z} $ solution of \eqref{eq:subp1}, and hence one can define $ \mathscr{L}^S(\bz,\phi) :=(\bsi_{\star},\vec{\bu}_{\star})$. Alternatively, for each $ \bz\in \mathbf{Y} $ satisfying \eqref{eq:as.z2} and $ \phi\in \mathrm{Y} $ there exists a unique $ (\brho_{\star},\varphi_{\star})\in \mathbf{X}\times\mathrm{Y} $ solution of \eqref{eq:subp2}, and hence one can define $ \mathscr{L}^T(\bz,\phi) :=(\brho_{\star},\varphi_{\star})$.	Moreover, there hold
\begin{equation}\label{eq:apr1}
	\Vert\mathscr{L}^S(\bz,\phi)\Vert\,=\,\Vert\bsi_{\star}\Vert_{\bdiv_{6/5},\Omega}+
	\Vert\vec{\bu}_{\star}\Vert_{0,6;\Omega}\,\leq\,  \dfrac{2}{\alpha_{\mathcal{A}^S}}\,\Vert\phi\Vert_{0,6;\Omega}\,\Vert\bg\Vert_{0,3/2;\Omega}\,,
\end{equation}
and
\begin{equation}\label{eq:apr2}
	\Vert\mathscr{L}^T(\bz,\phi)\Vert\,=\,\Vert\brho_{\star}\Vert_{\div_{6/5},\Omega}+\Vert\varphi\Vert_{0,6;\Omega}\,\leq\, \dfrac{2}{\alpha_{\mathcal{A}^T}}\, \Vert\varphi_{D}\Vert_{1/2,\Gamma_{D}}\,.
\end{equation}
\end{lemma}
\begin{proof}
Given $ (\bz,\phi)\in \mathbf{Y}\times\mathrm{Y} $ with $ \bz $ satisfying \eqref{eq:as.z1} (res. \eqref{eq:as.z2}), bearing in mind the boundedness of $ \mathcal{A}_{\phi}^S $ and $ \mathscr{O}_{\phi}^{S} $ (res. $ \mathcal{A}_{\phi}^T $ and $ \mathscr{O}^{T}_{\phi} $), the global inf-sup condition \eqref{eq:inf.A.1} (res. \eqref{eq:inf.At.1}), a straightforward
application of the Banach-Nečas-Babuška Theorem (cf. \cite[Theorem 2.6]{eg-SPRINGER-2004}) guarantees the existence of a unique solution $ (\bsi_{\star},\vec{\bu}_{\star})\in \mathbb{X}\times\mathcal{Z} $ to \eqref{eq:subp1} (res. $ (\brho_{\star},\varphi_{\star})\in \mathbf{X}\times\mathrm{Y} $ to  \eqref{eq:subp2}). Furthermore, the corresponding a priori
estimate (cf. \cite[Theorem 2.6, eq. (2.5)]{eg-SPRINGER-2004}) and the boundedness of $ \mathcal{F}_{\phi}^S $ and $ \mathscr{F}^{T} $
provide \eqref{eq:apr1} and \eqref{eq:apr2}, respectively.
\end{proof}
\subsection{Solvability analysis of the fixed-point equation}
It follows from Lemma \ref{l_wel.S.T} that $ \mathscr{L} $ is well-defined, and thus, an application of \eqref{eq:def.opr.E}
and the estimates \eqref{eq:apr1}, \eqref{eq:apr2}, yields
\begin{equation}\label{eq:apr.E}
\begin{array}{c}
\big\Vert\mathscr{L}(\bz,\phi)\big\Vert \,=\, \big\Vert \big( \mathscr{L}_2^S(\bz , \mathscr{L}_2^T(\bz, \phi)), \mathscr{L}_2^T(\bz, \phi)\big) \big\Vert\\[2ex]
\,=\, \big\Vert \mathscr{L}_2^S(\bz , \mathscr{L}_2^T(\bz, \phi))\Vert+ \Vert\mathscr{L}_2^T(\bz, \phi) \big\Vert
\,\leq\,\left(1+\dfrac{2}{\alpha_{\mathcal{A}^S}}\Vert\bg\Vert_{0,3/2;\Omega}\right)\Vert\mathscr{L}_2^T(\bz, \phi) \big\Vert\\[2ex]
\,  \leq \, \left(1+\dfrac{2}{\alpha_{\mathcal{A}^S}}\Vert\bg\Vert_{0,3/2;\Omega}\right)\dfrac{2}{\alpha_{\mathcal{A}^T}}\,\Vert\varphi_{D}\Vert_{1/2,\Gamma_{D}}\,.
\end{array}
\end{equation}

We proceed with the analysis by determining conditions that are adequate for $ \mathscr{L} $ to map a closed ball of $ \mathbf{Y} \times\mathrm{Y} $ back into itself.  Indeed, given a radius $ r $ as
\begin{equation}\label{eq:def.r}
r \,:=\, \min\{r_1,r_2\}\,,
\end{equation}
where $ r_1 $ and $ r_2 $ are defined in \eqref{eq:as.z1} and \eqref{eq:as.z2}, respectively, 
we first set
\begin{equation}\label{de:ball}
\mathbf{W}(r)\,:=\, \Big\{(\bz,\phi)\in \mathbf{Y} \times\mathrm{Y} :\quad \Vert (\bz,\phi)\Vert\,:=\, \Vert\bz\Vert_{0,6;\Omega}+\Vert\phi\Vert_{0,6;\Omega}\,\leq \, r  \Big\}\,.
\end{equation}
Then, we have the following result.
\begin{lemma}\label{l_cons}
Assume that the data are sufficiently small so that
\begin{equation}\label{eq:as1.data}
\left(1+\dfrac{2}{\alpha_{\mathcal{A}^S}}\Vert\bg\Vert_{0,3/2;\Omega}\right)\dfrac{2}{\alpha_{\mathcal{A}^T}}\,\Vert\varphi_{D}\Vert_{1/2,\Gamma_{D}}\, \leq \, r  \,.
\end{equation}
Then, given $ r $ in \eqref{eq:def.r}, there holds $ \mathscr{L} ( \mathbf{W}(r))\subseteq \mathbf{W}(r)$.
\end{lemma}
\begin{proof}
It is a direct consequence of the a priori estimate \eqref{eq:apr.E} and the assumption \eqref{eq:as1.data}.
\end{proof}
Next, our objective is to demonstrate the Lipschitz continuity of the operator $  \mathscr{L} $, for which, as indicated in \eqref{eq:def.opr.E}, it is adequate to establish appropriate continuity properties for $ \mathscr{L}^S $ and $ \mathscr{L}^T $.
To that end, we need to assume further regularities on the solutions of the problems defining these operators.
More precisely, from now on we suppose that there exists $ \epsilon\geq\frac{n}{6} $ and positive constants $ C_{\epsilon} $, $ \widetilde{C}_{\epsilon} $ such that\\[1ex]
($ \textbf{RA} $)  for each $ (\bz,\phi)\in\mathbf{Y} \times\mathrm{Y} $ there holds $ \mathscr{L}^S(\bz,\phi)=(\bsi_{\star},\vec{\bu}_{\star}) $, where $ (\bsi_\star , \vec{\bu}_\star)=(\bsi_\star, (\bu_\star, \bgam_\star))  $ is the solution of \eqref{eq:sub_1} and $ (\bsi_\star, \bu_\star )\in (\mathbb{X} \cap \mathbb{H}^{\epsilon}(\Omega))\times \mathbf{W}^{\epsilon,6}(\Omega) $ and
\begin{equation}\label{eq:reg.t}
\Vert \bsi_\star\Vert_{\epsilon,\Omega}+\Vert\bu_\star\Vert_{\epsilon,6;\Omega}\, \leq\, C_{\epsilon}\, \Vert\bg\Vert_{0,3/2;\Omega}\,\Vert\phi\Vert_{0,6;\Omega}\,.
\end{equation}
($ \widetilde{\textbf{RA} }$)  for each $ (\bz,\phi)\in\mathbf{Y} \times\mathrm{Y} $ there holds $ \mathscr{L}^T(\bz,\phi)=(\brho_{\star},\varphi_{\star}) $, where $ (\brho_{\star},\varphi_{\star})  $ is the solution of \eqref{eq:sub_2} and $ (\brho_{\star},\varphi_{\star})\in (\mathbf{X} \cap \mathbf{H}^{\epsilon}(\Omega))\times \mathrm{W}^{\epsilon,6}(\Omega) $ and
\begin{equation}\label{eq:reg.ze}
\Vert \brho_\star\Vert_{\epsilon,\Omega}+\Vert\varphi_\star\Vert_{\epsilon,6;\Omega}\, \leq\, \widetilde{C}_{\epsilon}\, \Vert\varphi_D\Vert_{1/2+\epsilon,\Gamma_{D}}
\end{equation}
In this regard, we recall now from \cite[Theorem 1.4.5.2, part e)]{g-MA-1985} that for each $ \epsilon<\frac{n}{2} $ there hold $ \mathbb{H}^{\epsilon}(\Omega)\subseteq\mathbb{L}^{\epsilon^*}(\Omega) $ and $ \mathbf{H}^{\epsilon}(\Omega)\subseteq\mathbf{L}^{\epsilon^*}(\Omega) $, with continuous injections $ \textbf{i}_{\epsilon}: \mathbb{H}^{\epsilon}(\Omega)\rightarrow\mathbb{L}^{\epsilon^*}(\Omega)$ and $ i_{\epsilon}: \mathbf{H}^{\epsilon}(\Omega)\rightarrow\mathbf{L}^{\epsilon^*}(\Omega)$, respectively, where $ \epsilon^{*}=\frac{2n}{n-2\epsilon} $. Note that the indicated lower and upper bounds for the additional regularity $ \epsilon $, which turn out to require that $ \epsilon\in [\frac{n}{6}, \frac{n}{2}) $.
					
We now apply $ (\textbf{RA}) $ to establish the announced property of $ \mathscr{L}^S $.

\begin{lemma}\label{l_lip.S}
There exists a positive constant $ \mathcal{L}_{\mathscr{L}^S} $, depending on $ \mathcal{L}_{\mu} $, $ \mu_1 $, $ r_1 $, $ \mathcal{C}_{\epsilon} $, $ \alpha_{\mathcal{A}^S} $, such that
\begin{equation}\label{eq:lip.S}
\begin{array}{l}
\big\Vert \mathscr{L}^S(\bz,\phi)-\mathscr{L}^S(\by,\varpi)\big\Vert_{\mathbb{X}\times\mathcal{Z}}\, \leq \, \mathcal{L}_{\mathscr{L}^S}\, \Big\{\Vert\mathscr{L}^S_2(\by,\varpi)\Vert_{0,6;\Omega}\Vert\bz-\by\Vert_{0,6;\Omega}
\\[2ex]
\,+\,\Big(\Vert\bg\Vert_{0,3/2;\Omega}(1+\Vert\varpi\Vert_{0,6;\Omega})+\Vert\mathscr{L}^S_2(\by,\varpi)\Vert_{0,6;\Omega}\Big)\Vert\phi-\varpi\Vert_{0,6;\Omega}	
\Big\}\,.
\end{array}
\end{equation}
\end{lemma}
\begin{proof}
Given $ (\bz , \phi), (\by , \varpi)\in \mathbf{Y} \times\mathrm{Y} $, we let $ \mathscr{L}^S(\bz,\phi):=(\bsi_{\star},\vec{\bu}_{\star}) \in \mathbb{X}\times\mathcal{Z}$ and $ \mathscr{L}^S(\by,\varpi):=(\bsi_{\circ},\vec{\bu}_{\circ})\in \mathbb{X}\times\mathcal{Z} $, where $ (\bsi_{\star},\vec{\bu}_{\star})=(\bsi_{\star},(\bu_{\star},\bgam_{\star}))\in \mathbb{X}\times\mathcal{Z} $ and
$ (\bsi_{\circ},\vec{\bu}_{\circ})=(\bsi_{\circ},(\bu_{\circ},\bgam_{\circ}))\in \mathbb{X}\times\mathcal{Z} $ are the respective solutions of \eqref{eq:sub_1}, that is, 
\begin{equation*}
\begin{array}{c}
\mathcal{A}_{\phi}^S \big((\bsi_{\star}\,,\vec{\bu}_{\star}), (\bta,\, \vec{\bv})\big)+\mathscr{O}_{\phi}^{S}(\bz;\bu_\star,\bta)\,=\, \mathcal{F}_{\phi}^S(\bta,\,\vec{\bv})\quad\forall\, (\bta,\, \vec{\bv})\in \mathbb{X}\times\mathcal{Z}\,,\qan\\[2ex]
\mathcal{A}_{\varpi}^S \big((\bsi_{\circ}\,,\vec{\bu}_{\circ}), (\bta,\, \vec{\bv})\big)+\mathscr{O}_\varpi^S(\by;\bu_\circ,\bta)\,=\, \mathcal{F}_{\varpi}^S(\bta,\,\vec{\bv})\quad\forall\, (\bta,\, \vec{\bv})\in \mathbb{X}\times\mathcal{Z}\,.
\end{array}
\end{equation*}
It follows from the foregoing identities and the definition of $ \mathcal{A}_{\phi}^S $ (cf. first row of \eqref{def.A}) that
\begin{equation}\label{eq:re1}
\begin{array}{c}
\mathcal{A}_{\phi}^S \big((\bsi_\star-\bsi_\circ,\, \vec{\bu}_\star-\vec{\bu}_\circ), (\bta,\, \vec{\bv})\big)\,+\, \mathscr{O}_{\phi}^{S}(\bz;\bu_\star-\bu_{\circ} , \bta)\\[2ex]
\, =\, 
\Big[\mathcal{A}_{\phi}^S \big((\bsi_\star,\, \vec{\bu}_\star), (\bta,\, \vec{\bv})\big)+\mathscr{O}_{\phi}^{S}(\bz;\bu_\star , \bta)\Big]-
\Big[\mathcal{A}_{\phi}^S \big((\bsi_\circ,\, \vec{\bu}_\circ), (\bta,\, \vec{\bv})\big)+\mathscr{O}_{\phi}^{S}(\bz;\bu_\circ , \bta)\Big]\\[2.5ex]
\, =\, 	\mathcal{A}_{\varpi}^S \big((\bsi_\circ,\, \vec{\bu}_\circ), (\bta,\, \vec{\bv})\big)+\mathscr{O}_\varpi^S(\by;\bu_\circ,\bta)+(\mathcal{F}_{\phi}^S-\mathcal{F}_{\varpi}^S)(\bta,\, \vec{\bv})\\[2ex]
					-
\Big[\mathcal{A}_{\phi}^S \big((\bsi_\circ,\, \vec{\bu}_\circ), (\bta,\, \vec{\bv})\big)+\mathscr{O}_{\phi}^{S}(\bz;\bu_\circ , \bta)\Big]\\[2.5ex]
\, =\,\big(\mathscr{A}_{\varpi}^S-\mathscr{A}_{\phi}^{S}\big)(\bsi_{\circ},\bta)\,+\,
\big(\mathcal{F}_{\phi}^S-\mathcal{F}_{\varpi}^S\big)(\bta,\, \vec{\bv})\,+\,\mathscr{O}_{\varpi}^S(\by-\bz;\bu_{\circ},\bta)\,+\,\big(\mathscr{O}_{\varpi}^S-\mathscr{O}_{\phi}^{S}\big)(\bz;\bu_{\circ},\bta)\,,
\end{array}
\end{equation}
			
and hence, applying the global inf-sup condition \eqref{eq:inf.A.1} to $ (\bze,\vec{\bw}):=(\bsi_\star-\bsi_\circ,\, \vec{\bu}_\star-\vec{\bu}_\circ) $ , and employing \eqref{eq:re1}, we find that
\begin{equation}\label{eq0:l.lip.S}
\begin{array}{c}
 \dfrac{\alpha_{\mathcal{A}^S}}{2}\,\Vert(\bsi_\star-\bsi_\circ,\, \vec{\bu}_\star-\vec{\bu}_\circ)\Vert_{\mathbb{X}\times\mathcal{Z}}\\[2ex]
\, \leq \, \sup\limits_{\mathbf{0}\neq (\bta,\vec{\bv})\in \mathbb{X}\times\mathcal{Z}}\dfrac{\big(\mathscr{A}_{\varpi}^S-\mathscr{A}_{\phi}^{S}\big)(\bsi_{\circ},\bta)}{\Vert (\bta,\vec{\bv})\Vert_{\mathbb{X}\times\mathcal{Z}}}\\[3ex]
\, + \, \sup\limits_{\mathbf{0}\neq (\bta,\vec{\bv})\in \mathbb{X}\times\mathcal{Z}}\dfrac{
	\big(\mathcal{F}_{\phi}^S-\mathcal{F}_{\varpi}^S\big)(\bta,\, \vec{\bv})\,+\,\mathscr{O}_{\varpi}^S(\by-\bz;\bu_{\circ},\bta)\,+\,\big(\mathscr{O}_{\varpi}^S-\mathscr{O}_{\phi}^{S}\big)(\bz;\bu_{\circ},\bta)}{\Vert (\bta,\vec{\bv})\Vert_{\mathbb{X}\times\mathcal{Z}}}\,.
\end{array}
\end{equation}
								
Thus, bearing in mind the definitions of $ \mathscr{A}_{\phi}^{S} $ and $ \mathscr{A}_{\varpi}^S $ (cf. first row of \eqref{def-c-tilde-c}), and utilizing the Lipschitz-continuity of $ \mu $ (cf. first column of \eqref{eq:lip.mu.kap}) along with the Cauchy--Schwarz and H\"older inequalities, we have
\begin{equation}\label{eq2:l.lip.S}
\begin{array}{c}
\big|\big(\mathscr{A}_{\varpi}^S-\mathscr{A}_{\phi}^{S}\big)(\bsi_{\circ},\bta)\big|\,=\,\Big|\disp\int_{\Omega}\left(\dfrac{\mu(\phi)-\mu(\varpi)}{\mu(\phi)\,\mu(\varpi)}\right)\bsi_{\circ}^{\mathtt{d}}:\bta^{\mathtt{d}} \Big|\\[2.4ex]
\,\leq\, \dfrac{\mathcal{L}_{\mu}}{\mu_1^{2}}\, \big\Vert (\phi-\varpi)\,\bsi_{\circ}^{\mathtt{d}}\big\Vert_{0,\Omega}\, \big\Vert\bta\big\Vert_{0,\Omega}\\[2ex]
\,\leq\, \dfrac{\mathcal{L}_{\mu}}{\mu_1 ^{2}}\, \Vert \phi-\varpi\Vert_{0,2q;\Omega}\,\Vert\bsi_{\circ}\Vert_{0,2p;\Omega}\, \Vert\bta\Vert_{0,\Omega}\,,
\end{array}
\end{equation}		
								
where $ p,q\in (1,\infty) $ are conjugate to each other. Now, choosing $ p $ such that $ 2p=\epsilon^{\star} $ (cf. (3.23)), we
get $ 2q=\frac{n}{\epsilon} $, which, according to the range stipulated for $ \epsilon $, yields $ 2q<6 $, and thus the norm of the
embedding of $ \mathrm{L}^{6}(\Omega) $ into $ \mathrm{L}^{2q}=\mathrm{L}^{\frac{n}{\epsilon}}(\Omega) $ is given by $ \mathcal{C}_{\epsilon}=|\Omega|^{\frac{6\epsilon-n}{6n}} $. In this way, using additionally the
continuity of $ i_{\epsilon} $ along with the regularity estimate \eqref{eq:reg.t}, the inequality \eqref{eq2:l.lip.S} becomes		
\begin{equation}\label{eq3:l.lip.S}
\begin{array}{c}
\big|\big(\mathscr{A}_{\varpi}^S-\mathscr{A}_{\phi}^{S}\big)(\bsi_{\circ},\bta)\big|\,\leq\,\dfrac{\mathcal{L}_{\mu}}{\mu_1 ^{2}}\,\mathcal{C}_{\epsilon}\,\Vert\phi-\varpi\Vert_{0,6;\Omega}\,\Vert i_{\epsilon}\Vert\, \Vert\bsi_{\circ}\Vert_{\epsilon,\Omega}\, \Vert\bta\Vert_{\bdiv_{6/5},\Omega}\\[2ex]
\,\leq\, \dfrac{\mathcal{L}_{\mu}}{\mu_1 ^{2}}\,\mathcal{C}_{\epsilon}\,\Vert i_{\epsilon}\Vert\, C_{\epsilon}\, \Vert\bg\Vert_{0,3/2;\Omega}\,\Vert\phi-\varpi\Vert_{0,6;\Omega}\,\Vert\varpi\Vert_{0,6;\Omega}\, \Vert\bta\Vert_{\bdiv_{6/5},\Omega}\,.
\end{array}		
\end{equation}
On the other hand, according the definition of $ \mathscr{O}_{\phi}^{S} $, using the first inequality of \eqref{eq:hh}, and the boundedness of $ \Vert\bz\Vert_{0,6;\Omega} $ by $ r_1 $, we obtain
\begin{equation*}
\begin{array}{c}
\big|\big(\mathscr{O}_{\varpi}^S-\mathscr{O}_{\phi}^{S}\big)(\bz;\bu_{\circ},\bta)\big|\,\leq\, \dfrac{\mathcal{L}_{\mu}}{\mu_1^{2}}\,\Vert\phi- \varpi\Vert_{0,6;\Omega}\, \Vert\bz\Vert_{0,6;\Omega}\, \Vert\bu_{\circ}\Vert_{0,6;\Omega}\,  \Vert\bta\Vert_{0,\Omega}\\[2ex]
\,\leq\, \dfrac{\mathcal{L}_{\mu}}{\mu_1^{2}}\,r_1\,\Vert\phi- \varpi\Vert_{0,6;\Omega}\, \Vert\bu_{\circ}\Vert_{0,6;\Omega}\,  \Vert\bta\Vert_{0,\Omega}\,.
\end{array}
\end{equation*}
					
In turn, the boundedness of $ \mathscr{O}_{\varpi}^S(\bz;\cdot,\cdot) $ (cf. third row of \eqref{eq:bound.c.ct}), and the definition of $ \mathcal{F}_{\phi}^S $ (cf. \eqref{eq:def.F}) and H\"older's inequality, yield
\begin{equation}\label{eq3a:l.lip.S}
\big|\mathscr{O}_{\varpi}^S(\bz-\by;\bu_{\circ},\bta)\big|\,\leq\, \dfrac{1}{\mu_1}\,\Vert	\bz-\by\Vert_{0,6;\Omega}\,\Vert\bu_{\circ}\Vert_{0,6;\Omega}\,\Vert\bta\Vert_{0,\Omega}\,,
\end{equation}
and
\begin{equation}\label{eq4:l.lip.S}
\big|(\mathcal{F}_{\phi}^S-\mathcal{F}_{\varpi}^S)(\bta,\, \vec{\bv})\big|\, \leq \,	\Vert\bg\Vert_{0,3/2;\Omega}\, \Vert \phi-\varpi\Vert_{0,6;\Omega}\, \Vert\bv\Vert_{0,6;\Omega}\,.
\end{equation}

Hence, replacing \eqref{eq2:l.lip.S}-\eqref{eq4:l.lip.S} back into \eqref{eq0:l.lip.S}, we deduce the existence of a positive constant $ \mathcal{L}_{\mathscr{L}^S} $, depending on $ \mathcal{L}_{\mu} $, $ \mu_1 $, $ r_1 $, $ \mathcal{C}_{\epsilon} $, $ \alpha_{\mathcal{A}^S} $,
such that
\begin{equation*}
\begin{array}{c}
\Vert(\bsi_\star-\bsi_\circ,\, \vec{\bu}_\star-\vec{\bu}_\circ)\Vert_{\mathbb{X}\times\mathcal{Z}}\\[2ex]
\, \leq \, \mathcal{L}_{\mathscr{L}^S}\, \Big\{\Big(\Vert\bg\Vert_{0,3/2;\Omega}(1+\Vert\varpi\Vert_{0,6;\Omega})+\Vert\bu_{\circ}\Vert_{0,6;\Omega}\Big)\Vert\phi-\varpi\Vert_{0,6;\Omega}+\Vert\bz-\by\Vert_{0,6;\Omega}\,\Vert\bu_{\circ}\Vert_{0,6;\Omega}	\Big\}\,,
\end{array}
\end{equation*}
which recalling that $ \bu_{\circ} \, =\, \mathscr{L}_2^S(\by,\varpi) $, we arrive a the required inequality 	\eqref{eq:lip.S} with a positive
constant $ \mathcal{L}_{\mathscr{L}^S} $ as indicated.
\end{proof}
Here, we use $ (\widetilde{\textbf{RA}}) $ to prove the continuity of $ \mathscr{L}^T $.
\begin{lemma}\label{l_lip.T}
There exists a positive constant $ \mathcal{L}_{\mathscr{L}^T} $, depending on $ \mathcal{L}_{\kappa} $, $ \kappa_1 $, $ r_2 $, $ \widetilde{\mathcal{C}}_{\epsilon} $, $ \alpha_{\mathcal{A}^T} $, such that
\begin{equation}\label{eq:lip.S.t}
\begin{array}{l}
\big\Vert\mathscr{L}^T(\bz,\phi)-\mathscr{L}^T(\by,\varpi)\big\Vert\\[2ex]
\, \leq\, \mathcal{L}_{\mathscr{L}^T}\, \Big\{\Vert\mathscr{L}^T_2(\by,\varpi)\Vert_{0,6;\Omega}\,\Vert\bz-\by\Vert_{0,6;\Omega}\,+\,\left(1+\Vert\varphi_{D}\Vert_{1/2+\epsilon,\Gamma_{D}}\right)\Vert\phi-\varpi\Vert_{0,6;\Omega}\Big\}\,.
\end{array}
\end{equation}
\end{lemma}
\begin{proof}
Given $ (\bz , \phi), (\by , \varpi)\in \mathbf{Y} \times\mathrm{Y} $ as indicated, we proceed similarly to the proof of Lemma \ref{l_lip.S} and let $ \mathscr{L}^T(\bz,\phi):=(\brho_{\star},\varphi_{\star}) \in \mathbf{X}\times\mathrm{Y} $ and $ \mathscr{L}^T(\by,\varpi):=(\brho_{\circ},\varphi_{\circ}) \in \mathbf{X}\times\mathrm{Y} $, where $ (\brho_{\star},\varphi_{\star})\in \mathbf{X}\times\mathrm{Y} $ and $ (\brho_{\circ},\varphi_{\circ})\in \mathbf{X}\times\mathrm{Y} $ are the respective solutions of \eqref{eq:sub_2}, that is
\begin{equation}\label{eq1:lip.T}
\mathcal{A}_{\phi}^T \big((\brho_{\star}\,,\varphi_{\star}), (\bet,\, \psi)\big)+\mathscr{O}^{T}_\phi(\bz;\varphi_\star,\bet)\,=\, \mathcal{F}^T(\bet) \quad \forall\,(\bet,\, \psi)\in\mathbf{X}\times\mathrm{Y} \,,
\end{equation}
and
\begin{equation}\label{eq2:lip.T}
\mathcal{A}_{\varpi}^T \big((\brho_{\circ}\,,\varphi_{\circ}), (\bet,\, \psi)\big)+\mathscr{O}^{T}_{\varpi}(\by;\varphi_\circ,\bet)\,=\, \mathcal{F}^T(\bet)\quad\forall\,(\bet,\, \psi)\in\mathbf{X}\times\mathrm{Y}\,.
\end{equation}
Next, proceeding analogously to the derivation of \eqref{eq:re1}, we deduce from the identities \eqref{eq1:lip.T} and \eqref{eq2:lip.T}, along with the definition of $ \mathcal{A}_{\phi}^T $ (cf. \eqref{def.At})  that
\begin{equation}\label{eq0:lip.T}
\begin{array}{c}
\mathcal{A}_{\phi}^T \big((\brho_{\star}-\brho_{\circ}\,,\varphi_{\star}-\varphi_{\circ}), (\bet,\, \psi)\big)+\mathscr{O}^{T}_\phi(\bz;\varphi_\star-\varphi_\circ,\bet)\\[2ex]
\, =\,\big(\mathscr{A}^{T}_{\varpi}-\mathscr{A}^{T}_{\phi}\big)(\brho_{\circ},\bet)\,+\,\mathscr{O}^{T}_{\varpi}(\by-\bz;\varphi_{\circ},\bet)\,+\,\big(\mathscr{O}^{T}_{\varpi}-\mathscr{O}^{T}_{\phi}\big)(\bz;\varphi_{\circ},\bet)\,.
\end{array}
\end{equation}
							
Next, using the definition of $ \mathscr{A}^{T}_{\phi} $ (cf. third row of \eqref{def-c-tilde-c}) and the Lipschitz-continuity of $ \kappa $ (cf. second column of \eqref{eq:lip.mu.kap}), along with H\"older's inequality, we easily get
\begin{equation*}
\big|\big(\mathscr{A}^{T}_{\varpi}-\mathscr{A}^{T}_{\phi}\big)(\brho_{\circ},\bet)\big|\,\leq\, \dfrac{\mathcal{L}_{\kappa}}{\kappa_1^{2}}\, \Vert \phi-\varpi\Vert_{0,2q;\Omega}\,\Vert\brho_{\circ}\Vert_{0,2p;\Omega}\, \Vert\bet\Vert_{0,\Omega}\,.
\end{equation*}
Following a similar approach to the proof of Lemma \ref{l_lip.S} (cf. \eqref{eq2:l.lip.S}), we choose $ p=\frac{\epsilon^{\star}}{2} $ , so that $ q=\frac{n}{2\epsilon} $, and hence
\begin{equation}\label{eq1:lip.St}
\big|\big(\mathscr{A}^{T}_{\varpi}-\mathscr{A}^{T}_{\phi}\big)(\brho_{\circ},\bet)\big|\,\leq\,\dfrac{\mathcal{L}_{\kappa}}{\kappa_1^{2}}\,\widetilde{\mathcal{C}}_{\epsilon}\,\Vert i_{\epsilon}\Vert\, \Vert\varphi_{D}\Vert_{1/2+\epsilon,\Gamma_{D}}\, \Vert \phi-\varpi\Vert_{0,6;\Omega}\, \Vert\bet\Vert_{0,\Omega}\,,
\end{equation}

whereas employing the second inequality of \eqref{eq:hh} and the boundedness of $ \Vert\bz\Vert_{0,6;\Omega} $ by $ r_2 $, yields
\begin{equation}\label{eq2:lip.St}
\begin{array}{c}
\big|\big(\mathscr{O}^{T}_{\varpi}-\mathscr{O}^{T}_{\phi}\big)(\bz;\varphi_{\circ},\bet)\big|\,\leq\, \dfrac{\mathcal{L}_{\kappa}}{\kappa_1^{2}}\, \Vert\phi-\varpi\Vert_{0,6;\Omega}\,\Vert\bz\Vert_{0,6;\Omega}\, \Vert\varphi_{\circ}\Vert_{0,6;\Omega}\,\Vert\bet\Vert_{0,\Omega}\\[2ex]
\,\leq\, \dfrac{\mathcal{L}_{\kappa}}{\kappa_1^{2}}\, r_2 \, \Vert\phi-\varpi\Vert_{0,6;\Omega}\,\Vert\bet\Vert_{0,\Omega}\,.
\end{array}
\end{equation}
In turn, thanks to the boundedness of $ \mathscr{O}_\varpi^{T} $ (cf. fourth row of \eqref{eq:bound.c.ct}), we get
\begin{equation}\label{eq3:lip.St}
\big|\mathscr{O}^{T}_{\varpi}(\by-\bz;\varphi_{\circ},\bet)\big|\,\leq\, \dfrac{1}{\kappa_1}\, \Vert\bz-\by\Vert_{0,6;\Omega}\, \Vert\varphi_{\circ}\Vert_{0,6;\Omega}\,\Vert\bet\Vert_{0,\Omega}\,.	
\end{equation}
								
In this way, replacing \eqref{eq1:lip.St}- \eqref{eq3:lip.St} in \eqref{eq0:lip.T}, then employing the global inf-sup conditions \eqref{eq:inf.At.1}, and proceeding analogously as for the proof of Lemma \ref{l_lip.S} yields the desired result \eqref{eq:lip.S.t}.
\end{proof}
After demonstrating Lemmas \ref{l_lip.S} and \ref{l_lip.T}, our current objective is to prove the continuity property of the main operator $ \mathscr{L} $ within the closed ball $ \mathbf{W}(r) $ (cf. \eqref{de:ball}). In fact, given $ (\bz,\phi), (\by,\varpi)\in \mathbf{W}(r) $, initially, we note from the definition of $ \mathscr{L} $ (cf. \eqref{eq:def.opr.E}) that
\begin{equation*}
\begin{array}{c}
\mathscr{L}\big(\bz,\phi\big)-\mathscr{L}\big(\by,\varpi\big)\,:=\, \big(\mathscr{L}^S_2(\bz,\mathscr{L}^T_2(\bz,\phi))-\mathscr{L}^S_2(\by,\mathscr{L}^T_2(\by,\varpi)), \mathscr{L}^T_2(\bz,\phi)-\mathscr{L}^T_2(\by,\varpi)\big)\,.
\end{array}
\end{equation*}
Hence, 
employing the continuity properties of $ \mathscr{L}^S $ (cf. Lemma \ref{l_lip.S}, \eqref{eq:lip.S}) and $ \mathscr{L}^T $ (cf. Lemma \ref{l_lip.T}, \eqref{eq:lip.S.t}), we obtain
\begin{equation*}
\begin{array}{c}
\big\Vert	\mathscr{L}\big(\bz,\phi\big)-\mathscr{L}\big(\by,\varpi\big)\big\Vert\, \leq\, 		\big\Vert\mathscr{L}^S_2(\bz,\mathscr{L}^T_2(\bz,\phi))-\mathscr{L}^S_2(\by,\mathscr{L}^T_2(\by,\varpi))\big\Vert\,+\, \big\Vert \mathscr{L}^T_2(\bz,\phi)-\mathscr{L}^T_2(\by,\varpi)\big\Vert\\[2ex]
\,\leq\, \mathcal{L}_{\mathscr{L}^S}\, \Vert\mathscr{L}^S_2(\by,\mathscr{L}^T_2(\by,\varpi))\Vert_{0,6;\Omega}\Vert\bz-\by\Vert_{0,6;\Omega}
\\[2ex]
\,+\,\bigg( \mathcal{L}_{\mathscr{L}^S}\Big(\Vert\bg\Vert_{0,3/2;\Omega}\,(1+\Vert\mathscr{L}^T_2(\by,\varpi)\Vert_{0,6;\Omega})\\[2ex]
\,+\,\Vert\mathscr{L}^S_2(\by,\mathscr{L}^T_2(\by,\varpi))\Vert_{0,6;\Omega}\Big)+1\bigg)\,\big\Vert \mathscr{L}^T_2(\bz,\phi)-\mathscr{L}^T_2(\by,\varpi)\big\Vert_{0,6;\Omega}	\\[2ex]
\,\leq\, 
\mathcal{L}_{\mathscr{L}^S}\, \Vert\mathscr{L}^S_2(\by,\mathscr{L}^T_2(\by,\varpi))\Vert_{0,6;\Omega}\Vert\bz-\by\Vert_{0,6;\Omega}
\\[2ex]
\,+\,\mathcal{L}_{\mathscr{L}^T}\,\bigg( \mathcal{L}_{\mathscr{L}^S}\left(\Vert\bg\Vert_{0,3/2;\Omega}\,(1+\Vert\mathscr{L}^T_2(\by,\varpi)\Vert_{0,6;\Omega})+\Vert\mathscr{L}^S_2(\by,\mathscr{L}^T_2(\by,\varpi))\Vert_{0,6;\Omega}\right)+1\bigg)\\[2ex]
\, \times\,\Big(\Vert\mathscr{L}^T_2(\by,\varpi)\Vert_{0,6;\Omega}\,\Vert\bz-\by\Vert_{0,6;\Omega}\,+\,\left(1+\Vert\varphi_{D}\Vert_{1/2+\epsilon,\Gamma_{D}}\right)\Vert\phi-\varpi\Vert_{0,6;\Omega}\Big)\,,
\end{array}
\end{equation*}
from which, applying the a priori estimate of $ \mathscr{L}^S $ and $ \mathscr{L}^T $ (cf. Lemma \ref{l_wel.S.T}, eqs. \eqref{eq:apr1} and \eqref{eq:apr2}),
we infer the
existence of a positive constant $ \mathcal{L}_{\mathscr{L}} $, depending only on $ \mathcal{L}_{\mathscr{L}^S} $, $ \mathcal{L}_{\mathscr{L}^T} $, $ \alpha_{\mathcal{A}^S} $, $\alpha_{\mathcal{A}^T} $, $ r $, such that
\begin{equation}\label{eq:lip.con.Xi}
\begin{array}{c}
\big\Vert	\mathscr{L}\big(\bz,\phi\big)-\mathscr{L}\big(\by,\varpi\big)\big\Vert\\[2ex]
\, \leq\, 	\mathcal{L}_{\mathscr{L}}\Big\{
\Big(\Vert\bg\Vert_{0,3/2;\Omega}\,+\,\left(1+\Vert\bg\Vert_{0,3/2;\Omega}(1+\Vert\varphi_{D}\Vert_{1/2,\Gamma_{D}})\right)\Vert\varphi_{D}\Vert_{1/2,\Gamma_{D}}\Big)\Vert\bz-\by\Vert_{0,6;\Omega}\\[2ex]
\,+\,\Big(1+\Vert\bg\Vert_{0,3/2;\Omega}(1+\Vert\varphi_{D}\Vert_{1/2,\Gamma_{D}})\Big)\big(1+\Vert\varphi_{D}\Vert_{1/2+\epsilon,\Gamma_{D}}\big)\Vert\phi-\varpi\Vert_{0,6;\Omega}
\Big\}\,,		
\end{array}
\end{equation}
for all $ (\bz,\phi), (\by,\varpi)\in \mathbf{W}(r) $.
						
Below is a summary of the main result of this section.
\begin{theorem}\label{t:exis}
Assume that the data are sufficiently small so that \eqref{eq:as1.data} hold. In addition, suppose that
\begin{equation}\label{eq:as2.data}
\begin{array}{c}
	\mathcal{L}_{\mathscr{L}}\Big\{
\Big(\Vert\bg\Vert_{0,3/2;\Omega}\,+\,\left(1+\Vert\bg\Vert_{0,3/2;\Omega}(1+\Vert\varphi_{D}\Vert_{1/2,\Gamma_{D}})\right)\Vert\varphi_{D}\Vert_{1/2,\Gamma_{D}}\Big)\\[2ex]
\,+\,\Big(1+\Vert\bg\Vert_{0,3/2;\Omega}(1+\Vert\varphi_{D}\Vert_{1/2,\Gamma_{D}})\Big)\big(1+\Vert\varphi_{D}\Vert_{1/2+\epsilon,\Gamma_{D}}\big)
\Big\}\, <\, 1\,.
\end{array}
\end{equation}
								
Then, the operator $ \mathscr{L} $ has a unique fixed point $ (\bu , \varphi)\in \mathbf{W}(r) $. Equivalently, the coupled problem \eqref{eq:v1a-eq:v4-a}-\eqref{eq:v1a-eq:v4-d} has a unique solution $ (\bsi,\vec{\bu})\in \mathbb{X}\times\mathcal{Z} $ and $ (\brho,\varphi)\in  \mathbf{X}\times\mathrm{Y} $ with $ (\bu , \varphi)\in \mathbf{W}(r) $. Moreover, there hold the following a priori estimates
\begin{equation}\label{eq:apr1.prob1}
\begin{array}{c}
\Vert (\bsi,\vec{\bu})\Vert\, \leq\, \dfrac{4}{\alpha_{\mathcal{A}^S}\,\alpha_{\mathcal{A}^T}}\,\Vert\varphi_{D}\Vert_{1/2,\Gamma_{D}}\Vert\bg\Vert_{0,3/2;\Omega}\,,
	\end{array}
\end{equation}
and
\begin{equation}\label{eq:apr2.prob2}
\begin{array}{c}
\Vert(\brho,\varphi)\Vert\,\leq\, \dfrac{2}{\alpha_{\mathcal{A}^T}}\,\Vert\varphi_{D}\Vert_{1/2,\Gamma_{D}}\,.	\end{array}
\end{equation}
\end{theorem}
\begin{proof}
It is clear from Lemma \ref{l_cons} and the assumption \eqref{eq:as2.data} that $ \mathscr{L} $ is a contraction that maps the ball
$ \mathbf{W}(r) $ into itself, and hence a straightforward application of the classical Banach fixed--point theorem
 implies the existence of a unique solution  $ (\bsi,\vec{\bu})\in \mathbb{X}\times\mathcal{Z} $, $ (\brho,\varphi)\in  \mathbf{X}\times\mathrm{Y} $ to \eqref{eq:v1a-eq:v4-a}-\eqref{eq:v1a-eq:v4-d} with $ (\bu , \varphi)\in \mathbf{W}(r) $.
Furthermore, the a priori estimates given in \eqref{eq:apr1.prob1} and \eqref{eq:apr2.prob2} follow straightforwardly from \eqref{eq:apr1} and \eqref{eq:apr2}, respectively.
\end{proof}									
\section{Mixed virtual element approximation}\label{sec.4}
										
In this section, we describe the virtual element discretization for the 2D version of  \eqref{eq:v1a-eq:v4-a}-\eqref{eq:v1a-eq:v4-d} on general polygonal meshes. 
										
\subsection{Preliminaries}
										
The aim of this part is to present some preliminary concepts and results to be employed in what follows. For more details, 
we refer to \cite{bfm-M2AN-2014,bbcmmr-MMMAS-2013,cg-IMANUM-2017}. We begin by letting $\{\Omega_{h}\}_{h > 0}$ 
be a sequence of partitions of $\Omega$ into general polygons $E$ 
with diameter and number of edges denoted by $h_{E}$ and $d_E$, respectively, and set, 
as usual, $h \,:=\, \max_{E \in \Omega_h} h_E$.
Then, we let
$ N_h^{\mathtt{ed}} $ and $ N_h^{\mathtt{el}} $ be the number of edges and elements, respectively, and
$\mathbf{n}_{e}^{E}$ be the unit outward normal on
edge $e \subset\partial E$. Also, we denote the edges of $ \partial E $ by $ e $, its length by $ h_{e}:=|e| $ and the set of edges $ e $ of $ \Omega_h $ by $ \Gamma_h $.
For any $ l\in N $ and
any mesh object $ \Theta\in \Omega_{h}\cup\Gamma_h $,  let $ \mathrm{P}_{l}(\Theta) $, $ \mathbf{P}_{l}(\Theta) $, $ \mathbb{P}_{l}(\Theta) $ be the space of scalar, vectorial and matrix polynomials deﬁned
on $ \Theta $ of degree less than or equal to $ l $, respectively (with the extended notation $ \mathrm{P}_{-1}(\Theta)=\{0\} $).
The dimension of such spaces, for each $ E\in\Omega_h $ and $ e\in \Gamma_h $, are
\[
\dim\big(\mathrm{P}_{l}(E)\big)\, =\, \pi_l^{\mathtt{el}} :=\dfrac{(l+1)(l+2)}{2}\quad \dim\big(\mathbf{P}_{l}(E)\big)\, =\,2\,\pi_l^{\mathtt{el}}\quad \dim\big(\mathbb{P}_{l}(E)\big)\, =\,4\,\pi_l^{\mathtt{el}}
\]
and
\[
\dim\big(\mathrm{P}_{l}(e)\big)\, =\, \pi_l^{\mathtt{ed}} :=l+1\quad \dim\big(\mathbf{P}_{l}(e)\big)\, =\,2\,\pi_l^{\mathtt{ed}}\quad \dim\big(\mathbb{P}_{l}(e)\big)\, =\,4\,\pi_l^{\mathtt{ed}}
\]
										
Also, for any  $ l\in N $ we consider the broken space
\begin{equation}\label{def-broken-P-ell-h}
\mathrm{P}_{\ell}(\Omega_h)~:=~\Big\{ v\in L^{2}(\Omega): 
\quad v|_{E}\in\mathrm{P}_{\ell}(E),\quad\forall~ E\in\Omega_h \Big\} \,.
\end{equation} 
										
In addition, a family of meshes $\{\Omega_{h}\}_{h > 0}$  is said to be regular if there exists a non-negative real number $ \rho $ independent of $ h $, such that										
\begin{itemize}[leftmargin=.6in]
\item[$ \mathbf{(A1)} $] 
every polygonal cell $ E\in \Omega_{h} $ is star-shaped with respect to every point of a
disk with radius $ \rho h_E $.
\item[$ \mathbf{(A2)} $] every edge $ e\subset\partial E $ of 
cell $E\in \Omega_h$, satisfies $h_e \ge \, \rho\, h_{E}$.
\end{itemize}
										
The regularity assumptions $ \mathbf{(A1)} $-$ \mathbf{(A2)} $ allow us to use meshes with cells having quite general geometric shapes. 
										
A key ingredient in the VEM construction is represented by the polynomial projections that will play a fundamental
role in the construction of the approximated virtual elements form. 
To this end, let us introduce $ \mathrm{L}^{1} $-projection operator $\Pcal_{\ell}^E: \L^{1}(E)\rightarrow\mathrm{P}_{\ell}(E) $, which is the (unique) solution of the following variational problem, for any function $ v \in \L^{1}(E) $:
\begin{equation}\label{def:l2proj}
\int_{E}\big(\mathcal P_{\ell}^E(v)-v\big)\, q \,=\,0\,,\qquad \forall~ q\in \mathrm{P}_{\ell}(E). 
\end{equation}									
We note that the definition of scaled projector $ \mathrm{P}_{\ell}^E $ can be extended to the vectorial and tensorial versions, and denoted by $\Pcalbf_{\ell}^E : \mathbf{L}^{1}(E)\rightarrow\mathbf{P}_{\ell}(E)$ 
and $\Pcalbb_{\ell}^E : \mathbb{L}^{1}(E)\rightarrow\mathbb{P}_{\ell}({K})$, respectively.
Under the assumption $ \mathbf{(A1)} $-$ \mathbf{(A2)} $, the following estimates can be obtained for the projections $\Pcal_{\ell}^E$, $\Pcalbf_{\ell}^E$, and $\Pcalbb_{\ell}^E$ (cf. \cite[Lemma 3.1]{gs-M3AS-2021}).
\begin{lemma}\label{l1}
Assume that Assumption $ \mathbf{(A1)} $-$ \mathbf{(A2)} $ is satisfied. Then, for all $ v\in \mathrm W^{s,p}(E) $, $ \bet\in \mathbf W^{s,p}(E) $, $ \bta\in \mathbb W^{s,p}(E) $  and $ E\in \Omega_{h} $, there
holds
\begin{equation}\label{eq_Q0}
\begin{array}{rcll}
|v-\Pcal_{\ell}^E(v)|_{m,p;E} &\le& C_{\ell}\,h^{s-m}_K \,|v|_{s,p;E} & \quad \forall\, m,s,p\in N, \, p>1, \, \, m \le s \le \ell+1\,,\\[2ex]
|\bet-\Pcalbf_{\ell}^E(\bet)|_{m,p;E} &\le& C_{\ell}\,h^{s-m}_K \,|\bet|_{s,p;E} & \quad \forall\, m,s,p\in N, \, p>1, \, \, m \le s \le \ell+1\,,  \\[2ex]
|\bta-\Pcalbb_{\ell}^E(\bta)|_{m,p;E} &\le& C_{\ell}\,h^{s-m}_K \,|\bta|_{s,p:E} & \quad \forall\, m,s,p\in N, \, p>1, \, \, m \le s \le \ell+1 \,. 
\end{array}
\end{equation}
\end{lemma}
										
\medskip
										
It is easy to see that these projection operators are continuous.
According to Lemma \ref{l1}, by taking  $m = s$ in \eqref{eq_Q0} we can deduce that there exist constant $M_{\ell}$, depending only on $\ell$ and $C_{\Omega}$, such that for each $E \in \Omega_h$ there holds
\begin{equation}\label{eq:bound.Pi}
\begin{array}{rcll}
|\Pcal_{\ell}^E(v)|_{s,p;E} &\le& M_{\ell}\,|v|_{s,p;E} & \quad \forall\, v\in \mathrm W^{s,p}(E) \,, \\[2ex]
|\Pcalbf_{\ell}^E(\bet)|_{s,p;E} &\le& M_{\ell}\,|\bet|_{s,p;E} & \quad \forall\, \bet\in \mathbf W^{s,p}(E)\,,  \\[2ex]
|\Pcalbb_{\ell}^E(\bta)|_{s,p;E} &\le& M_{\ell}\,|\bta|_{s,p:E} & \quad \forall\, \bta\in \mathbb W^{s,p}(E) \,. 
\end{array}
\end{equation}
										
\medskip
Finally, for any element $E \in \Omega_h$ and functions $ v\in \L^1(\Omega) $, $ \bet\in \mathbf L^1(\Omega) $, $ \bta\in \mathbb L^{1}(\Omega) $, the global projection operators $ \Pcal^h_\ell $, $ \Pcalbf^h_\ell $ and $ \Pcalbb^h_\ell $ 
are defined by
\begin{equation*}
\begin{array}{c}
\disp
\Pcal^h_\ell(v)|_E \,=\, \Pcal^E_\ell(v|_K)  \quad
\Pcalbf^h_\ell(\bet)|_E \,=\, \Pcalbf^E_\ell(\bet|_E) \quad \hbox{and}\quad \Pcalbb^h_\ell(\bta)|_E \,=\, \Pcalbb^E_\ell(\bta|_E)\,.
\end{array}
\end{equation*}
										
\subsection{Discrete spaces}\label{section32}
Here we discuss appropriate choices for both virtual and non-virtual approximation spaces for pairs $\big(\mathbb{X},
\mathbf{X}\big)$ and $\big(\mathbf{Y},\mathrm{Y}\big)$, respectively. We set notations
\[
\operatorname{rot}(\bta)\,:=\,\partial_x \eta_{2}-\partial_y\eta_1 \qan\operatorname{\mathbf{rot}}(\bta)~:=~(\partial_x\tau_{12}-\partial_y\tau_{11}, \partial_x\tau_{22}-\partial_y\tau_{21})^{\mathtt{t}}\,,
\]
for all sufficiently smooth vector $\boldsymbol{\eta}$ and tensor $ \bta$,
and begin by defining the virtual element spaces including vectorial and tensorial functions for approximation of the pseudostress and pseudoheat ﬁelds, for each $E \in \Omega_h$ and $ r\geq 0 $ (see e.g., \cite{cg-IMANUM-2017}):  
\begin{equation}\label{eq:4a}
\begin{array}{ll}
\mathbf{X}_r^{E}\,:=\,\Big\{ \bet\in \mathbf{H}(\div_{6/5};E)\cap \mathbf{H}(\operatorname{rot};E):
\quad  
&\div(\bet)\in\mathrm{P}_r(E)\,,\\[1ex] &\operatorname{rot}(\bet)\in\mathrm{P}_{r-1}(E)\qan\\[1ex]
&(\boldsymbol{\eta}\cdot\mathbf{n}_{e}^{E})|_e \in\mathrm{P}_r(e)\,\,\,\forall\, e\subset \partial E\Big\}\,,
\end{array}
\end{equation}
and
\begin{equation}\label{eq:4}
\begin{array}{ll}
\mathbb{X}_r^{E}\,:=\,\Big\{ \bta\in \mathbb{H}(\bdiv_{6/5};E)\cap 
\mathbb{H}(\operatorname{\mathbf{rot}};E):
\quad
&\bdiv(\bta)\in\mathbf{P}_r(E)\,,\\[1ex] &\operatorname{\mathbf{rot}}(\bta)\in\mathbf{P}_{r-1}(E)\qan\\[1ex]
&(\bta\mathbf{n}_{e}^{E})|_e \in\mathbf{P}_r(e)\,\,\,\forall\, e\subset \partial E
\Big\}\,.
\end{array}
\end{equation} 
												
In addition, according to \cite{cg-IMANUM-2017}, it is easy to see that the dimension of spaces	$ \mathbf{X}_r^{E} $ and $ \mathbb{X}_r^{E} $ is equal to
\[
\pi_r^{\mathtt{ed}}+(\pi_r^{\mathtt{el}}-1)+\big((r-1)(r+2)+2\big)/2 \qan
2\pi_r^{\mathtt{ed}}+2(\pi_r^{\mathtt{el}}-1)+\big((r-1)(r+2)+2\big)\,,
\]
respectively. This leads to introducing the following degrees of freedom; obviously, these degrees of freedom (DoF) are unisolvent for 
$ \mathbf{X}_r^{E} $ and $ \mathbb{X}_r^{E} $ (see \cite{cg-IMANUM-2017,cgs-SIAM.NA-2018,gs-M3AS-2021}):
\begin{itemize}
\item[•]  the edge moments	
\begin{equation}\label{DOF-m-E.1}
\begin{array}{l}
\widetilde{\mathbf{D1}}(\boldsymbol{\eta})\,:=\, \disp \int_{e}(\boldsymbol{\eta}\cdot\mathbf{n}_e^E)\, q 
\qquad\forall\, q\in \mathrm{P}_{r}(e) \qan\\[2ex]
\mathbf{D1}(\bta)\,:=\, \disp \int_{e}\bta\mathbf{n}_e^E \cdot \mathbf{q} 
\qquad \forall\, \mathbf{q}\in \mathbf{P}_{r}(e)\,,
\end{array}
\end{equation}
\item[•]  the element moments of the gradient
\begin{equation}\label{DOF-m-E.2}
\begin{array}{l}
\widetilde{\mathbf{D2}}(\bet) \,:=\,\disp \int_{ E}\bet\cdot\nabla q 
\qquad \forall\, q\in \mathrm{P}_{r}(E)\backslash\{1\}\qan\\[2ex]
\mathbf{D2}(\bta):= \disp \int_{ E}\bta:\boldsymbol{\nabla} \mathbf{q} 
\qquad \forall\, \mathbf{q}\in \mathbf{P}_{r}(E)\backslash\{(1,0)^{\mathtt{t}},(0,1)^{\mathtt{t}}\}\,,
\end{array}
\end{equation}
\item[•]  the element moments
\begin{equation}\label{DOF-m-E.3}
\begin{array}{l}												\widetilde{\mathbf{D3}}(\boldsymbol{\eta})\,:=\, \disp\int_{E}\boldsymbol{\eta}\cdot \mathbf q 
\qquad\forall\, \mathbf q\in {\mathbf P}_{r}^{\bot}(E)\qan\\[2ex]
\mathbf{D3}(\bta)\,:=\, \disp \int_{E}\bta: \bxi
\qquad \forall\, \bxi \in {\mathbb P}_{r}^{\bot}(E)\,,
\end{array}
\end{equation}
\end{itemize}
												
where ${\mathbf P}_{r}^{\bot}(E)$ and ${\mathbb P}_{r}^{\bot}(E)$ are subspaces of
${\mathbf P}_{r}(E)$ and ${\mathbb P}_{r}(E)$, respectively, such that satisfy
\[
{\mathbf P}_{r}(E) = \nabla \mathrm P_{r+1}(E) \oplus {\mathbf P}_{r}^{\bot}(E) \qan
{\mathbb P}_{r}(E) = \nabla \mathbf P_{r+1}(E) \oplus {\mathbb P}_{r}^{\bot}(E)\,.
\]
Thanks to these DoFs, it is possible to compute the following divergence and projection operators:
	\[
\div : \mathbf{X}_r^{E}\rightarrow \mathrm{P}_r(E)\,,\quad 
\bdiv: \mathbb{X}_r^{E}\rightarrow \mathbf{P}_r(E)\,,
\]
and
\[
\Pcalbf^E_r: \mathbf{X}_r^{E}\rightarrow \mathbf{P}_r(E)\,,\quad
\Pcalbb^E_r: \mathbb{X}_r^{E}\rightarrow \mathbb{P}_r(E)\,.
\]
The global virtual element subspaces of $\mathbf X$, $\mathbb X$ are obtained by gluing such local spaces $ \mathbf X_r^{E} $ and $ \mathbb X_r^{E} $, respectively, that is,
\begin{subequations}
\begin{align}
\mathbf X_{h}\,:=\, \Big\{ \bet\in \mathbf X: \quad \bet|_E \,\in\, \mathbf X_r^{E} \quad\forall\, E \in \Omega_h \Big\}\,, \label{def-mathbf-X-h}\\[2ex]
\mathbb X_{h} \,:=\, \Big\{ \bta\in \mathbb X: \quad \bta|_E \,\in\, \mathbb X_r^{E} \quad\forall\, E \in \Omega_h \Big\}\,.\label{def-mathbb-X-h}
\end{align}
\end{subequations}
On the other hand, for approximating the tempreture, velocity and vorticity unknowns we simply consider the piecewise polynomial spaces of degree $ r $ (cf. \eqref{def-broken-P-ell-h}), that is
\begin{equation}\label{eq:5a}
\begin{array}{c}
\mathrm{Y}_{h} \,=\, \mathrm P_r(\Omega_h)\,,\quad	\mathbf{Y}_{h}\,=\, \mathbf P_r(\Omega_h)\,,\qan
\mathbb{Z}_{h}\,:=\,\mathbb P_r(\Omega_h)  \cap\mathbb{L}_{\mathtt{skew}}^{2}(\Omega)\,.
\end{array}
\end{equation}
												
We remark from the foregoing definitions that there hold (cf. \cite{cgs-MMMAC-2017})
\begin{equation}\label{eq:propA}
\div(\mathbf{X}_{h}) \,\subseteq\, \mathrm{Y}_{h} \qan
\bdiv(\mathbb{X}_{h}) \,\subseteq\, \mathbf{Y}_{h}\,.
\end{equation}
												
\subsection{Discrete virtual forms and the discrete problem}\label{sub.sec.4.3} 
In order to solve problem \eqref{eq:prob1}-\eqref{eq:prob5a} via the virtual element method, we need to define a suitable set of discrete forms for the problem at hand. Such forms are constructed element-by-element and depend only on the local degrees of freedom $ \mathbf{D1} $-$ \mathbf{D3} $ and $ \widetilde{\mathbf{D1}} $-$ \widetilde{\mathbf{D3}} $, also via the projection operators $ \Pcalbb_{r} $ and $ \Pcalbf_{r} $.
In this regard, we first notice from the definitions of the discrete spaces (cf. \eqref{def-mathbf-X-h} - \eqref{def-mathbb-X-h}, \eqref{eq:5a}), for each  $ (\bet_h , \psi_h)\in \mathbf{X}_h\times\mathrm{Y}_h $, the quantity $\mathscr{B}^{T}(\bet_h,\psi_h)$
is fully computable as such, without further modifications. The same is valid for the linear functionals $\mathcal{F}_{\varphi_h}^F\in (\mathbb{X}_h \times\mathcal{Z}_h)^{'}$
and $\mathscr{F}^{T}\in \mathbf{X}_h^{'}$.
Whereas, given $\bz_h \in \mathbf Y_{h}$ and $ \varphi_h\in \mathrm{Y}_h $,  for the arbitrary vectorial and tensorial functions in $ \mathbf{X}_h $ and $ \mathbb{X}_h $ the bilinear forms 
\[
\mathscr{A}_{\varphi_h}^{S}(\cdot,\cdot)\,,\quad \mathscr{B}^S(\cdot,\cdot)\,, \quad \mathscr{O}_{\varphi_h}^S(\bz_h;\cdot,\cdot)\qan 
\mathscr{A}^{T}_{\varphi_h}(\cdot,\cdot)\,,\quad \mathscr{O}^{T}_{\varphi_h}(\bz_h;\cdot,\cdot)\,,
\]
 are not computable, respectively since the discrete functions are not known in closed form.
In order to overcome this difficulty, and thinking first of $\mathscr{A}_{\phi_h}^{S}$ and $\mathscr{A}^{T}_{\phi_h}$, we now resort to the set of degrees of freedom $\big\{\mathscr{M}_j^{S,E}\big\}_{j=1}^{n^E_r}$ and $\big\{\mathscr{M}^{T,E}_j\big\}_{j=1}^{\widetilde n^E_r}$,
and for each $ \phi_h\in \mathrm{Y}_h $ let $\mathscr{S}_{\phi_h}^{S;E} : \mathbb X_r^{E} \times \mathbb X_r^{E} \rightarrow \R$ and $\mathscr{S}_{\phi_h}^{T,E} : \mathbf X_r^{E} \times \mathbf X_r^{E} \rightarrow \R$
be the bilinear forms defined by
\begin{equation*}\label{s-K-widetilde-s-E}
\begin{array}{c}
	\disp
	\mathscr{S}_{\phi_h}^{S,E}\left(\bze_h,\bta_h\right)\,:=\,\dfrac{1}{\big|\mu(\mathcal P_{0}^E(\phi_h))\big|}\, \sum_{k=1}^{n^E_r} 
	\mathscr{M}^{S,E}_k(\bze_h) \, \mathscr{M}^{S,E}_k(\bta_h) \qquad\forall\, \bze_h, \, \bta_h \in \mathbb X_r^{E}\,, \\[2.5ex]
\disp
\mathscr{S}_{\phi_h}^{T,E}\left(\bxi_h,\bet_h\right)\,:=\, \dfrac{1}{\big|\kappa(\mathcal P_{0}^E(\phi_h))\big|}\,\sum_{k=1}^{\widetilde n^E_r} 
\mathscr{M}^{T,E}_k(\bxi_h) \, \mathscr{M}^{T,E}_k(\bet_h) \qquad\forall\, \bxi_h, \, \bet_h \in \mathbf X_r^{E} \,.
\end{array}
\end{equation*}

Then, recalling the continuous global forms $ \mathscr{A}_{\phi_h}^{S} $, $\mathscr{B}^S$, $\mathscr{O}_{\phi_h}^S$, $ \mathscr{A}^{T}_{\phi_h} $ and $\mathscr{O}^{T}_{\phi_h}$ (cf. \eqref{def-c-tilde-c}), and letting $ \mathscr{A}_{\phi_h}^{S,E} $, $\mathscr{B}^{S,E}$, $\mathscr{O}_{\phi_h}^{S,E}$, $ \mathscr{A}^{T,E}_{\phi_h} $ and $\mathscr{O}^{T,E}_{\phi_h}$ be the associated local versions, respectively for each $E \in \Omega_h$, we define their computable versions as
\begin{subequations}													\begin{align}
\mathscr{A}_{h,\phi_h}^{S,E}(\bze_h,\bta_h) &\,:=\, \mathscr{A}_{\phi_h}^{S,E}\left(\Pcalbb_{r}^{E}(\bze_h),\Pcalbb_{r}^{E}(\bta_h)\right)
\,+\,  \mathscr{S}_{\phi_h}^{S,E}\left(\bze_h - \Pcalbb_{r}^{E}(\bze_h),\bta_h - \Pcalbb_{r}^{E}(\bta_h)\right)\,,\label{eq:defAF.h}\\[2ex]													\mathscr{A}^{T,E}_{h,\phi_h}(\bxi_h,\bet_h) &\,:=\, \mathscr{A}^{T,E}_{\phi_h}\left(\Pcalbf_{r}^{E}(\bxi_h),\Pcalbf_{r}^{E}(\bet_h)\right)
\,+\, \mathscr{S}_{\phi_h}^{T,E}\left(\bxi_h - \Pcalbf_{r}^{E}(\bxi_h),\bet_h - \Pcalbf_{r}^{E}(\bet_h)\right)\,,\label{eq:defAT.h}\\[2ex]
\mathscr{O}^{S,E}_{h,\phi_h}(\bz_h;\bw_h,\bta_h)&\,:=\, \mathscr{O}_{\phi_h}^{S,E}(\bz_h;\bw_h,\Pcalbb_{r}^{E}(\bta_h))\,,\label{eq:defO.h}\\[2ex]
\mathscr{O}^{T,E}_{h,\phi_h}(\bz_h;\omega_h,\bet_h)&\,:=\, \mathscr{O}^{T,E}_{\phi_h}(\bz_h;\omega_h,\Pcalbf_{r}^{E}(\bet_h))\,,\label{eq:defOT.h}\\[2ex]
\mathscr{B}_{h}^{S,E}(\bta_h ,\vec{\bv}_h )&\,:=\, \disp  \int_{\Omega}\bv_h\cdot\bdiv(\bta_h)+\int_{\Omega}\bom_h:\Pcalbb_{r}^{E}(\bta_h)\,,\label{eq:def.bh}		
\end{align}
\end{subequations}
												
for each $ \bze_h,\bta_h \in \mathbb{X}_r^{E}$ and $ \bxi_h,\bet_h \in \mathbf{X}_r^{E}$.
In the usual way,  for each $(\bz_h,\phi_h)\in \mathbf{Y}_h \times\mathrm{Y}_h$ the global bilinear forms 
 are given by
\begin{subequations}
\begin{align}														\mathscr{A}_{h,\phi_h}^{S}(\cdot,\cdot): \mathbb{X}_h\times\mathbb{X}_h\rightarrow\mathrm{R}& 
	\qquad
	\mathscr{A}_{h,\phi_h}^{S}(\bze_h,\bta_h) \,:=\, \sum_{E\in\Omega_h} \mathscr{A}_{h,\phi_h}^{S,E}( \bze_h,\bta_h)			\,,\label{def-a-h-virtual}\\[1ex]							\mathscr{A}^{T}_{h,\phi_h}(\cdot,\cdot): \mathbf{X}_h\times\mathbf{X}_h\rightarrow\mathrm{R}&\qquad
	\mathscr{A}^{T}_{h,\phi_h}(\bxi_h,\bet_h) \,:=\, \sum_{E\in\Omega_h} \mathscr{A}^{T,E}_{h,\phi_h}( \bxi_h,\bet_h)	\,,\label{def-widetilde-a-h-virtual}\\[1ex]
\mathscr{O}^{S}_{h,\phi_h}(\bz_h;\cdot,\cdot):\mathbf{Y}_{h}\times\mathbb{X}_h\rightarrow\mathrm{R}&\qquad\mathscr{O}^{S}_{h,\phi_h}(\bz_h ;\bw_h,\bta_h) \,:=\, \sum_{E\in\Omega_h}\mathscr{O}^{S,E}_{h,\phi_h}(\bz_h ;\bw_h,\bta_h) \,,\label{def-o-h-virtual}\\[1ex]
\mathscr{O}^{T}_{h,\phi_h}(\bz_h;\cdot,\cdot):\mathrm{Y}_{h}\times\mathbf{X}_h\rightarrow\mathrm{R}&\qquad\mathscr{O}^{T}_{h,\phi_h}(\bz_h ;\omega_h,\bet_h) \,:=\, \sum_{E\in\Omega_h}\mathscr{O}^{T,E}_{h,\phi_h}(\bz_h ;\omega_h,\bet_h) \,,\qan\label{def-ot-h-virtual}\\[1ex]
\mathscr{B}_{h}^{S}(\cdot,\cdot):\mathbb{X}_h \times\mathbf{Y}_h \rightarrow\mathrm{R}&\qquad\mathscr{B}_{h}^{S}(\bta_h ,\vec{\bv}_h )\,:=\, \sum_{E\in\Omega_h}\mathscr{B}_{h}^{S,E}(\bta_h ,\vec{\bv}_h )\,.
\end{align}
\end{subequations}

\medskip
According to the above definitions, our virtual element scheme for \eqref{eq:v1a-eq:v4-a}-\eqref{eq:v1a-eq:v4-d}  reads as follows: 
find $ (\bsi_h, \vec{\bu}_h)\in\mathbb{X}_h\times\mathcal{Z}_h $ and $ (\brho_h, \varphi_h)\in \mathbf{X}_h\times \mathrm{Y}_h$, such that 
\begin{subequations}
\begin{align}
\mathscr{A}_{h,\varphi_h}^{S}(\bsi_h, \bta_h)+\mathscr{O}^{S}_{h,\varphi}(\bu_h;\bu_h,\bta_h)+\mathscr{B}_{h}^{S}(\bta_h ,\vec{\bu}_h)&\,=\,0  
\qquad \forall\,\bta_h\in \mathbb{X}_h \,,\label{eq:dis.ful.v1a-eq:v4-a}\\[1ex]
\mathscr{B}_{h}^{S}(\bsi_h ,\vec{\bv}_h )&\,=\, -\disp\int_{\Omega}\varphi_h\,\bg\cdot\bv_h
 \quad \forall\, \vec{v}_h\in \mathcal{Z}_h \,,\label{eq:dis.ful.v1a-eq:v4-b}\\[1ex]
\mathscr{A}^{T}_{h,\varphi_h}(\brho_h,\bet_h)+\mathscr{O}^{T}_{h,\varphi_h}(\bu_h;\varphi_h,\bet_h)+\mathscr{B}^{T}(\bet_h ,\varphi_h )&\,=\, \mathscr{F}^{T}(\bet_h)  \qquad \forall\, \bet_h\in\mathbf{X}_h \,,\label{eq:dis.ful.v1a-eq:v4-c}\\[1ex]
\mathscr{B}^{T}(\brho_h ,\psi_h )&\,=\, 0  \qquad\forall\, \psi_h\in \mathrm{Y}_h \,,\label{eq:dis.ful.v1a-eq:v4-d}
\end{align}
\end{subequations}

In what follows, we adopt the discrete version of the strategy employed in Section \ref{sec.3} to analyze the solvability of \eqref{eq:dis.ful.v1a-eq:v4-a}-\eqref{eq:dis.ful.v1a-eq:v4-d}. We now let $ \mathscr{L}^S_\mathtt{d}:  \mathbf{Y}_h \times\mathrm{Y}_h\rightarrow \mathbb{X}_h \times \mathcal{Z}_h$ be the operator given by
\begin{equation}\label{eq:dis.opr1}
\mathscr{L}^S_{\mathtt{d}}(\bz_h ,\phi_h)\, =\,(\bsi_{h,\star}\vec{\bu}_{h,\star})\, \qquad\forall\, (\bz_h ,\phi_h)\in \mathbf{Y}_h \times\mathrm{Y}_h\,,
\end{equation}
where $ (\bsi_{h,\star}\vec{\bu}_{h,\star})\in \mathbb{X}_h \times \mathcal{Z}_h $ is the unique solution of \eqref{eq:dis.ful.v1a-eq:v4-a}-\eqref{eq:dis.ful.v1a-eq:v4-b} with $ \mathscr{A}_{h,\phi_h}^{S} $ and $\mathscr{O}^{S}_{h,\phi_h}(\bz_h;\cdot,\cdot)$ instead of $ \mathscr{A}_{h,\varphi_h}^{S} $ and $\mathscr{O}^{S}_{h,\varphi_h}(\bu_h;\cdot,\cdot)$, respectively, that is,
\begin{equation}\label{eq:dis.sub_1}
\begin{array}{rcll}
\mathscr{A}_{h,\phi_h}^{S}(\bsi_{h,\star}, \bta_h)+\mathscr{O}^{S}_{h,\phi}(\bz_h;\bu_{h,\star},\bta_h)+\mathscr{B}_{h}^{S}(\bta_h ,\vec{\bu}_{h,\star})&=&0 & 
\quad \forall\,\bta_h\in \mathbb{X}_h \,,\\[1.5ex]
\mathscr{B}_{h}^{S}(\bsi_{h,\star},\vec{\bv}_h )&=& -\disp\int_{\Omega}\phi_h\,\bg\cdot\bv_h
& \quad \forall\, \vec{\bv}_h\in \mathcal{Z}_h \,.
\end{array}	
\end{equation}
												
or, equivalently, 
\begin{equation}\label{eq:dis.sub_1a}
\mathcal{A}_{h,\phi_h}^{S} \big((\bsi_{h,\star}\,,\vec{\bu}_{h,\star}), (\bta_h,\, \vec{\bv}_h)\big)+\mathscr{O}^{S}_{h,\phi_h}(\bz_h;\bu_{h,\star},\bta_h)\,=\, \mathcal{F}_{\phi_h}^{S}(\bta_h,\,\vec{\bv}_h)\quad \forall\,(\bta_h,\, \vec{\bv}_h)\in \mathbb{X}_h\times\mathcal{Z}_h\,,	
\end{equation}
where $ \mathcal{A}_{h,\phi_h}^{S} $ is the discrete version of $ \mathcal{A}_{\phi_h}^S $ (cf. \eqref{def.A}), that is, 
\begin{equation}\label{def.dis.A}
\mathcal{A}_{h,\phi_h}^{S}\big((\bze_h,\vec{\bw}_h), (\bta_h,\vec{\bv}_h)\big)\,:=\, \mathscr{A}_{h,\phi}^S(\bze_h,\bta_h)+\mathscr{B}_{h}^{S} (\bta_h ,\vec{\bw}_h )+\mathscr{B}_{h}^{S}(\bze_h ,\vec{\bv}_h )\,,
\end{equation}

In turn, we let $ \mathscr{L}^T_{\mathtt{d}}: \mathbf{Y}_h\times\mathrm{Y}_h\rightarrow \mathbf{X}_h\times\mathrm{Y}_h $ be the operator defined by
\begin{equation}\label{eq:dis.opr2}
\mathscr{L}^T_{\mathtt{d}}(\bz_h, \phi_h)\, =\, 
(\brho_{h,\star},\varphi_{h,\star}) \qquad\forall\, (\bz_h, \phi_h)\in \mathbf{Y}_h\times\mathrm{Y}_h \,,
\end{equation}
												
where $ (\brho_{h,\star},\varphi_{h,\star})\in \mathbf{X}_h\times\mathrm{Y}_h $ is the unique solution of \eqref{eq:dis.ful.v1a-eq:v4-c}-\eqref{eq:dis.ful.v1a-eq:v4-d} with $ \mathscr{A}^{T}_{h,\phi_h} $ and $\mathscr{O}^{T}_{h,\phi_h}(\bz_h;\cdot,\cdot)$ instead of $\mathscr{A}^{T}_{h,\varphi_h} $ and $\mathscr{O}^{T}_{h,\varphi_h}(\bu_h;\cdot,\cdot)$, respectively, that is
\begin{equation}\label{eq:dis.sub_2}
\begin{array}{rcll}
\mathscr{A}^{T}_{h,\phi_h}(\brho_{h,\star},\bet_h)+\mathscr{O}^{T}_{h,\phi_h}(\bz_{h};\varphi_{h,\star},\bet_h)+\mathscr{B}^{T}(\bet_h ,\varphi_{h,\star})&=& \mathscr{F}^{T}(\bet_h) & \quad \forall\, \bet_h\in\mathbf{X}_h \,,\\[1.5ex]
\mathscr{B}^{T}(\brho_{h,\star},\psi_h )&=& 0 & \quad\forall\, \psi_h\in \mathrm{Y}_h \,,														\end{array}
\end{equation}
or equivalently, 								
\begin{equation}\label{eq:dis.sub_2a}
	\mathcal{A}_{h,\phi_h}^T \big((\brho_{h,\star}\,,\varphi_{h,\star}), (\bet_h,\, \psi_h)\big)+\mathscr{O}^{T}_{h,\phi_h}(\bz_h;\varphi_{h,\star},\bet_h)\,=\, \mathscr{F}^{T}(\bet_h)\quad \forall\, (\bet_h,\psi_h)\in \mathbf{X}_h\times\mathrm{Y}_h \,,
\end{equation}
where $ \mathcal{A}_{h,\phi_h}^{T} $ is the discrete version of $ \mathcal{A}_{\phi_h}^T $ (cf. \eqref{def.At}), that is,
\begin{equation*}
\mathcal{A}_{h,\phi_h}^{T}\big((\bxi_h,\omega_h), (\bet_h,\psi_h)\big)\,:=\, \mathscr{A}^{T}_{h,\phi_h}(\bxi_h,\bet_h)+\mathscr{B}^{T} (\bet_h ,\omega_h)+\mathscr{B}^{T}(\bxi_h ,\psi_h )\,.
\end{equation*}
												
Then, similarly as in the continuous case, 
defining the operator $ \mathscr{L}_{\mathtt{d}} :\mathbf{Y}_h \times\mathrm{Y}_h\rightarrow \mathbf{Y}_h \times\mathrm{Y}_h$ as
\begin{equation}\label{eq:dis.def.opr.E}
\mathscr{L}_{\mathtt{d}}\big(\bz_h,\phi_h\big)\,:=\, \big(\mathscr{L}_{2,\mathtt{d}}^S(\bz_h,\mathscr{L}_{2,\mathtt{d}}^T(\bz_h,\phi_h)),\, \mathscr{L}_{2,\mathtt{d}}^T(\bz_h,\phi_h)\big)\quad\forall\,(\bz_h,\phi_h)\in \mathbf{Y}_h\times \mathrm{Y}_h\,,
\end{equation}
we find out that solving \eqref{eq:dis.ful.v1a-eq:v4-a}-\eqref{eq:dis.ful.v1a-eq:v4-d} is equivalent to finding a fixed
point of $ 	\mathscr{L}_{\mathtt{d}} $, that is $ (\bu_{h},\varphi_h)\in \mathbf{Y}_h \times\mathrm{Y}_h$ such that
\begin{equation}\label{eq:dis.fixp}
\mathscr{L}_{\mathtt{d}}(\bu_h, \varphi_h )\, =\,(\bu_h, \varphi_h ) \,.
\end{equation}
																		
\section{Solvability Analysis: discrete scheme}\label{sec.5}
In this section we proceed analogously to Section \ref{sec.3} and establish the well-posedness of discrete scheme  \eqref{eq:dis.ful.v1a-eq:v4-a}-\eqref{eq:dis.ful.v1a-eq:v4-d} by means of the solvability study of the equivalent fixed point equation \eqref{eq:dis.fixp}.
In this
regard, we will require
 the stability properties of the discrete forms $\mathscr{A}_{h,\phi_h}^{S}$, $\mathscr{O}_{h,\phi_h}^S(\bz_h;\cdot,\cdot)$, $\mathscr{A}^{T}_{h,\phi_h}$, and $\mathscr{O}^{T}_{h,\phi_h}(\bz_h;\cdot,\cdot)$, for each $ \phi_h\in \mathrm{Y}_h $ and $ \bz_h\in\mathbf{Y}_h $. 

\subsection{Stability properties}\label{sebsec.5.1}

We first recall from a slight
modification of \cite[eqs. (3.36) and (6.2)]{bbmr-M2AN-2016} and \cite[eq. (5.8)]{bfm-M2AN-2014} 
(see also \cite[Lemma 4.5]{cg-IMANUM-2017} and \cite[Lemma 4.1]{gms-M3AS-2018}) that under mesh Assumption $ \mathbf{(A1)} $ and \eqref{eq:bound.m.k}, there exist
positive constants $c_\star$, $c^\star$, $\widetilde c_\star$, and $\widetilde c^\star$, depending only on $C_{E}$,
such that for each $ \bw_h\in \mathbf{Y}_h $, $ \omega_h\in \mathrm{Y}_h $ and $E \in \Omega_h$, there hold
\begin{equation}\label{eq28a-eq28}
	\begin{array}{c}
		\disp
		c_\star\,\mathscr{A}_{\phi_h}^{S,E}(\bta_h,\bta_h) \,\le\,\mathscr{S}_{\phi_h}^{S,E}(\bta_h,\bta_h) \,\le\,c^\star\,\mathscr{A}_{\phi_h}^{S,E}(\bta_h,\bta_h) \qquad\forall\, \bta_h\in \mathbb X_r^{E}\cap\mathrm{Ker}(\Pcalbb_{r}^{E})\,, \\[2ex]
		\disp
		\widetilde c_\star\,\mathscr{A}^{T,E}_{\phi_h}(\bet_h,\bet_h) \,\le\,\mathscr{S}_{\phi_h}^{T,E}(\bet_h,\bet_h) \,\le\,\widetilde c^\star\,\mathscr{A}^{T,E}_{\phi_h}(\bet_h,\bet_h) \qquad\forall\, \bet_h\in \mathbf X_r^{E}\cap\mathrm{Ker}(\Pcalbf_{r}^{E})\,.
	\end{array}
\end{equation}

Then, given $ \phi_h\in \mathrm{Y}_h $, the boundedness properties of $\mathscr{A}_{h,\phi_h}^{S}$ (cf. \eqref{def-a-h-virtual}) and $\mathscr{A}^{T}_{h,\phi_h}$ (cf. \eqref{def-widetilde-a-h-virtual})
are established as follows. 
\begin{lemma}\label{l_cA-boundedness}
	Given $ \phi_h\in \mathrm{Y}_h $, 
	there exist positive constants, denoted $\mathcal{C}_{\mathscr{A}^{S}}$ and $\mathcal{C}_{\mathscr{A}^{T}}$, independent of $h$, such that
	\begin{subequations}
		\begin{align}
			\big|\mathscr{A}_{h,\phi_h}^{S}(\bze_h,\bta_h)\big| \,\le\, \mathcal{C}_{\mathscr{A}^{S}} \, 
			\Vert \bze_h\Vert_{\bdiv_{6/5},\Omega} \, \Vert \bta_h\Vert_{\bdiv_{6/5},\Omega}
			\qquad \forall\, \bze_h, \,\bta_h\in\mathbb{X}_{h}\,,\label{eq:bounda}\\[2ex]
			\big| \mathscr{A}^{T}_{h,\phi_h}(\bxi_h,\bet_h)\big| \,\le\, \mathcal{C}_{\mathscr{A}^{T}} \, 
			\Vert \bxi_h\Vert_{\div_{6/5},\Omega} \, \Vert \bet_h\Vert_{\div_{6/5},\Omega}
			\qquad \forall\, \bxi_h, \, \bet_h\in\mathbf{X}_{h}\,.\label{eq:boundb}
		\end{align}
	\end{subequations}
\end{lemma}
\begin{proof}
It reduces to a minor modification of the proof of \cite[Lemma 4.1]{gg-Pre-2023}.
\end{proof}
Now, for deriving the coerciveness properties of $\mathscr{A}_{h,\phi_h}^{S}$ and $\mathscr{A}^{T}_{h,\phi_h}$, we first look at the discrete kernel of $ \mathscr{B}_{h}^{S}|_{\mathbb{X}_h\times\mathcal{Z}_h} $ and $ \mathscr{B}^{T}|_{\mathbf{X}_h\times\mathrm{Y}_h} $, that is
\[
\mathrm{K}_h \,:=\,\Big\{ \bta_h \in \mathbb X_{h}: \quad \mathscr{B}_{h}^{S}(\bta_h ,\vec{\bv}_h ) \,=\, 0 \quad\forall\,\vec{\bv}_h \in \mathcal{Z}_h\Big\}\,,
\]
and
\[
\widetilde{\mathrm{K}}_h \,:=\,\Big\{ \bet_h \in \mathbf X_h: \quad \mathscr{B}^{T}(\bet_h ,\psi_h ) \,=\, 0 \quad\forall\,\psi_h \in \mathrm Y_h\Big\}\,.
\]

Bearing in mind the definitions of $ \mathscr{B}_{h}^{S} $ (cf. last row of \eqref{eq:def.bh}) and $ \mathscr{B}^{T} $ (cf. second row of \eqref{def-a-b-tilde-a-tilde-b}), and employing facts that $\bdiv\big(\mathbb X_{h}\big) \,\subseteq\, \mathbf Y_h$ and $\div\big(\mathbf X_{h}\big) \,\subseteq\, \mathrm Y_h$ (cf. first and second columns \eqref{eq:propA}), it follows that
\begin{equation}\label{def-mathbb-V-h-mathbf-V-h}
	\begin{array}{c}
		\disp
		\mathrm{K}_h \,:=\,\Big\{ \bta_h \in \mathbb X_{h}: \quad \bdiv(\bta_h) \,=\, \mathbf 0 \qin \Omega \qan
		\disp\int_{\Omega_h} \bta_h : \bom_h =0\quad\forall\, \bom_h \in \mathbb{Z}_h
		
		\Big\}\,, \qan\\[2.4ex]
		\widetilde{\mathrm{K}}_h \,:=\,\Big\{ \bet_h \in \mathbf X_h: \quad \div(\bet_h) \,=\, 0 \qin \Omega \Big\}\,.
	\end{array}
\end{equation}
Hence, we are in position to state the following result.
\begin{lemma}\label{coerciveness-a-h-widetilde-a-h}
Given $ \phi_h\in \mathrm{Y}_h $,	there exist positive constants $\alpha_{S,\mathtt{d}}$ and $\alpha_{T,\mathtt{d}}$, independent of $h$, such that
	\begin{equation}\label{eq-coerciveness-a-h-widetilde-a-h}
		\begin{array}{l}
			\disp
			\mathscr{A}_{h,\phi_h}^{S}(\bta_h,\bta_h) \,\ge\,\alpha_{S,\mathtt{d}}\,\|\bta_h\|^2_{\bdiv_{6/5},\Omega} \qquad\forall\,\bta_h \in \mathrm K_h \,, \\[2ex]
\mathscr{A}^{T}_{h,\phi_h}(\bet_h,\bet_h) \,\ge\,\alpha_{T,\mathtt{d}}\,\Vert\bet_h\Vert_{\div_{6/5},\Omega}^{2}\qquad\forall\,\bet_h \in \widetilde{\mathrm{K}}_h \,.
		\end{array}
	\end{equation}
\end{lemma}
\begin{proof}
Given $ \phi_{h}\in \mathrm{Y}_h $ and $E \in \Omega_h$, bearing in mind the definitions of $\mathscr{A}_{h,\phi_h}^{S}$ (cf. \eqref{eq:defAF.h}) and $\mathscr{A}_{\phi_h}^{S}$ (cf. \eqref{def-a-b-tilde-a-tilde-b}), employing the lower bounds in the first row of \eqref{eq28a-eq28} and in the first column of \eqref{eq:bound.m.k}, along with fact that $\|\bta_h\|_{0,\Omega} \ge \|\bta_h^{\tt d}\|_{0,\Omega}$, we obtain
\begin{equation}\label{eq-coerciveness-a-K-h-1}
\begin{array}{c}
\disp
\mathscr{A}_{h,\phi_h}^{S,E}(\bta_h,\bta_h) \,\ge\, \,\dfrac{1}{\mu_1}\,\Big( \big\|\big(\Pcalbb^E_r(\bta_h)\big)^{\tt d}\big\|^2_{0,E}
\,+\, c_\star\,\big\|\bta_h - \Pcalbb^E_r(\bta_h)\big\|^2_{0,E}\Big) \\[2ex]
\disp\qquad
\ge\, \dfrac{1}{\mu_1}\,\Big( \big\|\big(\Pcalbb^E_r(\bta_h)\big)^{\tt d}\big\|^2_{0,E}
\,+\, c_\star\,\big\|\bta_h^{\tt d} - \big(\Pcalbb^E_r(\bta_h)\big)^{\tt d}\big\|^2_{0,E}\Big) \\[2ex]
\disp\qquad
\ge\, \dfrac{1}{2 \mu_1}\,\min\big\{1,c_\star\big\}\, \big\|\bta_h^{\tt d}\big\|^2_{0,E} \,,
\end{array}
\end{equation}
where the inequality $b^2 + (a-b)^2 \,\ge\, \frac12\, a^2$ was used in the last step.
	In this way, summing 
	over all $E \in \Omega_h$ in \eqref{eq-coerciveness-a-K-h-1}, and using \eqref{eq:div.ta}, we conclud the first inequality of \eqref{eq-coerciveness-a-h-widetilde-a-h}
	with setting $ \alpha_{S,\mathtt{d}} \,:=\, \dfrac{c_{\Omega}}{2 \mu_1}\,\min\big\{1,c_\star\big\} $.
	In turn, the proof of second one follows a similar approach, utilizing the definitions of $\mathscr{A}^{T,E}_{h,\phi_h}$ (cf. \eqref{eq:defAT.h}),
		$\mathscr{A}^{T,E}$ (cf. \eqref{def-a-b-tilde-a-tilde-b}), and $\mathscr{A}_{h,\phi_h}^{T}$ (cf. \eqref{def-widetilde-a-h-virtual}), and the lower bounds in the second row of \eqref{eq28a-eq28} and the second column of \eqref{eq:bound.m.k}. Further details are omitted.
\end{proof}
Next, for each $\phi_h\in \mathrm{Y}_h$ and $ \bz_h\in \mathbf{Y}_h $ we establish the boundedness property of $\mathscr{O}^{S}_{h,\phi_h}(\bz_h;\cdot,\cdot)$ and
 $\mathscr{O}^{T}_{h,\phi_h}(\bz_h;\cdot,\cdot)$, which 
coincide with those of $\mathscr{O}_{\phi_h}^S(\bz_h;\cdot,\cdot)$ and
 $\mathscr{O}^{T}_{\phi_h}(\bz_h;\cdot,\cdot)$ (cf. third and last rows of \eqref{eq:bound}), respectively, as
stated next.
\begin{lemma}\label{c-widetilde-c-boundedness}
Given $ \phi_h\in \mathbf{Y}_h $,	denoting $$\mathcal{C}_{\mathscr{O}^{S}} := \frac{|\Omega|^{1/6}}{\mu_1}\,,\qan \mathcal{C}_{\mathscr{O}^{T}} := \frac{|\Omega|^{1/6}}{\kappa_1}\,,$$ there hold 
	\begin{equation}\label{eq:boundCF}
\begin{array}{c}
			\big| \mathscr{O}^{S}_{h,\phi_h}(\bz_h;\bw_h,\bta_h)\big| \,\leq\,\mathcal{C}_{\mathscr{O}^{S}}  \,\Vert \bz_h\Vert_{0,6;\Omega}\,\Vert \bw_h\Vert_{0,6;\Omega}\, \Vert \bta_h\Vert_{0,\Omega}
			\quad\forall\, (\bz_h , \bw_h  , \bta_h ) \in  \mathbf{Y}_h \times \mathbf{Y}_h \times \mathbb{X}_h \,, \\[2ex]
\big| \mathscr{O}^{T}_{h,\phi_h}(\bz_h;\omega_h,\bet_h)\big| \,\leq\, \mathcal{C}_{\mathscr{O}^{T}}  \,\Vert \bz_h\Vert_{0,6;\Omega}\,\Vert \omega_h\Vert_{0,6;\Omega}\, \Vert \bet_h\Vert_{0,\Omega}
			\quad\forall\, (\bz_h , \omega_h  , \bet_h ) \in  \mathbf{Y}_h \times \mathrm{Y}_h \times \mathbf{X}_h \,.
\end{array}
	\end{equation}
\end{lemma}
\begin{proof}
It is a slight modification of the proof of \cite[Lemma 4.4]{gg-Pre-2023}. We omit further details.
\end{proof}

\subsubsection{Discrete inf-sup condition}
Here, we focus on deriving the discrete inf-sup condition for $\mathscr{B}_{h}^{S}|_{\mathbb{X}_h \times\mathcal{Z}_h }$.
In order to do that, first we recall the abstract result established in (cf. \cite[Theorem 3.1]{hw-NM-2011}, and \cite[Exercise 5.6.4]{D-SIAM-2023}), as essential tool to
be applied to our aim abovementioned.
\begin{lemma}\label{l_recal1}
Let $ u\in U $ be a Hilbert space and $v=(v_1 , v_2)\in V_1 \times V_2$ be a group variable, where $V_1$ and $V_2$ are Hilbert spaces. Consider a composite bilinear form
\[
b(u,v)\,:=\, b_1(u , v_1)+b_{2}(u , v_2) \,,
\]
Define the kernel space
\[
U_{1} \,:=\, \Big\{ u\in U: \quad b_{1}(u , v_1)\,=\,0\quad \forall\, v_1 \in V_1 \Big\}\,,
\]

and assume the following inf-sup conditions
\begin{equation}\label{eq:inf1}
\sup_{u\in U}\dfrac{b_{1}(u , v_1)}{\Vert u\Vert_{U}}\,\geq\, \beta_1 \, \Vert v_1 \Vert_{V_1}\quad \forall\, v_1 \in V_1 \,,
\end{equation}
and
\begin{equation}\label{eq:inf2}
\sup_{u_1 \in U_1}\dfrac{b_2 (u_1 , v_2)}{\Vert u_1 \Vert_{U}}\, \geq\, \beta_2 \, \Vert v_2 \Vert_{V_2}\quad \forall\, v_2 \in V_2 \,,
\end{equation}

then, there exists a constant $\beta$, depending on $\beta_1$, $\beta_2$, $\| b_1\|$ and $\| b_2\|$, such that
\begin{equation}\label{eq:inf3}
\sup_{u\in U}\dfrac{b(u , v)}{\Vert u\Vert_{U}}\,\geq\, \beta \, \Vert v \Vert_{V}\quad \forall\, v \in V \,.
\end{equation}
\end{lemma}
\begin{proof}
Given $(u, v) \in U \times V$, we note that inf-sup conditions \eqref{eq:inf1} and \eqref{eq:inf2} guarantee the existence of $\widetilde{u} \in U$ and $\widehat{u} \in U_{1}$, such that there hold
\begin{equation}\label{eq:o1}
	\|\widetilde{u}\|_{U}=1 \quad \text { and } \quad b_{1}\left(\widetilde{u},v_{1}\right) \geq \frac{\beta_{1}}{2}\left\|v_{1}\right\|_{V}\,,
\end{equation}
and
\begin{equation}\label{eq:o2}
	\|\widehat{u}\|_{U}=1 \quad \text { and } \quad b_{2}\left(\widehat{u},v_{2}\right) \geq \frac{\beta_{2}}{2}\left\|v_{2}\right\|_{V}\,.
\end{equation}
Now, defining $\bar{u}:=c_{1} \widetilde{u}+c_{2} \widehat{u}$, with positive constants $c_{1}$ and $c_{2}$ to be chosen later on, which yields $\|\bar{u}\|_{U} \leq c_{1}+c_{2}$, and using that $b_{1}\left(\widehat{u},v_{1}\right)=0$ along with \eqref{eq:o1} and \eqref{eq:o2} we have

$$
\begin{gathered}
	\sup _{u \in U} \frac{b(u, v)}{\|u\|_{U}} \geq \frac{|b(\bar{u},v)|}{\|\bar{u}\|_{U}}=\frac{\left|b_{1}\left(\bar{u},v_{1}\right)+b_{2}\left( \bar{u},v_{2}\right)\right|}{\|\bar{u}\|_{U}} \\
	=\frac{\left|c_{1} b_{1}\left(\widetilde{u},v_{1}\right)+c_{2} b_{2}\left(\widehat{u},v_{2}\right)+c_{1} b_{2}\left( \widetilde{u},v_{2}\right)\right|}{\|\bar{u}\|_{U}} \\
	\geq \frac{1}{c_{1}+c_{2}}\left(\frac{c_{1} \beta_{1}}{2}\left\|v_{1}\right\|_{V}+\left(\frac{c_{2} \beta_{2}}{2}-c_{1}\left\|b_{2}\right\|\right)\left\|v_{2}\right\|_{V}\right)\,,
\end{gathered}
$$

from which, choosing $c_{1}$ and $c_{2}$ such that $c_{1} \geq 0$ and $\frac{c_{2} \beta_{2}}{2} \geq c_{1}\left\|b_{2}\right\|$ we arrive at \eqref{eq:inf3} with considering

$$
\beta:=\frac{1}{c_{1}+c_{2}} \min \left\{\frac{c_{1} \beta_{1}}{2},\left(\frac{c_{2} \beta_{2}}{2}-c_{1}\left\|b_{2}\right\|\right)\right\}\,,
$$

and finish the proof.
\end{proof}
Now, we are ready to utilize the results above for establishing discrete inf-sup condition for $\mathscr{B}_{h}^{S} |_{\mathbb{X}_{h}\times\mathcal{Z}_h}$. We begin with define spaces $ U $, $ V_1 $, and $ V_2 $ in Lemma \ref{l_recal1} by
\begin{equation}\label{eq:space.uvw}
U \,:=\,\mathbb{X}_{h} \,, \qquad	V_1 \,:=\,\mathbf{Y}_{h}\,,\qan V_2  \,:=\,\mathbb{Z}_{h}\,,
\end{equation}
and letting
\begin{equation*}
 U_{1}\:=\, \Big\{\bta_h\in \mathbb{X}_h :\quad \mathscr{B}_{h}^{S}(\bta_h ,(\bv_{h},0))\,=\, 0 \quad\forall\, \bv_{h}\in \mathbf{Y}_h  \Big\}\,.
\end{equation*}
In this way, considering the above notations, and according to Lemma \ref{l_recal1}, in order to establish the discrete inf-sup condition for $\mathscr{B}_{h}^{S} |_{\mathbb{X}_{h}\times\mathcal{Z}_h}$, we need to show the discrete inf-sup conditions of $\mathscr{B}_{h}^{S} |_{\mathbb{X}_{h}\times\mathbf{Y}_h}$ and $ \mathscr{B}_{h}^{S}|_{U_{1}\times\mathbb{Z}_h} $. We note that due to the definition of $ \mathscr{B}_{h}^{S} $ (cf. \eqref{eq:def.bh}) and fact that $\bdiv\big(\mathbb X_{h}\big) \,\subseteq\, \mathbf Y_h$ (cf. first column of \eqref{eq:propA}), space $ U_1 $ reduce to
\begin{equation}\label{eq:def.V1}
	 U_{1}\:=\, \Big\{\bta_h\in \mathbb{X}_h :\quad \bdiv(\bta_h)\,=\, \mathbf{0}\qin\Omega\, \Big\}\,.
\end{equation}

The discrete inf-sup condition for the discrete bilinear form  $\mathscr{B}_{h}^{S} |_{\mathbb{X}_{h}\times\mathbf{Y}_h}$, which is a slight modification of the proof of \cite[Section 4.2, Lemma 4.9]{gs-M3AS-2021}, established by the following lemma.
\begin{lemma}\label{l_h1.inf}
	There exists the positive constant $ \beta_{1,\mathtt{d}} $, independent of $ h $, such that
	\begin{equation}\label{eq:inf.bh.c2}
		\sup_{\mathbf{0}\neq \bta_h\in \mathbb{X}_{h}}\dfrac{\mathscr{B}_{h}^{S}(\bta_h ,(\bv_{h},0))}{\Vert\bta_h\Vert_{\bdiv_{6/5},\Omega}}\,\geq\,\beta_{1,\mathtt{d}}\,\Vert\bv_{h}\Vert_{0,6;\Omega}\qquad\forall\, \bv_{h}\in \mathbf{Y}_{h}\,.
	\end{equation}
\end{lemma}
\begin{proof}
	The proof follows from a direct application of \cite[Lemma 4.9]{gs-M3AS-2021}.
\end{proof}
The following result provides  the discrete inf-sup condition of $ \mathscr{B}_{h}^{S}|_{U_{1}\times\mathbb{Z}_h} $, where $ U_{1} $ is given by \eqref{eq:def.V1}.
\begin{lemma}\label{l_h2.inf}
	There exists the positive constant $ \beta_{2,\mathtt{d}} $, independent of $ h $, such that
	\begin{equation}\label{eq:dis.inf1}
		\sup_{\mathbf{0}\neq \bta_h\in U_1}\dfrac{\mathscr{B}_{h}^{S}(\bta_h ,(0,\bom_h))}{\Vert \bta_h\Vert_{\bdiv_{6/5},\Omega}}\,\geq\,\beta_{2,\mathtt{d}}\,\Vert \bom_h\Vert_{0,\Omega}\,\qquad\forall\, \bom_h\in \mathbb{Z}_{h}\,.
	\end{equation}
\end{lemma}
\begin{proof}
	Given $ g\in (L_{0}^{2}(\Omega))' $, we consider $ (\mathbf{z},p)\in \mathbf{H}_{0}^{1}(\Omega)\times L_{0}^{2}(\Omega) $ as the solution of the variational formulation arising from stationary Stokes
	problem, that is,
	\begin{equation}\label{eq:var.St}
		\begin{split}
			\int_{\Omega}\nabla \textbf{z}:\nabla \textbf{w} +\int_{\Omega}p \,\div(\mathbf{w})&\,=\,\mathbf{0}\qquad\forall\, \mathbf{w}\in \mathbf{H}_{0}^{1}(\Omega)\,,\\
			\int_{\Omega}q \,\div(\mathbf{z})&\,=\,g(q)\qquad\forall\, q\in L_{0}^{2}(\Omega)\,.
		\end{split}
	\end{equation}
	Now, considering  $ k\geq 2 $, defining
	\[
	\mathbf{B}_{k}(\partial K)\,:=\,\Big\{\bw\in \mathbf{C}^{0}(K):\quad \bw|_{e}\in \mathbf{P}_{k}(e) \quad\forall\, e\subseteq\partial K  \Big\}\,,
	\]
	and letting $ X_{h} $ and $ \widehat{Y}_{h} $ be the virtual element subspaces of $ \mathbf{H}_{0}^{1}(\Omega) $ and $ L^{2}(\Omega) $, which are defined by (see \cite[eqs.  (11) and (12)]{blv-MMNA-2017})
	\begin{equation}\label{eq:def.Xh}
		\begin{array}{ll}
			X_{h}|_{K}\,:=\,\bigg\{ \bw \in \mathbf{H}_{0}^{1}(K):&\quad \bw\mid_{ \partial K} \in\mathbf{B}_{k}(\partial K)\,,\quad 
			\div (\bw) \in \mathrm{P}_{k-1}(K)\,\qan\\
			&- \boldsymbol{\Delta} \bw-\nabla v \in {\mathbf P}_{k-2,\nabla}(K)\quad \forall \, v \in \mathrm{L}^{2}(K)\bigg\}\,,
		\end{array}
	\end{equation}
	and 
	\begin{equation*}
		\widehat{Y}_{h}|_{K}\,:=\, \mathrm{P}_{k-1}(K)\,,
	\end{equation*}
	
	we find that $ (\mathbf{z}_h,p_h)\in X_{h}\times Y_{h} $, where $ Y_{h}\,=\,\widehat{Y}_{h}\cap L_0^{2}(\Omega) $, is the approximation virtual element solution for problem \eqref{eq:var.St} (see, e.g., \cite[Theorem 4.4]{blv-MMNA-2017}), that is, find $ (\mathbf{z}_h,p_h)\in X_{h}\times Y_{h} $ such that
	\begin{equation}\label{eq:vem.St}
		\begin{split}
			\mathscr{D}^{S}_{h}(\mathbf{z}_h, \mathbf{w}_h)+\int_{\Omega}p_h\div(\mathbf{w}_h)&\,=\,0\qquad\forall\, \mathbf{w}_h\in X_{h}\,,\\
			\int_{\Omega}q_h\div(\mathbf{z}_h)&\,=\,g(q_h)\qquad\forall\, q_h\in Y_{h}\,,
		\end{split}
	\end{equation}
	where the discrete bilinear form $ \mathscr{D}^{S}_{h}: X_{h}\times X_{h}\rightarrow\mathrm{R} $ satisfy the $ k $-consistency property (cf. \cite[equation (3.24)]{blv-MMNA-2017}) given by 
	\begin{equation}\label{eq:k.consi}
		\mathscr{D}^{S}_{h}(\mathbf{z}_h, \mathbf{q})\,=\,\int_{\Omega}\boldsymbol{\nabla} \mathbf{z}_{h}: \boldsymbol{\nabla} \mathbf{q}\qquad \forall\, \mathbf{q}\in\mathbf{P}_{k}(\Omega_h)\,.
	\end{equation} 
	Moreover, there hold the following a priori estimate (cf. \cite[Theorem 4.4]{blv-MMNA-2017})
	\begin{equation}\label{eq:stab.est}
		\Vert\mathbf{z}_h\Vert_{1,\Omega}+\Vert p_h\Vert_{0,\Omega}\,\leq\, \mathcal{C}_{\mathtt{st}}\, \Vert g\Vert\,,
	\end{equation}
	where $ \mathcal{C}_{\mathtt{st}} $ is a positive constant independent of $ h $. Next, proceeding analogously to \cite[Subsection 5.2]{cgm-ESAIM-2020}, we consider $ g $ as the functional induced by a given $ \widehat{q}_h\in \widehat{Y}_{h} $, that is, 
	\[
	g(q_h)\,:=\,\int_{\Omega} \widehat{q}_h \, q_h\qquad \forall\, q_h\in Y_{h}\,.
	\]
	In this way, knowing that $ \mathbf{P}_{k}(\Omega_h)\subseteq X_{h} $ for each $ k\geq 2 $ (cf. \cite[Section 3.1]{blv-MMNA-2017}), we take $ \mathbf{w}_{h}\,:=\,(-x_2,x_1)^{\mathtt{t}} $ in the first equation of \eqref{eq:vem.St} and then employ  \eqref{eq:k.consi}, to arrive at
	\begin{equation}\label{eq:nabla.z}
		\int_{\Omega} \boldsymbol{\nabla}\mathbf{z}_{h}:\left(\begin{matrix}
			0 & -1\\1 & 0
		\end{matrix}\right)\,=\,\int_{\Omega} (-\partial_{x_2} z_{h,1}+\partial_{x_1} z_{h,2})\, =\, 0\,.
	\end{equation}
	Now, we define the discrete tensor as
	\begin{equation}\label{eq:def.tau.dis}
		\widetilde{\bta}_{h}\,:=\,\operatorname{\textbf{curl}}(\mathbf{z}_h)+\hat{c}_{h}\left(\begin{matrix}
			0 & 1\\ 0 &0
		\end{matrix}\right)\,,
	\end{equation}
	where the components $ \operatorname{\textbf{curl}}(\mathbf{z}_h) $ and $ \hat{c}_h $ are given by 
	\[
	\operatorname{\textbf{curl}}(\mathbf{z}_h)\,:=\,\left(\begin{matrix}
		-\partial_{x_2} z_{h,1} & \partial_{x_1} z_{h,1}\\[1mm]
		-\partial_{x_2} z_{h,2} & \partial_{x_1} z_{h,2}
	\end{matrix}\right)\qan \hat{c}_{h}\,:=\,\dfrac{1}{|\Omega|}\int_{\Omega}\widehat{q}_{h}\,.
	\]
	
	Our goal now is to show that $ \widetilde{\bta}_{h} $ belongs to $ U_1 $. To this end, we first apply the rot operator to both sides of \eqref{eq:def.tau.dis}, which gives
	\begin{equation}\label{eq:ga}
		\operatorname{\textbf{rot}}(
		\widetilde{\bta}_{h})\,=\, \operatorname{\textbf{rot}}(\operatorname{\textbf{curl}}(\mathbf{z}_h))\,=\,\boldsymbol{\Delta}\mathbf{z}_{h}\,.
	\end{equation}
Bearing in mind the definition of $ X_h $ (cf. \eqref{eq:def.Xh}) and knowing that $ \mathbf{z}_{h}\in X_h $, we readily see that	
	\begin{equation}\label{eq:g0}
		(\boldsymbol{\Delta}\mathbf{z}_{h}+\nabla v)|_K\in {\mathbf P}_{k-2,\nabla}(K)\subseteq \mathbf{P}_{k-2}(K) \qquad \forall \, v \in L^{2}(K)\,.
	\end{equation}
	In addition, thanks to definition of $ Y_h $ and knowing that $ Y_{h}|_{K}\subseteq L^{2}(K) $, we take $ v\in P_0(K) $, and observe from \eqref{eq:g0} that $ \boldsymbol{\Delta}\mathbf{z}_{h}|_{K}\in \mathbf{P}_{k-2}(K) $, for each $ k\geq 2 $, from which and \eqref{eq:ga} conclude that  
	\begin{equation}\label{eq:e0}
		\operatorname{\textbf{rot}}(\widetilde{\bta}_{h})|_{K}\in \mathbf{P}_{r-1}(K)\,\quad\forall\, r\in \mathrm{N} \,.
	\end{equation}
	In particular, according to the definition of $ X_h $ again (first sentence in \eqref{eq:def.Xh}), it is readily seen that $ (\widetilde{\bta}_{h}\mathbf{n})|_e \in \mathbf{P}_{k-1}(e) $ for any $ e\subseteq\partial K $ and $ k\geq 2 $, or equivalently
	\begin{equation}\label{eq:e1}
		(\widetilde{\bta}_{h}\mathbf{n})|_e \in \mathbf{P}_{r}(e)\qquad\forall\,e\subseteq\partial K \,\quad \forall\, r\geq 1\,.
	\end{equation}
	On the other hand, knowing the curl of a vector has zero divergence, clearly $ \widetilde{\bta}_{h} $ is divergence free, as well implies that  
	\begin{equation}\label{eq:e2}
		\bdiv(\widetilde{\bta}_{h})|_{K}\in \mathbf{P}_{r}(K)\,\quad\forall\, r\in \mathrm{N} \,.
	\end{equation}
	Then, noting from \eqref{eq:def.tau.dis} and using the identity
	in \eqref{eq:nabla.z}, we obtain
	\begin{equation}\label{eq:e3}
		\begin{split}
			\int_{\Omega} \operatorname{tr}(\widetilde{\bta}_{h})\, &= \, 	\int_{\Omega} \widetilde{\bta}_{h}:\mathbb{I}\,\\[1mm]
			& =\, \int_{\Omega} \operatorname{\textbf{curl}}(\mathbf{z}_h): \mathbb{I}\, + \hat{c}_h \int_{\Omega}\left(\begin{matrix}
				0 & 1\\ 0 &0
			\end{matrix}\right): \mathbb{I}\\[1mm]
			\,&=\,\int_{\Omega} (-\partial_{x_2} z_{h,1}+\partial_{x_1} z_{h,2})\, =\, 0\,.
		\end{split}
	\end{equation}
	Finally, bearing in mind the definition of $ U_1 $ (cf. \eqref{eq:def.V1}) and a straightforward application of \eqref{eq:e0}-\eqref{eq:e2} allow to conclude that $ \widetilde{\bta}_{h}\in U_1 $. Hence, considering the particular choice $$ \bom_h\,:=\,\left(\begin{matrix}
		0 & \widehat{q}_h \\ -\widehat{q}_h & 0
	\end{matrix}\right)\in \mathbb{Y}_{h,\mathtt{skw}}\,, $$ 
	employing \eqref{eq:def.tau.dis}, the fact that $ \mathbf{z}_h $ vanishes on boundary, and the second equation of \eqref{eq:vem.St} by taking $ q_h \,:=\,(\widehat{q}_h-\hat{c}_h)\in Y_{h} $, we find that
	\begin{equation}\label{eq:er1}
		\begin{split}
			\int_{\Omega}\widetilde{\bta}_{h}:\bom_h &\,=\,\int_{\Omega}\,\left(\begin{matrix}
				-\partial_{x_2} z_{h,1} & \partial_{x_1} z_{h,1}\\[1mm]
				-\partial_{x_2} z_{h,2} & \partial_{x_1} z_{h,2}
			\end{matrix}\right):\left(\begin{matrix}
				0 & \widehat{q}_h \\ -\widehat{q}_h & 0
			\end{matrix}\right)+\hat{c}_{h}\int_{\Omega}\left(\begin{matrix}
				0 & 1\\ 0 &0
			\end{matrix}\right):\left(\begin{matrix}
				0 & \widehat{q}_h \\ -\widehat{q}_h & 0
			\end{matrix}\right)\\[1mm]	
			&\,=\,\int_{\Omega} \left(\widehat{q}_{h}\div(\bz_h)+\widehat{q}_{h}\hat{c}_h\right)\\[1mm]
			&~=~\int_{\Omega} \left((\widehat{q}_{h}-\hat{c}_h)\div(\mathbf{z}_h)+\widehat{q}_{h}\hat{c}_h\right)\\
			&~=~\int_{\Omega} \left(\widehat{q}_h(\widehat{q}_{h}-c_h)+\widehat{q}_{h}c_h\right)~=~\Vert \widehat{q}_h\Vert_{0,\Omega}^{2}\,.
		\end{split}
	\end{equation} 
	Moreover, an application of fact the divergence free of $ \widetilde{\bta}_h $, definitions of $ \widetilde{\bta}_h $ and $ \hat{c}_h $, and the stability estimate \eqref{eq:stab.est}, yields
	\begin{equation*}
		\Vert \widetilde{\bta}_h\Vert_{\bdiv_{6/5},\Omega}\,=\,\Vert \widetilde{\bta}_h\Vert_{0,\Omega}\,\leq\, \vert\mathbf{z}_h\vert_{1,\Omega}+\Vert \hat{c}_h\Vert_{0,\Omega}\,\leq\, \widetilde{\mathcal{C}}_{\mathtt{st}}\Vert\widehat{q}_h\Vert_{0,\Omega}\,,
	\end{equation*}
	with $ \widetilde{\mathcal{C}}_{\mathtt{st}} $ depending on $ \mathcal{C}_{\mathtt{st}} $ and $ \Omega $. This result jointly with \eqref{eq:er1}, implies
	\begin{equation*}
		\sup_{\textbf{0}\neq \bta_h\in U_1}\dfrac{\mathscr{B}_{h}^{S}( \bta_h ,(0,\bom_h))}{\Vert \bta_h\Vert_{\bdiv_{6/5},\Omega}}\,\geq\,	\dfrac{\mathscr{B}_{h}^{S}(\widetilde{\bta}_h ,(0,\bom_h))}{\Vert \widetilde{\bta}_h\Vert_{\bdiv_{6/5},\Omega}}
		\, = \, \dfrac{\disp\int_{\Omega}\widetilde{\bta}_{h}:\bom_h}{\Vert \widetilde{\bta}_h\Vert_{\bdiv_{6/5}}}	
		\,\geq\, \dfrac{1}{\widetilde{\mathcal{C}}_{\mathtt{st}}}\Vert\widehat{q}_h\Vert_{0,\Omega}\,,
	\end{equation*}
	which, along with fact that $ \Vert\bom_h\Vert_{0,\Omega}^{2}\,=\,2\Vert\widehat{q}_h\Vert_{0,\Omega}^{2} $ concludes \eqref{eq:dis.inf1} by considering $ \beta_{2,\mathtt{d}}\,:=\,1/\sqrt{2}\widetilde{\mathcal{C}}_{\mathtt{st}} $.
\end{proof}
The following result establishes the discrete version of \eqref{eq:inf.bi}-\eqref{eq:inf.bit}.
\begin{lemma}\label{l_dis.inf.b.bt}
There exist a positive constants $ \beta_{S,\mathtt{d}} $ and $ \beta_{T,\mathtt{d}} $, independent of $ h $, such that
\begin{equation}\label{eq:dis.inf.1}
	\sup_{\mathbf{0}\neq \bta_h\in \mathbb{X}_{h}}\dfrac{\mathscr{B}_{h}^{S}(\bta_h ,\vec{\bv}_{h})}{\Vert\bta_h\Vert_{\bdiv_{6/5},\Omega}}\,\geq\,\beta_{S,\mathtt{d}}\,\Vert\vec{\bv}_{h}\Vert\qquad\forall\, \vec{\bv}_{h}\in \mathcal{Z}_{h}\,.
\end{equation}
and
\begin{equation}\label{eq:dis.inf.2}
	\sup_{\mathbf{0}\neq \bet_h\in \mathbf{X}_{h}}\dfrac{\mathscr{B}^{T}(\bet_h ,\psi_{h})}{\Vert\bet_h\Vert_{\div_{6/5},\Omega}}\,\geq\,\beta_{T,\mathtt{d}}\,\Vert\psi_{h}\Vert\qquad\forall\, \psi_{h}\in \mathrm{Y}_{h}\,.
\end{equation}
\end{lemma}
\begin{proof}
The proof of \eqref{eq:dis.inf.1} is a direct consequence of Lemmas \ref{l_recal1}-\ref{l_h2.inf}, whereas second one, being actually a scalar version of \eqref{eq:inf.bh.c2} (cf. Lemma \ref{l_h1.inf}), follows by similar arguments.
\end{proof}
\subsection{Well-definedness of the operators \texorpdfstring{$ \mathscr{L}^S_{\mathtt{d}} $}{Lg} and \texorpdfstring{$ \mathscr{L}_{\mathtt{d}}^T $}{Lg}}

We begin by addressing the well-definedness of $ \mathscr{L}^S_{\mathtt{d}} $ and $ \mathscr{L}^T_{\mathtt{d}} $,  equivalently, the discrete schemes \eqref{eq:dis.sub_1a} and \eqref{eq:dis.sub_2a}, respectively, with following similar processes to Section \ref{sec.3}.
 Indeed, using the stability properties of discrete bilinear froms given in Section \ref{sebsec.5.1} (see Lemmas \ref{l_cA-boundedness}, \ref{coerciveness-a-h-widetilde-a-h}, \ref{l_dis.inf.b.bt}), and applying
 \cite[Proposition 2.36]{eg-SPRINGER-2004} it is not difficult to see that for each $ \phi_h\in \mathrm{Y}_h $ the discrete bilinear forms $ \mathcal{A}_{h,\phi_{h}}^S  $ and $ \mathcal{A}_{h,\phi_{h}}^T  $ satisfy global discrete inf-sup conditions on $ \mathbb{X}_h\times\mathcal{Z}_h $ and $ \mathbf{X}_h \times\mathrm{Y}_h $, respectively. More precisely, there exist positive constants $ \alpha_{\mathcal{A}^S,\mathtt{d}} $ and $ \alpha_{\mathcal{A}^T,\mathtt{d}} $, depending on $ \mathcal{C}_{\mathscr{A}^S}$, $ ||\mathscr{B}^S|| $, $ \alpha_{S,\mathtt{d}} $, $ \beta_{S,\mathtt{d}} $, and $ \mathcal{C}_{\mathscr{A}^T}$, $ ||\mathscr{B}^{T}|| $, $ \alpha_{T,\mathtt{d}} $, $ \beta_{T,\mathtt{d}} $, respectively, and hence independent of $ h $, such that
	\begin{equation}\label{eq:dis.inf.A}
	\sup\limits_{\mathbf{0}\neq (\bta_h,\vec{\bv}_h)\in \mathbb{X}_h\times\mathcal{Z}_h}\dfrac{\mathcal{A}_{h,\phi_{h}}^S 
		\big((\bze_h, \vec{\bw}_h), (\bta_h,\vec{\bv}_h)\big)}{\Vert (\bta_h,\vec{\bv}_h)\Vert}\,\geq\,\alpha_{\mathcal{A}^S,\mathtt{d}}\,\Vert(\bze_h, \vec{\bw}_h)\Vert\qquad\forall\, (\bze_h, \vec{\bw}_h)\in \mathbb{X}_h\times\mathcal{Z}_h \,,
\end{equation}
and
\begin{equation}\label{eq:dis.inf.At}
	\sup\limits_{\mathbf{0}\neq (\bet_h, \psi_h)\in \mathbf{X}_h\times\mathrm{Y}_h}\dfrac{\mathcal{A}^{T}_{h,\phi_{h}}\big((\bxi_h, \omega_h), (\bet_h, \psi_h)\big)}{\Vert (\bet_h, \psi_h)\Vert}\,\geq\,\alpha_{\mathcal{A}^{T},\mathtt{d}}\,\Vert(\bxi_h, \omega_h)\Vert\qquad\forall\, (\bxi_h, \omega_h)\in \mathbf{X}_h\times\mathrm{Y}_h \,,
\end{equation}
Thus, for each $ \phi_{h}\in \mathrm{Y}_h $, using the boundedness property of $ \mathscr{O}^{S}_{h,\phi_h} $ and $ \mathscr{O}^{T}_{h,\phi_h} $ (cf. first and second rows of \eqref{eq:boundCF}), similarly as for the derivations \eqref{eq:inf.A.1} and \eqref{eq:inf.At.1}, we deduce that under the assumptions 
\begin{equation}\label{eq:dis.as.z1}
	\Vert\bz_h\Vert_{0,6;\Omega}\, \leq \, r_{1,\mathtt{d}}\quad \text{with}\quad r_{1,\mathtt{d}} \,:=\, \dfrac{\mu_1\alpha_{\mathcal{A}^S,\mathtt{d}}}{2|\Omega|^{1/6}}\,,
\end{equation}
and
\begin{equation}\label{eq:dis.as.z2}
	\Vert\bz_h\Vert_{0,6;\Omega}\, \leq \, r_{2,\mathtt{d}} \quad\text{with}\quad r_{2,\mathtt{d}}\,:=\, \dfrac{\kappa_1\alpha_{\mathcal{A}^{T},\mathtt{d}}}{2|\Omega|^{1/6}}\,,
\end{equation}
there holds
	\begin{equation}\label{eq:dis.inf.A.o}
	\sup\limits_{\mathbf{0}\neq (\bta_h,\vec{\bv}_h)\in \mathbb{X}_h\times\mathcal{Z}_h}\dfrac{\mathcal{A}_{h,\phi_{h}}^S 
		\big((\bze_h, \vec{\bw}_h), (\bta_h,\vec{\bv}_h)\big)+\mathscr{O}^{S}_{h,\phi_{h}}(\bz_{h};\bw_h,\bta_{h})}{\Vert (\bta_h,\vec{\bv}_h)\Vert}\,\geq\,\dfrac{\alpha_{\mathcal{A}^S,\mathtt{d}}}{2}\,\Vert(\bze_h, \vec{\bw}_h)\Vert  \,,
\end{equation}
for all $ (\bze_h, \vec{\bw}_h)\in \mathbb{X}_h\times\mathcal{Z}_h $
and
\begin{equation}\label{eq:dis.inf.At.ot}
	\sup\limits_{\mathbf{0}\neq (\bet_h, \psi_h)\in \mathbf{X}_h\times\mathrm{Y}_h}\dfrac{\mathcal{A}^{T}_{h,\phi_{h}}\big((\bxi_h, \omega_h), (\bet_h, \psi_h)\big)+\mathscr{O}^{T}_{h,\phi_{h}}(\bz_{h};\omega_h,\bet_h)}{\Vert (\bet_h, \psi_h)\Vert}\,\geq\,\dfrac{\alpha_{\mathcal{A}^{T},\mathtt{d}}}{2}\,\Vert(\bxi_h, \omega_h)\Vert \,,
\end{equation}
for all $ (\bxi_h, \omega_h)\in \mathbf{X}_h\times\mathrm{Y}_h $.
According to the above and the fact that right-hand of \eqref{eq:dis.sub_1a} and \eqref{eq:dis.sub_2a} are a linear and bounded functional,
a straightforward application of the Banach-Nečas-Babuška Theorem allows to conclude the following result, which is the discrete analogue of Lemma \ref{l_wel.S.T}.
\begin{lemma}\label{l_dis.wel.S.St}
	For each $\big(\mathbf{z}_h,\phi_h\big) \in\mathbf{Y}_{h} \times \mathrm Y_h$, satisfying \eqref{eq:dis.as.z1} and \eqref{eq:dis.as.z2}
	there exist unique $(\bsi_{h,\star},\vec{\bu}_{h,\star})\in\mathbb{X}_{h}\times\mathcal{Z}_h$ and
	$(\brho_{h,\star},\varphi_{h,\star})\in\mathbf{X}_h \times\mathrm{Y}_h$ solutions to \eqref{eq:dis.sub_1a} and \eqref{eq:dis.sub_2a}, 
	respectively, so that one can define 
	\[
	\mathscr{L}_{\mathtt{d}}^S\big(\bz_h,\phi_h\big) \,:=\, (\bsi_{h,\star},\vec{\bu}_{h,\star})  \qan \mathscr{L}_{\mathtt{d}}^T(\bz_h , \phi_{h}) \,:=\, (\brho_{h,\star},\varphi_{h,\star})  \,.
	\]
	In addition, there hold the following a priori estimates
	\begin{equation}\label{eq:dis.apr1}
		\Vert\mathscr{L}^S_{\mathtt{d}}(\bz_h,\phi_h)\Vert\,=\,\Vert\bsi_{h,\star}\Vert_{\bdiv_{6/5},\Omega}+
		\Vert\vec{\bu}_{h,\star}\Vert_{0,6;\Omega}\,\leq\,  \dfrac{2}{\alpha_{\mathcal{A}^S,\mathtt{d}}}\,\Vert\phi_h\Vert_{0,6;\Omega}\,\Vert\bg\Vert_{0,3/2;\Omega}\,,
	\end{equation}
	and
	\begin{equation}\label{eq:dis.apr2}
		\Vert\mathscr{L}^T_{\mathtt{d}}(\bz_h,\phi_h)\Vert\,=\,\Vert\brho_{h,\star}\Vert_{\div_{6/5},\Omega}+\Vert\varphi_h\Vert_{0,6;\Omega}\,\leq\, \dfrac{2}{\alpha_{\mathcal{A}^{T},\mathtt{d}}}\, \Vert\varphi_{D}\Vert_{1/2,\Gamma_{D}}\,,
	\end{equation}
\end{lemma}

\subsection{Solvability analysis of the discrete fixed point}
Here we provide the main result of this section, namely, the existence of 
the discrete solution to scheme \eqref{eq:dis.ful.v1a-eq:v4-a}-\eqref{eq:dis.ful.v1a-eq:v4-d}.
More precisely, in a manner analogous to the continuous case, we proceed to demonstrate that $ \mathscr{L}_{\mathtt{d}} $ satisfies the hypotheses of the Banach theorem.
 First, we note that using now the a priori estimates \eqref{eq:dis.apr1} and \eqref{eq:dis.apr2}, for each $ (\bz_{h},\phi_{h})\in \mathbf{Y}_h \times\mathrm{Y}_h $ we easily obtain the discrete analogue of \eqref{eq:apr.E} as follow:
\begin{equation*}
	\Vert\mathscr{L}_{\mathtt{d}}(\bz_{h},\phi_{h})\Vert	
	\,  \leq \, \left(1+\dfrac{2}{\alpha_{\mathcal{A}^S,\mathtt{d}}}\Vert\bg\Vert_{0,3/2;\Omega}\right)\dfrac{2}{\alpha_{\mathcal{A}^{T},\mathtt{d}}}\,\Vert\varphi_{D}\Vert_{1/2,\Gamma_{D}}\,,
\end{equation*}
Hence, the discrete counterpart of Lemma \ref{l_cons} reads as follows.
\begin{lemma}\label{l_dis.cons}
Let 
 $ r_{\mathtt{d}} \,:=\, \min\{r_{1,\mathtt{d}},r_{2,\mathtt{d}}\} $
	and $ \mathbf{W}_{\mathtt{d}} $ be the closed ball in $ \mathbf{Y}_h \times\mathrm{Y}_h $ with center at the origin and radius $ r_{\mathtt{d}} $, and
assume that the data satisfy 
\begin{equation}\label{eq:as.data.3}
 \left(1+\dfrac{2}{\alpha_{\mathcal{A}^S,\mathtt{d}}}\Vert\bg\Vert_{0,3/2;\Omega}\right)\dfrac{2}{\alpha_{\mathcal{A}^{T},\mathtt{d}}}\,\Vert\varphi_{D}\Vert_{1/2,\Gamma_{D}}\, \leq \, r_{\mathtt{d}}\,.
\end{equation} 
Then, there holds $ \mathscr{L}_{\mathtt{d}}(\mathbf{W}_{\mathtt{d}}) \subseteq \mathbf{W}_{\mathtt{d}}$.
\end{lemma}
Now, the Lipschitz continuity of $ \mathscr{L}_{\mathtt{d}} $ will be addressed by the fact that both $ \mathscr{L}^S_{\mathtt{d}} $ and $ \mathscr{L}^T_{\mathtt{d}} $ satisfy this property, as in the continuous case. Indeed, we will provide discrete analogues to Lemmas \ref{l_lip.S} and \ref{l_lip.T} in what follows.
\begin{lemma}\label{l_dis.lip.S}
	There exists a positive constant $ \mathcal{L}_{\mathscr{L}^S,\mathtt{d}} $, depending on $ \mathcal{L}_{\mu} $, $ \mu_1 $, $ r_{1,\mathtt{d}} $, $ \alpha_{\mathcal{A}^S,\mathtt{d}} $, $ c_{\star} $, $ c^{\star} $, such that
	\begin{equation}\label{eq:dis.lip.S}
		\begin{array}{l}
			\big\Vert \mathscr{L}^S_{\mathtt{d}}(\bz_h,\phi_h)-\mathscr{L}^S_{\mathtt{d}}(\by_h,\varpi_h)\big\Vert_{\mathbb{X}\times\mathcal{Z}}\, \leq \, \mathcal{L}_{\mathscr{L}^S,\mathtt{d}}\, \Big\{\Vert\mathscr{L}_{2,\mathtt{d}}^S(\by_h,\varpi_h)\Vert_{0,6;\Omega}\Vert\bz_h-\by_h\Vert_{0,6;\Omega}
			\\[2ex]
			\,+\,\Big(\Vert\bg\Vert_{0,3/2;\Omega}(1+\Vert\mathscr{S}_{1,\mathtt{d}}^S(\by_h,\varpi_h)\Vert_{0,3;\Omega})+\Vert\mathscr{S}_{2,\mathtt{d}}^S(\by_h,\varpi_h)\Vert_{0,6;\Omega}\Big)\Vert\phi_h-\varpi_h\Vert_{0,6;\Omega}	
			\Big\}\,.
		\end{array}
	\end{equation}
\end{lemma}
\begin{proof}
Given $ (\bz_{h},\phi_h)\in \mathbf{Y}_h \times\mathrm{Y}_h, (\by_h,\varpi_h)\in \mathbf{Y}_h \times\mathrm{Y}_h $, we let $ \mathscr{L}^S_{\mathtt{d}}(\bz_h,\phi_h):=(\bsi_{h,\star},\vec{\bu}_{h,\star}) \in \mathbb{X}_h \times\mathcal{Z}_h$ and $ \mathscr{L}^S_{\mathtt{d}}(\by_h,\varpi_h):=(\bsi_{h,\circ},\vec{\bu}_{h,\circ}) \in \mathbb{X}_h \times\mathcal{Z}_h$, where $ (\bsi_{h,\star},\vec{\bu}_{h,\star})\in \mathbb{X}_h \times\mathcal{Z}_h $ and
$ (\bsi_{h,\circ},\vec{\bu}_{h,\circ})\in \mathbb{X}_h \times\mathcal{Z}_h $ are the respective solutions of \eqref{eq:dis.sub_1a}. Then, the proof of \eqref{eq:dis.lip.S}, starting now from
the global
discrete inf-sup condition of $ \mathcal{A}_{h,\phi_h}^S+\mathscr{O}^{S}_{h,\phi_{h}}(\bz_{h};\cdot,\cdot) $ (cf. \eqref{eq:dis.inf.A.o}) for each  $ (\bz_{h},\phi_h)\in \mathbf{Y}_h \times\mathrm{Y}_h $,
is very closely to the proof of Lemma \ref{l_lip.S}. 
Nevertheless, since the regularity assumption $ (\textbf{RA}) $ is not applicable in this discrete context, we proceed by using $ \mathrm{L}^{6}-\mathbb{L}^{3}-\mathbb{L}^{2} $ argument for deriving \eqref{eq:dis.lip.S}.
 Indeed, instead of proceeding as in \eqref{eq3:l.lip.S}, we use the definition of $ \mathscr{A}_{h,\phi_h}^{S} $, the lower and uper bounds in the first row of \eqref{eq28a-eq28} and the Lipschitz continuity of $ \mu $ (cf. first column of \eqref{eq:lip.mu.kap}) to
obtain
\begin{equation}\label{eq:lip.A}
	\begin{array}{c}
	\Big|\big(\mathscr{A}_{h,\phi_h}^{S}-\mathscr{A}_{h,\varpi_h}^S\big)(\bsi_{h,\circ},\bta_h)\Big|
	\,\leq \,
	\disp\sum\limits_{E\in \Omega_h}\big(\mathscr{A}_{\phi_h}^{S,E}-\mathscr{A}_{\varpi_h}^{S,E}\big)\left(\Pcalbb_{r}^{E}(\bsi_{h,\circ}),\Pcalbb_{r}^{E}(\bta_h)\right)\\[2ex]
	\,+\, \disp\sum\limits_{E\in \Omega_h} \big(\mathscr{S}_{\phi_h}^{S,E}-\mathscr{S}_{\varpi_h}^{S,E}\big)\left(\bsi_{h,\circ} - \Pcalbb_{r}^{E}(\bsi_{h,\circ}),\bta_{h} - \Pcalbb_{r}^{E}(\bta_{h})\right)\\[3ex]
	\,\leq \, \disp\sum\limits_{E\in \Omega_h}\dfrac{\mathcal{L}_{\mu}}{\mu_1^{2}}\,\Vert \phi_h-\varpi_h\Vert_{0,6,E}\, \big\Vert(\Pcalbb_{r}^{E}(\bsi_{h,\circ}))^{\mathtt{d}}\big\Vert_{0,3;E}\, \big\Vert (\Pcalbb_{r}^{E}(\bta_h))^{\mathtt{d}}\big\Vert_{0,E}\\[2.5ex]
	\,+\, \disp\sum\limits_{E\in \Omega_h} \Big\{\big(\mathscr{S}_{\phi_h}^{S,E}-\mathscr{S}_{\varpi_h}^{S,E}\big)\left(\bsi_{h,\circ} - \Pcalbb_{r}^{E}(\bsi_{h,\circ}),\bsi_{h,\circ} - \Pcalbb_{r}^{E}(\bsi_{h,\circ})\right)\Big\}^{1/2}\\[2ex]
	\,\times\,\Big\{\big(\mathscr{S}_{\phi_h}^{S;E}-\mathscr{S}_{\varpi_h}^{S,E}\big)\left(\bsi_{h,\circ} - \Pcalbb_{r}^{E}(\bta_{h}),\bta_{h} - \Pcalbb_{r}^{E}(\bta_{h})\right)\Big\}^{1/2}\\[2.5ex]	
	 \, \leq\, \dfrac{\mathcal{L}_{\mu}}{\mu_1^{2}}\, \big(1+\max\{c_\star,c^{\star}\}\big)\, \Vert \phi_h-\varpi_h \Vert_{0,6;\Omega}\,\Vert \bsi_{h,\circ} \Vert_{0,3;\Omega} \Vert\bta_{h}\Vert_{0,\Omega}\,,
	\end{array}
\end{equation}
and instead of \eqref{eq3a:l.lip.S}, employing the first row of \eqref{eq:boundCF} we have
\begin{equation*}
	\begin{array}{c}
	\big|\mathscr{O}^{S}_{h,\varpi_h}(\bz_h-\by_h;\bu_{h,\circ},\bta_h)\big|\,\leq\, \dfrac{1}{\mu_1}\,\Vert	\bz_h-\by_h\Vert_{0,6;\Omega}\,\Vert\bu_{h,\circ}\Vert_{0,6;\Omega}\,\Vert\bta_h\Vert_{0,\Omega}\,,
	\end{array}
\end{equation*}
The rest of the estimates are similar to those in the proof of Lemma \ref{l_lip.S}, and hence further details are omitted.
\end{proof}
\begin{lemma}\label{l_dis.lip.St} 
		There exists a positive constant $ \mathcal{L}_{\mathscr{L}^T,\mathtt{d}} $, depending on $ \mathcal{L}_{\kappa} $, $ \kappa_1 $, $ r_{2,\mathtt{d}} $, $ \alpha_{\mathcal{A}^T,\mathtt{d}} $, $ \tilde{c}_{\star} $, $ \tilde{c}^{\star} $, such that
	\begin{equation}\label{eq:dis.lip.S.t}
		\begin{array}{l}
			\big\Vert\mathscr{L}^T_{\mathtt{d}}(\bz_h,\phi_h)-\mathscr{L}^T_{\mathtt{d}}(\by_h,\varpi_h)\big\Vert\\[2ex]
			\hspace{-.25cm}\, \leq\, \mathcal{L}_{\mathscr{L}^T,{\mathtt{d}}}\, \Big\{\big\Vert\mathscr{L}^T_{2,\mathtt{d}}(\by_h,\varpi_h)\big\Vert_{0,6;\Omega}\,\Vert\bz_h-\by_h\Vert_{0,6;\Omega}\,+\,\left(1+\big\Vert\mathscr{L}^T_{1,\mathtt{d}}(\by_h,\varpi_h)\big\Vert_{0,3;\Omega}\right)\Vert\phi_h-\varpi_h\Vert_{0,6;\Omega}\Big\}\,.
		\end{array}
	\end{equation}
\end{lemma}
\begin{proof}
The result similarly follows the arguments presented in the proof of Lemma \ref{l_lip.T}, albeit without taking into account the regularity assumption $ (\widetilde{ \mathbf{RA}}) $, which is not authentic in the discrete case at hand.
 For that purpose, we
use arguments utilized in inequality \eqref{eq:lip.A} given by Lemma \ref{l_dis.lip.S}, and hence we omit the corresponding details.
\end{proof}
Hence, according to the definition of $ \mathscr{L}_{\mathtt{d}} $ (cf. \eqref{eq:dis.def.opr.E}), and utilizing the a prior estimate \eqref{eq:dis.lip.S}, we conclude that
\begin{equation*}
	\begin{array}{c}
	\big\Vert	\mathscr{L}_{\mathtt{d}}\big(\bz_h,\phi_h\big)-\mathscr{L}_{\mathtt{d}}\big(\by_h,\varpi_h\big)\big\Vert\,
	\,\leq\, \mathcal{L}_{\mathscr{L}^S,\mathtt{d}}\, \big\Vert\mathscr{L}_{2,\mathtt{d}}^S(\by_h,\mathscr{L}_{2,\mathtt{d}}^T(\by_h,\varpi_h))\big\Vert_{0,6;\Omega}\Vert\bz_h-\by_h\Vert_{0,6;\Omega}
	\\[2ex]
	\,+\,\Big\{ \mathcal{L}_{\mathscr{L}^S,\mathtt{d}}\Big(\Vert\bg\Vert_{0,3/2;\Omega}\,\big(1+\Vert\mathscr{L}_{1,\mathtt{d}}^S\big(\by_h,\mathscr{L}_{2,\mathtt{d}}^T(\by_h,\varpi_h)\big)\Vert_{0,3;\Omega}\big)\\[2ex]
	+\Vert\mathscr{L}_{2,\mathtt{d}}^S\big(\by_h,\mathscr{L}_{2,\mathtt{d}}^T(\by_h,\varpi_h)\big)\Vert_{0,6;\Omega}\Big)+1\Big\}\,\big\Vert \mathscr{L}_{2,\mathtt{d}}^T(\bz_h,\phi_h)-\mathscr{L}_{2,\mathtt{d}}^T(\by_h,\varpi_h)\big\Vert_{0,6;\Omega}	\\[2.5ex]
	\,\leq\, 
 \mathcal{L}_{\mathscr{L}^S,\mathtt{d}}\, \big\Vert\mathscr{L}_{2,\mathtt{d}}^S(\by_h,\mathscr{L}_{2,\mathtt{d}}^T(\by_h,\varpi_h))\big\Vert_{0,6;\Omega}\Vert\bz_h-\by_h\Vert_{0,6;\Omega}
	\\[2ex]
	\,+\,\mathcal{L}_{\mathscr{L}^T,\mathtt{d}}\,\Big\{ \mathcal{L}_{\mathscr{L}^S,\mathtt{d}}\,\Big(\Vert\bg\Vert_{0,3/2;\Omega}\,\big(1+\big\Vert\mathscr{L}^S_{1,\mathtt{d}}(\by_h,\mathscr{L}^T_{2,\mathtt{d}}(\by_h,\varpi_h))\big\Vert_{0,3;\Omega}\big)\\[2ex]
	\,+\,\big\Vert\mathscr{L}^S_{2,\mathtt{d}}(\by_h,\mathscr{L}^T_{2,\mathtt{d}}(\by_h,\varpi_h))\big\Vert_{0,6;\Omega}\Big)+1\Big\}\\[2ex]
	\, \times\,\Big(\big\Vert \mathscr{L}^T_{2,\mathtt{d}}(\by_h,\varpi_h)\big\Vert_{0,6;\Omega}\,\Vert\bz_h-\by_h\Vert_{0,6;\Omega}\,+\,\big(1+\big\Vert \mathscr{L}^T_{1,\mathtt{d}}(\by_h,\varpi_h)\big\Vert_{0,3;\Omega}\big)\Vert\phi_h-\varpi_h\Vert_{0,6;\Omega}\Big)\,,
	\end{array}
\end{equation*}
from which, applying Lemma \ref{l_dis.wel.S.St}, we find that
\begin{equation}\label{eq:lip.con.Th}
	\begin{array}{c}
		\big\Vert	\mathscr{L}_{\mathtt{d}}\big(\bz_h,\phi_h\big)-\mathscr{L}_{\mathtt{d}}\big(\by_h,\varpi_h\big)\big\Vert
		\, \leq \,  \mathcal{L}_{\mathscr{L},\mathtt{d}} \, \Big\{
	\Big(1+\Vert\varphi_{D}\Vert_{1/2,\Gamma_{D}}\big\Vert\mathscr{L}^S_{1,\mathtt{d}}(\by_h,\mathscr{L}^T_{2,\mathtt{d}}(\by_h,\varpi_h))\big\Vert_{0,3;\Omega}\\[2ex]
	\,+\,\Vert\varphi_{D}\Vert_{1/2,\Gamma_{D}}\Big)\, \Vert\bg\Vert_{0,3/2;\Omega}\, \Vert\bz_{h}-\by_h\Vert_{0,6;\Omega}\\[2ex]
	\,+\, \Vert\bg\Vert_{0,3/2;\Omega}\,\Big(1+\big\Vert\mathscr{L}^S_{1,\mathtt{d}}(\by_h,\mathscr{L}^T_{2,\mathtt{d}}(\by_h,\varpi_h))\big\Vert_{0,3;\Omega}
	\,+\,\Vert\varphi_{D}\Vert_{1/2,\Gamma_{D}}\Big)
	\, \Big(1\\[2ex]
	\,+\,\big\Vert\mathscr{L}_{1,\mathtt{d}}^T(\by_h,\varpi_h)\big\Vert_{0,3;\Omega}\Big)\Vert\phi_h-\varpi_h\Vert_{0,6;\Omega}
		\Big\}\,,
	\end{array}
\end{equation}

where $ \mathcal{L}_{\mathscr{L},\mathtt{d}} $ is the positive constant depending on $ \mathcal{L}_{\mathscr{L}^S,\mathtt{d}} $, $ \mathcal{L}_{\mathscr{L}^T,\mathtt{d}} $, $ \alpha_{\mathcal{A}^S,\mathtt{d}}$ and $ \alpha_{\mathcal{A}^{T},\mathtt{d}} $.
We remark that the absence of control over terms $ \big\Vert\mathscr{L}^S_{1,\mathtt{d}}(\by_h,\mathscr{L}^T_{2,\mathtt{d}}(\by_h,\varpi_h))\big\Vert_{0,3;\Omega} $ and $ \big\Vert\mathscr{L}_{1,\mathtt{d}}^T(\by_h,\varpi_h)\big\Vert_{0,3;\Omega} $ prevents us from deducing Lipschitz continuity, and therefore, the operator's contractivity cannot be established.

We end this section by establishing the following main result.
\begin{theorem}\label{t:dis.exis}
Assume that data satisfy \eqref{eq:as.data.3}. Then, the operator $ \mathscr{L}_{\mathtt{d}} $ has at least one fixed point $ (\bu_h , \varphi_h) \in \mathbf{W}_{\mathtt{d}} $. Equivalently, the virtual scheme \eqref{eq:dis.ful.v1a-eq:v4-a}-\eqref{eq:dis.ful.v1a-eq:v4-d} has at least one solution $ (\bsi_h, \vec{\bu}_h)\in\mathbb{X}_h\times\mathcal{Z}_h $ and $ (\brho_h, \varphi_h)\in \mathbf{X}_h\times \mathrm{Y}_h$. Moreover, there holds
	\begin{equation}\label{eq:dis.apr.prob1}
	\Vert\bsi_{h}\Vert_{\bdiv_{6/5},\Omega}+
	\Vert\vec{\bu}_{h}\Vert_{0,6;\Omega}\,\leq\,  \dfrac{4}{\alpha_{\mathcal{A}^S,\mathtt{d}}\,\alpha_{\mathcal{A}^S,\mathtt{d}}}\, \Vert\varphi_{D}\Vert_{1/2,\Gamma_{D}}\,\Vert\bg\Vert_{0,3/2;\Omega}\,,
\end{equation}
and
\begin{equation}\label{eq:dis.apr.prob2}
\Vert\brho_{h}\Vert_{\div_{6/5},\Omega}+\Vert\varphi_h\Vert_{0,6;\Omega}\,\leq\, \dfrac{2}{\alpha_{\mathcal{A}^{T},\mathtt{d}}}\, \Vert\varphi_{D}\Vert_{1/2,\Gamma_{D}}\,,
\end{equation}
\end{theorem}
\begin{proof}
The first conclusion of this theorem is acquired by Lemma \ref{l_dis.cons}, the continuity of $\mathscr{L}_{\mathtt{d}}$ (cf. 
\eqref{eq:lip.con.Th}), the equivalence between \eqref{eq:dis.fixp} and
\eqref{eq:dis.ful.v1a-eq:v4-a}-\eqref{eq:dis.ful.v1a-eq:v4-d}, and a direct application of Banach fixed-point theorem (cf. \cite[Theorem 9.9-2]{c-SIAM-2013}).
Next, noting that $ \mathscr{L}_{\mathtt{d}}^S(\bu_{h},\varphi_h)=(\bsi_{h}, \vec{\bu}_{h}) $ and $ \mathscr{L}_{\mathtt{d}}^T(\bu_{h},\varphi_h)=(\bsi_{h},\varphi_h) $ priori
estimates \eqref{eq:dis.apr.prob1} and \eqref{eq:dis.apr.prob2} follows from \eqref{eq:dis.apr1} and \eqref{eq:dis.apr2}, respectively.
\end{proof}
\newpage	
\section{Error estimates}\label{sec.6}
		
In this section, considering the product spaces 
\[
\mathcal{V}\,:=\, \mathbb{X}\times\mathcal{Z}\,,\quad \mathcal{V}_h \,:=\, \mathbb{X}_h\times\mathcal{Z}_h \qan  \mathcal{W}\,:=\, \mathbf{X}\times\mathrm{Y}\quad \mathcal{W}_h\,:=\, \mathbf{X}_h\times\mathrm{Y}_h \,,
\]
 equipped with the norms
 \[
 \Vert (\bta,\vec{\bv})\Vert_{\mathcal{V}}\,:=\, \Vert\bta\Vert_{\bdiv_{6/5};\Omega}+\Vert\vec{\bv}\Vert \qan 
 \Vert (\bet,\psi)\Vert_{\mathcal{W}}\,:=\, \Vert\bet\Vert_{\div_{6/5};\Omega}+\Vert\psi\Vert_{0,6;\Omega}\,,
 \]
we derive the Céa estimate for the global error
$$
\|(\bsi, \vec{\bu}) - (\bsi_h, \vec{\bu}_h)\|_{\mathcal{V}} \,+\, \|(\brho , \varphi) - (\brho_h , \varphi_h)\|_{\mathcal{W}}\,,
$$%
		
where $\big((\bsi, \vec{\bu}),(\brho , \varphi)\big) \in \mathcal{V}\times\mathcal{W}$, with $(\bu,\varphi) \in \mathbf W$ (cf. \eqref{de:ball}),
is the unique solution of \eqref{eq:v1a-eq:v4-a}-\eqref{eq:v1a-eq:v4-d}, and $\big((\bsi_h, \vec{\bu}_h),(\brho_h , \varphi_h)\big) \in \mathcal{V}_h\times\mathcal{W}_h$, with $(\bu_h , \varphi_h) \in \mathbf W_{\mathtt{d}}$ (cf. Lemma \ref{l_dis.cons}),
is the solution of \eqref{eq:dis.ful.v1a-eq:v4-a}-\eqref{eq:dis.ful.v1a-eq:v4-d}. 
As a result of this, we also obtain the a priori estimate for pressure, considering that the discrete pressure calculated according to the postprocessing formula given by \eqref{eq:postpr}, that is,
\begin{equation}\label{eq:dis.postpr}
p_h \,=\,-\dfrac{1}{n}\,\Big\{ \tr(\Pcalbb_{r}^{h}\bsi_h)+\tr(\bu_h\otimes\bu_h)\Big\}
\qin \Omega \,.
\end{equation}
To do that, we will utilize known Strang-type estimates to the pairs of corresponding continuous and discrete schemes resulting from \eqref{eq:v1a-eq:v4-a}-\eqref{eq:v1a-eq:v4-d} and \eqref{eq:dis.ful.v1a-eq:v4-a}-\eqref{eq:dis.ful.v1a-eq:v4-d} after they have been divided into the two decoupled problems.

\subsection{The main Céa estimate}
We begin the analysis with considering \eqref{eq:v1a-eq:v4-a}-\eqref{eq:v1a-eq:v4-b} and \eqref{eq:dis.ful.v1a-eq:v4-a}-\eqref{eq:dis.ful.v1a-eq:v4-b}, which can be rewritten as
\begin{equation*}
	\begin{array}{c}
\mathcal{A}_{h,\phi}^S \big((\bsi\,,\vec{\bu}), (\bta,\, \vec{\bv})\big)+\mathscr{O}_{h,\phi}^{S}(\bz;\bu,\bta)\,=\, \mathcal{F}_{\phi}^S(\bta,\,\vec{\bv})\qquad\forall\,(\bta , \vec{\bv})\in \mathcal{V}\, \qan\\[2ex]
\mathcal{A}_{h,\phi_h}^{S} \big((\bsi_{h}\,,\vec{\bu}_{h}), (\bta_h,\, \vec{\bv}_h)\big)+\mathscr{O}^{S}_{h,\phi}(\bz_h;\bu_{h},\bta_h)\,=\, \mathcal{F}_{\phi_h}^{S}(\bta_h, \vec{\bv}_h)\qquad\forall\,(\bta_h , \vec{\bv}_h)\in \mathcal{W}_h\,,
	\end{array}
\end{equation*}
where $ \mathcal{A}_{\phi}^S $ and $ \mathcal{A}_{h,\phi_h}^{S} $ are defined in \eqref{def.A} and \eqref{def.dis.A}, respectively.


Now, utilizing the a priori error estimate presented in \cite[Lemma 6.1]{cgm-ESAIM-2020}, and subsequently appropriately bounding the resulting consistency estimate, represented by 
\[
\Vert \mathcal{A}^S_{\varphi}\big((\bsi,\vec{\bu}),(\cdot,\cdot)\big)-\mathcal{A}^S_{h,\varphi_h}\big((\bsi,\vec{\bu}),(\cdot,\cdot)\big)\Vert_{\mathcal{V}_h^{'}}\,,
\]
and
\[
\Vert \mathscr{O}_{\varphi}^{S}(\bu;\bu,\cdot)- \mathscr{O}^{S}_{h,\varphi_h}(\bu_h;\bu,\cdot)\Vert_{\mathbb{X}_h^{'}}\,,
\]
we infer the existence of a constant $ \mathcal{C}_{\mathtt{st}}^S >0$, depending only on $ \| \mathcal{A}^S\| $, $ \mathcal{C}_{\mathcal{A}^{S}} $, $ \mathcal{C}_{\mathcal{O}^{S}} $, and $ \alpha_{\mathcal{A}^S,\mathtt{d}} $, and thus, easily shown to be independent of $ h $, such that
\begin{equation}\label{eq0:est.prob1}
	\begin{array}{c}
\Vert (\bsi,\vec{\bu})-(\bsi_h,\vec{\bu}_h)\Vert_{\mathcal{V}}\,\leq\, \mathcal{C}_{\mathtt{st}}^S\, \Big\{\dist\big((\bsi,\vec{\bu}),\mathcal{V}_h\big) + \Vert\mathcal{F}_{\varphi}^S-\mathcal{F}_{\varphi_h}^S\Vert_{\mathcal{V}_h^{'}}\\[2ex]
\,+\,\big\Vert \big(\mathscr{A}_{\varphi}^{S}-\mathscr{A}_{h,\varphi_h}^{S}\big)(\bze_h,\cdot)\big\Vert_{\mathbb{X}_h^{'}}+\big\Vert\big(\mathscr{B}^S-\mathscr{B}_{h}^{S}\big)(\bze_h,\cdot)\big\Vert_{\mathcal{Z}_{h}^{'}}
 +\big\Vert \big(\mathscr{B}^S-\mathscr{B}_{h}^{S}\big)(\cdot,\vec{\bw}_h)\big\Vert_{\mathbb{X}_h^{'}}\\[2ex]
 \,+\,\big\Vert \mathscr{O}_{\varphi}^{S}(\bu;\bw_h,\cdot)- \mathscr{O}^{S}_{h,\varphi_h}(\bu_h;\bw_h,\cdot)\big\Vert_{\mathbb{X}_h^{'}}
 \Big\}\,.
	\end{array}
\end{equation}

Our objective now is to bound each of the consistency terms that appear on the right-hand side of \eqref{eq0:est.prob1}.
First, following the same procedure used to derive \eqref{eq4:l.lip.S}, we easily obtain
\begin{equation}\label{eq:k1}
\Vert\mathcal{F}_{\varphi}^S-\mathcal{F}_{\varphi_h}^S\Vert_{\mathcal{V}_h^{'}}\, \leq \,	\Vert\bg\Vert_{0,3/2;\Omega}\, \Vert \varphi-\varphi_h\Vert_{0,6;\Omega}\,.
\end{equation}

On the other hand, regarding terms $ \Vert \big(\mathscr{A}_{\varphi}^{S}-\mathscr{A}_{h,\varphi_h}^{S}\big)(\bsi,\cdot)\Vert_{\mathbb{X}_h^{'}} $ and $ \Vert \mathscr{O}_{\varphi}^{S}(\bu;\bu,\cdot)- \mathscr{O}^{S}_{h,\varphi_h}(\bu_h;\bu,\cdot)\Vert_{\mathbb{X}_h^{'}} $, we add and subtract the following terms $$ \mathscr{A}_{\varphi}^{S}(\bsi,\bta_{h})\,,\,\, \mathscr{A}_{h,\varphi_h}^{S}(\bsi,\bta_{h}) \qan \mathscr{O}_{\varphi}^{S}(\bu;\bu,\bta_{h}) \,,\,\, \mathscr{O}^{S}_{h,\varphi_h}(\bu_h;\bu,\bta_{h}) \,,\,\, \mathscr{O}^{S}_{h,\varphi_h}(\bu;\bu,\bta_{h})\,, $$
separately to each ones, that for all $ \bta_{h}\in \mathbb{X}_h $, yields, respectively,
\begin{equation}\label{eq:p1}
	\begin{array}{c}
		\big(\mathscr{A}_{\varphi}^{S}-\mathscr{A}_{h,\varphi_h}^{S}\big)(\bze_h,\bta_h)\,=\, \mathscr{A}_{\varphi}^{S}(\bze_h-\bsi, \bta_{h}) +\mathscr{A}_{h,\varphi_h}^{S}(\bsi-\bze_h,\bta_{h})+ \big(\mathscr{A}_{\varphi}^{S}-\mathscr{A}_{h,\varphi_h}^{S}\big)(\bsi, \bta_{h})\,,
	\end{array}
\end{equation}
and
\begin{equation}\label{eq:p2}
\begin{array}{c}
		\mathscr{O}_{\varphi}^{S}(\bu;\bw_h,\bta_h)- \mathscr{O}^{S}_{h,\varphi_h}(\bu_h;\bw_h,\bta_h)\,=\,\mathscr{O}_{\varphi}^{S}(\bu;\bw_h-\bu,\bta_h)+\mathscr{O}^{S}_{h,\varphi_h}(\bu_h;\bu-\bw_h,\bta_h)\\[2ex] \,+\, \mathscr{O}^{S}_{h,\varphi_h}(\bu-\bu_h;\bu,\bta_h)+\big(\mathscr{O}_{\varphi}^{S}-\mathscr{O}^{S}_{h,\varphi_h}\big)(\bu;\bu,\bta_{h})\,.
\end{array}
\end{equation}

In this regard, employing the boundedness of $ \mathscr{A}_{\varphi}^{S} $, $ \mathscr{A}_{h,\varphi_h}^{S} $, and $ \mathscr{O}_{\varphi}^{S} $, $ \mathscr{O}^{S}_{h,\varphi_h} $ with respective constants $$ \| \mathscr{A}^S\| =1/\mu_1 \,, \,\, \mathcal{C}_{\mathcal{A}^{S}} = (1+c^\star)/\mu_{1} \qan \| \mathscr{O}^S\| =|\Omega|^{1/6}/\mu_1  \,,\,\, \mathcal{C}_{\mathcal{O}^{S}} =|\Omega|^{1/6}/\mu_1 \,,  $$
 we get
\begin{equation}\label{eq:k2}
	\begin{array}{c}
		\big|\mathscr{A}_{\varphi}^{S}(\bze_h-\bsi, \bta_{h})\big|\, \leq \, \| \mathscr{A}^S\|\, \Vert \bze_h-\bsi\Vert_{0,\Omega}\, \Vert \bta_{h}\Vert_{0,\Omega}\,,\\[2ex]
		\big|\mathscr{A}_{h,\varphi_h}^{S}(\bsi-\bze_h,\bta_{h})\big|\, \leq \, \mathcal{C}_{\mathcal{A}^{S}}\, \Vert \bze_h-\bsi\Vert_{0,\Omega}\, \Vert \bta_{h}\Vert_{0,\Omega}\,,\\[2ex]
		\big|\mathscr{O}_{\varphi}^{S}(\bu;\bw_h-\bu,\bta_h)\big|\,\leq\, \|  \mathscr{O}^S\| \Vert\bu\Vert_{0,6;\Omega}\, \Vert\bw_h-\bu\Vert_{0,6;\Omega}\, \Vert \bta_{h}\Vert_{0,\Omega}\,,\\[2ex]
		\big|\mathscr{O}^{S}_{h,\varphi_h}(\bu_h;\bu-\bw_h,\bta_h)\big|\,\leq\, \mathcal{C}_{\mathcal{O}^{S}}\, \Vert\bu_h\Vert_{0,6;\Omega}\, \Vert\bw_h-\bu\Vert_{0,6;\Omega}\, \Vert \bta_{h}\Vert_{0,\Omega}\,,\\[2ex]
		\big|\mathscr{O}^{S}_{h,\varphi_h}(\bu-\bu_h;\bu,\bta_h)\big|\,\leq\, \mathcal{C}_{\mathcal{O}^{S}}\,\Vert \bu-\bu_h\Vert_{0,6;\Omega}\, \Vert\bu\Vert_{0,6;\Omega}\, \Vert \bta_{h}\Vert_{0,\Omega}\,.
	\end{array}
\end{equation}

In addition, for last terms in \eqref{eq:p1} and \eqref{eq:p2}, by adding and subtracting some suitable terms, there hold
\begin{equation}\label{eq:p4}
	\begin{array}{c}
	\big(\mathscr{A}_{\varphi}^{S}-\mathscr{A}_{h,\varphi_h}^{S}\big)(\bsi, \bta_{h})\, =\, \big(\mathscr{A}_{h,\varphi}^S-\mathscr{A}_{h,\varphi_h}^{S}\big)(\bsi, \bta_{h})+\big(\mathscr{A}_{\varphi}^{S}-\mathscr{A}_{h,\varphi}^S\big)(\bsi, \bta_{h})\,,\qan\\[2ex]
	\big(\mathscr{O}_{\varphi}^{S}-\mathscr{O}^{S}_{h,\varphi_h}\big)(\bu;\bu,\bta_{h})\,=\, \big(\mathscr{O}^{S}_{h,\varphi}-\mathscr{O}^{S}_{h,\varphi_h}\big)(\bu;\bu,\bta_{h})+\big(\mathscr{O}_{\varphi}^{S}-\mathscr{O}^{S}_{h,\varphi}\big)(\bu;\bu,\bta_{h})\,.
	\end{array}
\end{equation}

Concerning the first terms mentioned above, utilizing the Lipschitz continuity of $ \mu $ (cf. first column of \eqref{eq:lip.mu.kap}), along with Cauchy-Schwarz and H\"older's inequalities, followed by the regularity estimate \eqref{eq:reg.t} and considering that the norm of $ \varphi $ is bounded by the radius $ r $ of the ball $ \mathbf{W} $ (cf. Theorem \ref{t:exis}), we easily get
\begin{equation}\label{eq:j1}
	\begin{array}{c}
		\big|\big(\mathscr{A}_{h,\varphi}^S-\mathscr{A}_{h,\varphi_h}^{S}\big)(\bsi, \bta_{h})\big|
		\,\leq \,
		\disp\sum\limits_{E\in \Omega_h}\big(\mathscr{A}_{\varphi}^{S,E}-\mathscr{A}_{\varphi}^{S,E}\big)\left(\Pcalbb_{r}^{E}(\bsi),\Pcalbb_{r}^{E}(\bta_h)\right)\\[2ex]
		\,+\, \disp\sum\limits_{E\in \Omega_h} \big(\mathscr{L}_{\varphi}^{S,E}-\mathscr{S}_{\varphi_h}^{S,E}\big)\left(\bsi - \Pcalbb_{r}^{E}(\bsi),\bta_{h} - \Pcalbb_{r}^{E}(\bta_{h})\right)\\[2.8ex]	
		\, \leq\, \dfrac{\mathcal{L}_{\mu}}{\mu_1^{2}}\,\mathcal{C}_{\epsilon}\| i_\epsilon\|\, \big(1+\max\{c_\star,c^{\star}\}\big)\, \Vert \varphi-\varphi_h \Vert_{0,6;\Omega}\,\Vert \bsi \Vert_{\epsilon,\Omega} \Vert\bta_{h}\Vert_{0,\Omega}\\[2.5ex]
		\, \leq\, \widehat{\mathcal{C}}_{\epsilon}\, \Vert \varphi-\varphi_h \Vert_{0,6;\Omega}\,\Vert\bta_{h}\Vert_{0,\Omega}\,,
	\end{array}
\end{equation}
where $ \widehat{\mathcal{C}}_{\epsilon} $ is constant depending on $ \mathcal{L}_{\mu} $, $ \mu_{1} $, $ \mathcal{C}_{\epsilon} $, $ c_\star $, $ c^{\star} $, $ \| \bg\| $ and $ r $,
whereas an application of definition of $ \mathscr{O}^{S}_{h,\varphi} $, the Lipschitz continuity of $ \mu $ (cf. first column of \eqref{eq:lip.mu.kap}), the H\"older's inequality
 guarantee that
\begin{equation}\label{eq:j2}
	\big|\big(\mathscr{O}^{S}_{h,\varphi}-\mathscr{O}^{S}_{h,\varphi_h}\big)(\bu;\bu,\bta_{h})\big|\,\leq \, \dfrac{\mathcal{L}_{\mu}}{\mu_1^{2}}\,\Vert \varphi-\varphi_h \Vert_{0,6;\Omega}\,\Vert\bu\Vert_{0,6;\Omega}^{2}\,  \Vert\bta_{h}\Vert_{0,\Omega}\,.
\end{equation}

To control the expressions involving $ \big(\mathscr{A}_{\varphi}^{S}-\mathscr{A}_{h,\varphi}^S\big) $ and $ \big(\mathscr{O}_{\varphi}^{S}-\mathscr{O}^{S}_{h,\varphi}\big) $ in the first and second rows of \eqref{eq:p4}, respectively, as well as the term $ (\mathscr{B}^S-\mathscr{B}_{h}^{S}) $ in \eqref{eq0:est.prob1}, we state the next result.
\begin{lemma}\label{l_A-l_C}
	There exist positive constants $\ell_{\mathscr{A}^{S}}$ and $\ell_{\mathscr{B}^{S}}$, independent of $h$, such that for each $ \phi\in \mathrm{Y}_h $ there hold
	\begin{subequations}
		\begin{align}
			\big|\big(\mathscr{A}_{\phi}^{S}-\mathscr{A}_{h,\phi}^S\big)(\bze_h,\bta_{h})\big|\, \leq\, \ell_{\mathscr{A}^{S}}\, \|\bze_h - \Pcalbb^h_r(\bze_h)\|_{0,\Omega}\,\|\bta_h\|_{0,\Omega}\quad\forall\,(\bze_h,\bta_{h})\in \mathbb{X}_h \times\mathbb{X}_h \,, \label{eq:aah}\\[1ex]
			\big| (\mathscr{B}^S - \mathscr{B}_{h}^{S})(\bta_h ,\vec{\bv}_h )\big| \,\le\, \ell_{\mathscr{B}^{S}}\,
			\|\bom_h - \Pcalbb^h_r(\bom_h)\|_{0,\Omega} \,\|\bta_h\|_{0,\Omega} \quad\forall\, (\bta_h,\vec{\bv}_h)\in \mathbb{X}_h \times\mathcal{Z}_h \,, \label{eq:bbh.a}\\[1ex]
			\big| (\mathscr{B}^S - \mathscr{B}_{h}^{S})(\bta_h ,\vec{\bv}_h )\big| \,\le\, \ell_{\mathscr{B}^{S}}\, \|\bta_h - \Pcalbb^h_r(\bta_h)\|_{0,\Omega}\,\Vert\bom_h\Vert_{0,\Omega}\quad\forall\, (\bta_h,\vec{\bv}_h)\in \mathbb{X}_h \times\mathcal{Z}_h \,,\label{eq:bbh.b} 
		\end{align}
	\end{subequations}
In addition, there hold
\begin{equation}\label{eq:ooh}
\big|\big(\mathscr{O}_{\phi}^{S}-\mathscr{O}^{S}_{h,\phi}\big)(\bz_h;\bw_h,\bta_{h})\big|\, \leq\, \dfrac{1}{\mu_{1}}\, \Vert (\bz_{h}\otimes\bw_h)-\Pcalbb^h_r(\bz_{h}\otimes\bw_h)\Vert_{0,\Omega}\,\|\bta_h\|_{0,\Omega} \,,
\end{equation}
for all $ (\bz_h,\bw_h,\bta_{h})\in \mathbf{Y}_h \times\mathbf{Y}_h \times\mathbb{X}_h $.
\end{lemma}
\begin{proof}
	Firstly, using \cite[Lemmas 4.4 and 4.6]{gs-M3AS-2021}, the first estimate of \eqref{eq28a-eq28} and the the Lipschitz continuity of $ \mu $ (cf. first column of \eqref{eq:lip.mu.kap}), we get \eqref{eq:aah} with $ \ell_{\mathscr{A}^{S}}\,:=\, \frac{1}{\mu_1}\big(\frac{\mathcal{L}_\mu}{\mu_1}+c^*\big) $, and \eqref{eq:ooh}. In turn, in order to prove \eqref{eq:bbh.a}, employing the definitions of $ \mathscr{B}^S $ and $ \mathscr{B}_{h}^{S} $, and the orthogonality property of $ \Pcalbb^E_r $, we get
	\begin{equation*}
		\begin{array}{l}
			\mathscr{B}^{S}(\bta_h,\vec{\bv}_h )-\mathscr{B}_{h}^{S}(\bta_h ,\vec{\bv}_h )\,=\,
			\disp
			\sum_{E\in\Omega_h}\int_{E}\left(\boldsymbol{\tau}_h-\Pcalbb^E_r(\bta_h)\right):\bom_h \\[2mm]
			\disp
			\,=\,\sum_{E\in\Omega_h}\int_{E}\left(\bta_h-\Pcalbb^E_r(\bta_h)\right):\left(\bom_h-\Pcalbb^E_r(\bom_h)\right)\,,
		\end{array}
	\end{equation*}
	from which, applying Cauchy--Schwarz's inequaltiy and the estimate \eqref{eq_Q0} with setting $ m=s=0 $ and $ p=2 $, and then summing over all $ E\in\Omega_h $, we arrive at \eqref{eq:bbh.a} and \eqref{eq:bbh.b} by taking $ \ell_{\mathscr{B}^{S}} = 1 $.
\end{proof}

\medskip
In order to use inequalities \eqref{eq:bbh.a} and \eqref{eq:bbh.b} for bounding terms $ (\mathscr{B}^S-\mathscr{B}_{h}^{S})(\bze_{h},\vec{\bv}_h) $ and $ (\mathscr{B}^S-\mathscr{B}_{h}^{S})(\bta_{h},\vec{\bw}_h) $ in terms of $ \bsi $ and $ \vec{\bu} $, we need to add and subtract some suitable terms, which is done 
next. More precisely, adding and subtracting the pairs terms of $(\bsi, \bgam)$ and $(\Pcalbb^h_r(\bsi), \Pcalbb^h_r(\bgam))$, and using the continuity property of $ \Pcalbb^h_r $, we easily find that
\[
\big\|\bze_h - \Pcalbb^h_r(\bze_h)\big\|_{0,\Omega} \,\le\, \big\|\bsi - \Pcalbb^h_r(\bsi)\big\|_{0,\Omega}
\,+\, 2\,\|\bsi - \bze_h\|_{0,\Omega}\,,\qan
\]
\[
\big\|\bdel_h - \Pcalbb^h_r(\bdel_h)\big\|_{0,\Omega}\,\le\, \big\|\bgam - \Pcalbb^h_r(\bgam)\big\|_{0,\Omega}
\,+\, 2\,\|\bgam - \bdel_h\|_{0,\Omega}\,,
\]
which yields
\begin{equation}\label{a-a-h-bze-h-bta-h}
	\begin{array}{l}
		\big|(\mathscr{B}^S - \mathscr{B}_{h}^{S})(\bze_h ,\vec{\bv}_h )\big| \,\le\, \ell_{\mathscr{B}^{S}}\,
		\Big\{\big\|\bsi - \Pcalbb^h_r(\bsi)\big\|_{0,\Omega}
		\,+\, 2\,\|\bsi - \bze_h\|_{0,\Omega}\Big\} \,\|\bom_h\|_{0,\Omega} \,,\qan\\[3mm]
		\big|(\mathscr{B}^S - \mathscr{B}_{h}^{S})(\bta_h ,\vec{\bw}_h )\big| \,\le\, \ell_{\mathscr{B}^{S}}\,\Big\{
		\big\|\bgam - \Pcalbb^h_r(\bgam)\big\|_{0,\Omega}
		\,+\, 2\,\|\bgam - \bdel_h\|_{0,\Omega}\Big\}\, \Vert\bta_h\Vert_{0,\Omega}\,.
	\end{array}
\end{equation}

On the other hand, by replacing inequalities \eqref{eq:aah} and \eqref{eq:ooh} in the first and second rows of \eqref{eq:p4}, along with \eqref{eq:j1} and \eqref{eq:j2}, respectively, we obtain 
\begin{equation}\label{eq:k3}
\begin{array}{c}
	\big|\big(\mathscr{A}_{\varphi}^{S}-\mathscr{A}_{h,\varphi_h}^{S}\big)(\bsi, \bta_{h})\big|\,\leq\, \widehat{\mathcal{C}}_{\epsilon}\, \Vert \varphi-\varphi_h \Vert_{0,6;\Omega}\,\Vert\bta_{h}\Vert_{0,\Omega} + 
\ell_{\mathscr{A}^{S}}\, \big\|\bsi - \Pcalbb^h_r(\bsi)\big\|_{0,\Omega}\,\|\bta_h\|_{0,\Omega}\qan\\[2ex]
\big|\big(\mathscr{O}_{\varphi}^{S}-\mathscr{O}^{S}_{h,\varphi_h}\big)(\bu;\bu,\bta_{h})\big|\,\leq\,  \Big(\dfrac{\mathcal{L}_{\mu}}{\mu_1^{2}}\,\Vert \varphi-\varphi_h \Vert_{0,6;\Omega}\,\Vert\bu\Vert_{0,6;\Omega}^{2}\\[2ex]
\qquad\qquad+\,\dfrac{1}{\mu_{1}}\, \Vert (\bu\otimes\bu)-\Pcalbb^h_r(\bu\otimes\bu)\Vert_{0,\Omega}\Big)\,  \Vert\bta_{h}\Vert_{0,\Omega}\,.
\end{array}
\end{equation}

Finally, utilizing \eqref{eq:k1}, \eqref{eq:k2}, \eqref{a-a-h-bze-h-bta-h} and \eqref{eq:k3} back into \eqref{eq0:est.prob1}, and bounding $ \Vert\bu\Vert_{0,6;\Omega} $ and $ \Vert\bu_h\Vert_{0,6;\Omega} $ by $ r $ and $ r_{\mathtt{d}} $, respectively, we deduce the existence of a positive constant $ \mathcal{C}_{1,\mathtt{st}} $, depending only on $ \mathcal{C}_{\mathtt{st}}^S $, $ \| \mathscr{A}^{S}\| $, $ \mathcal{C}_{\mathcal{A}^{S}} $, $ \| \mathscr{O}^{S}\| $, $ \mathcal{C}_{\mathcal{O}^{S}} $, $ \ell_{\mathscr{A}^{S}} $, $ \ell_{\mathscr{B}^{S}} $, $ r $, $ r_{\mathtt{d}} $, $ \widehat{C}_{\epsilon} $, such that
\begin{equation}\label{eq:est.prob1}
	\begin{array}{c}
		\Vert (\bsi,\vec{\bu})-(\bsi_h,\vec{\bu}_h)\Vert_{\mathcal{V}}\,\leq\, \mathcal{C}_{1,\mathtt{st}}\, \Big\{\dist\big((\bsi,\vec{\bu}),\mathcal{V}_h\big) +\big\Vert\bsi-\Pcalbb^h_r(\bsi)\big\Vert_{0,\Omega}+\big\Vert\bgam-\Pcalbb^h_r(\bgam)\big\Vert_{0,\Omega}\\[2ex]
		\,+\,\big\Vert (\bu\otimes\bu)-\Pcalbb^h_r(\bu\otimes\bu)\big\Vert_{0,\Omega}
		+\Vert\bu\Vert_{0,6;\Omega}\, \Vert\bu-\bu_h\Vert_{0,6;\Omega}\\[2ex]
		\,+\,\left(1+\Vert\bg\Vert_{0,3/2;\Omega}+\Vert\bu\Vert_{0,6;\Omega}^{2}\right)\Vert\varphi-\varphi_h\Vert_{0,6;\Omega}
		\Big\}\,.
	\end{array}
\end{equation}

Our next step is to employ the Strang error estimate from \cite[Proposition 2.1, Corollary 2.3, and Theorem 2.3]{bcm-SIAMNA-1988} to the context provided by \eqref{eq:v1a-eq:v4-c}-\eqref{eq:v1a-eq:v4-d} and \eqref{eq:dis.ful.v1a-eq:v4-c}-\eqref{eq:dis.ful.v1a-eq:v4-d}, where each term that involves $ \mathscr{A}^{T}_{\varphi} $ and $ \mathscr{O}^{T}_{\varphi} $ is considered part of the respective functional on the right side.
Thus,
we infer the existence of a constant $ \mathcal{C}_{\mathtt{st}}^T >0$ depending only $ \| \mathscr{A}^{T}\| $, $ \| \mathscr{O}^{T}\|,\, \Vert\bu\Vert_{0,6;\Omega}  $, $ \mathcal{C}_{\mathcal{O}^{T}},\, \Vert\bu_h\Vert_{0,6;\Omega}  $, $ \| \mathscr{B}^{T}\| $ with reminding that $ \Vert\bu\Vert_{0,6;\Omega} $ and $ \Vert\bu_h\Vert_{0,6;\Omega} $ are bounded by $ r $ and $ r_{\mathtt{d}} $, espectively, and hence,  independent of $ h $,
\begin{equation}\label{eq0:est.prob2}
	\begin{array}{c}
		\big\Vert (\brho,\varphi)-(\brho_h,\varphi_h)\big\Vert_{\mathcal{W}}\,\leq\, \mathcal{C}_{\mathtt{st}}^T\, \Big\{\dist\big((\brho,\varphi),\mathcal{W}_h\big)\\[2ex]
		\,+\,\big\Vert \big(\mathscr{A}^{T}_{\varphi}-\mathscr{A}^{T}_{h,\varphi_h}\big)(\bxi_h,\cdot)\big\Vert_{\mathbf{X}_h^{'}}
		+\big\Vert \mathscr{O}^{T}_{\varphi}(\bu;\varpi_h,\cdot)- \mathscr{O}^{T}_{h,\varphi_h}(\bu_h;\varpi_h,\cdot)\big\Vert_{\mathbf{X}_h^{'}}
		\Big\}\,,
	\end{array}
\end{equation}

In turn, for the second term above by processing analogously to the previous case for derivations \eqref{eq:p1} and \eqref{eq:p4}, and making use of triangle inequality and boundedness property of $ \mathscr{A}^{T}_{\varphi} $ and $ \mathscr{A}^{T}_{h,\varphi} $, we find that
\begin{equation}\label{eq:l1}
\begin{array}{c}
		\big|\big(\mathscr{A}^{T}_{\varphi}-\mathscr{A}^{T}_{h,\varphi_h}\big)(\bxi_h,\bet_h)\big|\,\leq\, \big|\mathscr{A}^{T}_{\varphi}(\bxi_h-\brho, \bet_{h})\big| +\big|\mathscr{A}^{T}_{h,\varphi_h}(\brho-\bxi_h,\bet_{h})\big|\\[2ex]
		\,+\, \big|\big(\mathscr{A}^{T}_{\varphi}-\mathscr{A}^{T}_{h,\varphi_h}\big)(\brho, \bet_{h})\big|\\[2ex]
		\, \leq\, \left(\| \mathscr{A}^{T}\| + \mathcal{C}_{\mathcal{A}^{T}}\right)\, \Vert \bxi_h-\brho\Vert_{0,\Omega}\, \Vert \bet_{h}\Vert_{0,\Omega}\\[2ex]
		\,+\,  \big|\big(\mathscr{A}^{T}_{h,\varphi}-\mathscr{A}^{T}_{h,\varphi_h}\big)(\brho, \bet_{h})\big|+\big|\big(\mathscr{A}^{T}_{\varphi}-\mathscr{A}^{T}_{h,\varphi}\big)(\brho, \bet_{h})\big|
		\,,
\end{array}
\end{equation}

whereas for the third term, adding and subtracting $ \varphi $ in the second component of both $ \mathscr{O}^{T}_{\varphi}(\bu;\varpi_h,\cdot) $ and $ \mathscr{O}^{T}_{h,\varphi_h}(\bu_h;\varpi_h,\cdot) $, as
well as adding and subtracting $ \mathscr{O}^{T}_{h,\varphi_h}(\bu;\varpi,\cdot) $ to the resulting terms, and employing the boundedness properties of $ \mathscr{O}^{T}_{\varphi} $ and $ \mathscr{O}^{T}_{h,\varphi_h} $ with respective constants $ \| \mathscr{O}^{T}\| $ and $\mathcal{C}_{\mathcal{O}^{T}} $, yields
\begin{equation}\label{eq:i2}
	\begin{array}{c}
	\big|\mathscr{O}^{T}_{\varphi}(\bu;\varpi_h,\bet_h)- \mathscr{O}^{T}_{h,\varphi_h}(\bu_h;\varpi_h,\bet_h)\big|	\,\leq\, \big|\mathscr{O}^{T}_{\varphi}(\bu;\varpi_h -\varphi,\bet_h)\big|+\big|\mathscr{O}^{T}_{h,\varphi_h}(\bu_h;\varphi-\varpi_h,\bet_{h})\big|\\[2ex]
	\,+\,\big|\mathscr{O}^{T}_{h,\varphi_h}(\bu-\bu_h ;\varphi,\bet_{h})\big|+\big|\big(\mathscr{O}^{T}_{\varphi}-\mathscr{O}^{T}_{h,\varphi_h}\big)(\bu;\varphi,\bet_{h})\big|\\[2ex]
	\,\leq\, \Big(\big(\| \mathscr{O}^{T}\|\, \Vert\bu\Vert_{0,6;\Omega}+\mathcal{C}_{\mathcal{O}^{T}}\,\Vert\bu_h\Vert_{0,6;\Omega}\big)\, \Vert\varpi_h-\varphi\Vert_{0,6;\Omega}\\[2ex]
	\,+\,\mathcal{C}_{\mathcal{O}^{T}}\,\Vert\varphi\Vert_{0,6;\Omega}\,\Vert\bu-\bu_h\Vert_{0,6;\Omega}\Big)\Vert\bet_h\Vert_{0,\Omega}\\[2ex]
	\,+\,\big|\big(\mathscr{O}^{T}_{h,\varphi}-\mathscr{O}^{T}_{h,\varphi_h}\big)(\bu;\varphi,\bet_{h})\big|+\big|\big(\mathscr{O}^{T}_{\varphi}-\mathscr{O}^{T}_{h,\varphi}\big)(\bu;\varphi,\bet_{h})\big|\,.
	\end{array}
\end{equation}

Regarding the second term on the right hand of \eqref{eq:l1}, that is, $ \big(\mathscr{A}^{T}_{h,\varphi}-\mathscr{A}^{T}_{h,\varphi_h}\big) $,  we employ the definition of $ \mathscr{A}^{T}_{h,\varphi} $ (cf. \eqref{eq:defAT.h}), the Lipschitz continuity of $ \kappa $ (cf. second column of \eqref{eq:lip.mu.kap}), the Cauchy-Schwarz and H\"older's inequalities, and the reugularity assumption $ (\widetilde{\mathbf{RA}}) $ (cf. \eqref{eq:reg.ze}), to arrive at
\begin{equation}\label{eq:i1}
	\begin{array}{c}
		\big|\big(\mathscr{A}^{T}_{h,\varphi}-\mathscr{A}^{T}_{h,\varphi_h}\big)(\brho, \bet_{h})\big|
		\,\leq \,
		\disp\sum\limits_{E\in \Omega_h}\big(\mathscr{A}^{T,E}_{\varphi}-\mathscr{A}^{T,E}_{\varphi}\big)\left(\Pcalbf_{r}^{E}(\brho),\Pcalbf_{r}^{E}(\bet_h)\right)\\[2ex]
		\,+\, \disp\sum\limits_{E\in \Omega_h} \big(\mathscr{L}_{\varphi}^{T,E}-\mathscr{S}_{\varphi_h}^{T,E}\big)\left(\brho - \Pcalbf_{r}^{E}(\brho),\bta_{h} - \Pcalbf_{r}^{E}(\bet_{h})\right)\\[2.8ex]	
		\, \leq\, \dfrac{\mathcal{L}_{\kappa}}{\kappa_1^{2}}\,\widetilde{\mathcal{C}}_{\epsilon}\| i_\epsilon\|\, \big(1+\max\{\widetilde{c}_\star,\widetilde{c}^{\star}\}\big)\, \Vert \varphi-\varphi_h \Vert_{0,6;\Omega}\,\Vert \brho \Vert_{\epsilon,\Omega} \Vert\bet_{h}\Vert_{0,\Omega}\\[2.5ex]
		\, \leq\, \bar{\mathcal{C}}_{\epsilon}\,\Vert\varphi_{D}\Vert_{1/2+\epsilon,\Gamma_{D}}\, \Vert \varphi-\varphi_h \Vert_{0,6;\Omega}\,\Vert\bet_{h}\Vert_{0,\Omega}\,,
	\end{array}
\end{equation}

where $ \bar{\mathcal{C}}_{\epsilon} $ is a constant, depending on $ \mathcal{L}_{\kappa} $, $ \kappa_1 $, $ \widetilde{\mathrm{C}}_{\epsilon} $, $ \widetilde{c}_{\star} $, $ \widetilde{c}^{\star} $.
In addition, the term $ (\mathscr{O}^{T}_{h,\varphi}-\mathscr{O}^{T}_{h,\varphi_h}) $ in \eqref{eq:i2} can be estimated by bearing in mind the definition of $ \mathscr{O}^{T}_{h,\varphi} $,  and utilizing the H\"older's inequality, 
 as:
\begin{equation}\label{eq:i3}
	\begin{array}{c}
	\big|\big(\mathscr{O}^{T}_{h,\varphi}-\mathscr{O}^{T}_{h,\varphi_h}\big)(\bu;\varphi,\bet_{h})\big|\, \leq\, \dfrac{\mathcal{L}_{\kappa}}{\kappa_1^{2}}\, \Vert\varphi-\varphi_h\Vert_{0,6;\Omega}\, \Vert\bu\Vert_{0,6;\Omega}\, \Vert\varphi\Vert_{0,6;\Omega}\, \Vert\bet_{h}\Vert_{0,\Omega}\,.
	\end{array}
\end{equation}

Hence, replacing \eqref{eq:i1} and \eqref{eq:i3} back into \eqref{eq:l1} and \eqref{eq:i2}, respectivey, give
\begin{equation}\label{eq:i4}
	\begin{array}{c}
		\big|\big(\mathscr{A}^{T}_{\varphi}-\mathscr{A}^{T}_{h,\varphi_h}\big)(\bxi_h,\bet_h)\big|\,\leq\,
		 \left(\| \mathscr{A}^{T}\| + \mathcal{C}_{\mathcal{A}^{T}}\right)\, \Vert \bxi_h-\brho\Vert_{0,\Omega}\, \Vert \bet_{h}\Vert_{0,\Omega}\\[2ex]
		 \,+\,\bar{\mathcal{C}}_{\epsilon}\,\Vert\varphi_{D}\Vert_{1/2+\epsilon,\Gamma_{D}} \Vert \varphi-\varphi_h \Vert_{0,6;\Omega}\,\Vert\bet_{h}\Vert_{0,\Omega}\\[2ex]
		 \,+\,\big|\big(\mathscr{A}^{T}_{\varphi}-\mathscr{A}^{T}_{h,\varphi}\big)(\brho, \bet_{h})\big|\,,
	\end{array}
\end{equation}
and
\begin{equation}\label{eq:i5}
	\begin{array}{c}
		\big|\mathscr{O}^{T}_{\varphi}(\bu;\varpi_h,\bet_h)- \mathscr{O}^{T}_{h,\varphi_h}(\bu_h;\varpi_h,\bet_h)\big|\\[2ex]
			\,\leq\, \big(\| \mathscr{O}^{T}\|\, \Vert\bu\Vert_{0,6;\Omega}+\mathcal{C}_{\mathcal{O}^{T}}\,\Vert\bu_h\Vert_{0,6;\Omega}\big)\, \Vert\varpi_h-\varphi\Vert_{0,6;\Omega}\Vert\bet_{h}\Vert_{0,\Omega}\\[2ex]
		\,+\,\mathcal{C}_{\mathcal{O}^{T}}\,\Vert\varphi\Vert_{0,6;\Omega}\,\Vert\bu-\bu_h\Vert_{0,6;\Omega}\Vert\bet_{h}\Vert_{0,\Omega}+\dfrac{\mathcal{L}_{\kappa}}{\kappa_1^{2}}\, \Vert\varphi-\varphi_h\Vert_{0,6;\Omega}\, \Vert\bu\Vert_{0,6;\Omega}\, \Vert\varphi\Vert_{0,6;\Omega}\, \Vert\bet_{h}\Vert_{0,\Omega}\\[2.4ex]
		\,+\,\big|\big(\mathscr{O}^{T}_{\varphi}-\mathscr{O}^{T}_{h,\varphi}\big)(\bu;\varphi,\bet_{h})\big|\,.
	\end{array}
\end{equation}

We continue the analysis by providing the proper upper bounds for last terms involving in \eqref{eq:i4} and \eqref{eq:i5}. In fact, by following the strategies used in proof of Lemma \ref{l_A-l_C} (cf. \eqref{eq:aah} and \eqref{eq:ooh}), we conclude the existence of constants $ \ell_{\mathscr{A}^{T}}>0 $ and $ \ell_{\mathscr{O}^{T}}:=\dfrac{1}{\kappa_1} $,  such that there hold
\begin{equation*}
	\begin{array}{c}
		\big|\big(\mathscr{A}^{T}_{\varphi}-\mathscr{A}^{T}_{h,\varphi}\big)(\brho, \bet_{h})\big|\,\leq\, \ell_{\mathscr{A}^{T}}\, \|\brho - \Pcalbf^h_r(\brho)\|_{0,\Omega}\,\|\bet_h\|_{0,\Omega}\quad\forall\,\bet_h\in \mathbf{X}_h \,,\\[2ex]
		\big|\big(\mathscr{O}^{T}_{\varphi}-\mathscr{O}^{T}_{h,\varphi}\big)(\bu;\varphi,\bet_{h})\big|\,\leq\, \ell_{\mathscr{O}^{T}}\, \Vert (\bu\,\varphi)-\Pcalbf^h_r(\bu\,\varphi)\Vert_{0,\Omega}\,\|\bet_h\|_{0,\Omega}\quad\forall\,\bet_h\in \mathbf{X}_h  \,,
	\end{array}
\end{equation*}

which, replaced back in \eqref{eq:i4} and \eqref{eq:i5}, respectively, then combining the result estimate with \eqref{eq0:est.prob2}, and bounding terms $ \Vert\bu\Vert_{0,6;\Omega} $ and $ \Vert\bu_h\Vert_{0,6;\Omega} $ with $ r $ and $ r_{\mathtt{d}} $, respectively, yields 
\begin{equation}\label{eq:est.prob2}
		\begin{array}{c}
		\big\Vert (\brho,\varphi)-(\brho_h,\varphi_h)\big\Vert_{\mathcal{W}}\,\leq\, \mathcal{C}_{2,\mathtt{st}}\, \Big\{\dist\big((\brho,\varphi),\mathcal{W}_h\big)\\[2ex]
		\,+\,\left(\Vert\varphi_{D}\Vert_{1/2+\epsilon,\Gamma_{D}}+\Vert\bu\Vert_{0,6;\Omega}\,\Vert\varphi\Vert_{0,6;\Omega}\right)\Vert\varphi - \varphi_h\Vert_{0,6;\Omega}+\Vert\varphi\Vert_{0,6;\Omega}\,\Vert\bu-\bu_h\Vert_{0,6;\Omega}\\[2ex]
		\,+\,\|\brho - \Pcalbf^h_r(\brho)\|_{0,\Omega}\,+\,\Vert (\bu\,\varphi)-\Pcalbf^h_r(\bu\,\varphi)\Vert_{0,\Omega}
		\Big\}\,,
	\end{array}
\end{equation}
where $ \mathcal{C}_{2,\mathtt{st}} $ is a constant depending on $ \widetilde{\mathrm{C}}_{\mathtt{st}} $, $ \| \mathscr{A}^{T}\| $, $ \mathcal{C}_{\mathcal{A}^{T}} $, $ \| \mathscr{O}^{T}\| $, $ \mathcal{C}_{\mathcal{O}^{T}} $, $ \ell_{\mathscr{A}^{T}} $, $ \ell_{\mathscr{O}^{T}} $, $ r $, $ r_{\mathtt{d}} $, $ \bar{C}_{\epsilon} $.

For the rest of the analysis, we suitably combine the estimates \eqref{eq:est.prob1} and \eqref{eq:est.prob2}.
More precisely,
multiplying \eqref{eq:est.prob1} by $ \frac{1}{4 \mathcal{C}_{1,\mathtt{st}}} $, adding the resulting inequality and \eqref{eq:est.prob2}, and conducting some algebraic computations, yield
\begin{equation}\label{eq:l2}
	\begin{array}{c}
\dfrac{1}{4 \mathcal{C}_{1,\mathtt{st}}}	\big\Vert (\bsi,\vec{\bu})-(\bsi_h,\vec{\bu}_h)\big\Vert_{\mathcal{V}}+\big\Vert (\brho,\varphi)-(\brho_h,\varphi_h)\big\Vert_{\mathcal{W}}\\[2ex]
\,\leq\,\dfrac{1}{4 \mathcal{C}_{1,\mathtt{st}}}\,\dist\big((\bsi,\vec{\bu}),\mathcal{V}_h\big)+\mathcal{C}_{2,\mathtt{st}}\,\dist\big((\brho,\varphi),\mathcal{W}_h\big)\\[2ex]
\,+\,\dfrac{1}{4}\Big(\big\Vert\bsi-\Pcalbb^h_r(\bsi)\big\Vert_{0,\Omega}+\big\Vert\bgam-\Pcalbb^h_r(\bgam)\big\Vert_{0,\Omega}
\,+\,\big\Vert (\bu\otimes\bu)-\Pcalbb^h_r(\bu\otimes\bu)\big\Vert_{0,\Omega}\Big)\\[2.5ex]
\,+\,\mathcal{C}_{2,\mathtt{st}}\,\Big(\|\brho - \Pcalbf^h_r(\brho)\|_{0,\Omega}\,+\,\Vert (\bu\,\varphi)-\Pcalbf^h_r(\bu\,\varphi)\Vert_{0,\Omega}\Big)\\[2ex]
\,+\,\Big(\dfrac{1}{4}\Vert\bu\Vert_{0,6;\Omega}+\mathcal{C}_{2,\mathtt{st}}\,\Vert\varphi\Vert_{0,6;\Omega}\Big)\, \Vert\bu-\bu_h\Vert_{0,6;\Omega}\\[2ex]
\,+\,\Big(\dfrac{1}{4}\,\left(1+\Vert\bg\Vert_{0,3/2;\Omega}+\Vert\bu\Vert_{0,6;\Omega}^{2}\right)\\[2ex]
\,+\,\mathcal{C}_{2,\mathtt{st}}\,\left(\Vert\varphi_{D}\Vert_{1/2+\epsilon,\Gamma_{D}}+\Vert\bu\Vert_{0,6;\Omega}\,\Vert\varphi\Vert_{0,6;\Omega}\right)\Big)\Vert\varphi-\varphi_h\Vert_{0,6;\Omega}\,.
	\end{array}
\end{equation}
Consequently, thanks to \eqref{eq:l2} we are in position to establish the announced Céa estimate.
\begin{theorem}\label{theorem-Cea}
	Assume that the data satisfy
	\begin{equation}\label{eq:as1.conv}
		\begin{array}{c}
			\dfrac{1}{4}\Big(1+\dfrac{16}{\alpha_{\mathcal{A}^S}^{2}\alpha_{\mathcal{A}^T}^{2}}\Vert\bg\Vert_{0,3/2;\Omega}\,\Vert\varphi_{D}\Vert_{1/2,\Gamma_{D}}^{2}\Big)\Vert\bg\Vert_{0,3/2;\Omega}\\[2ex]
			\,+\,\mathcal{C}_{2,\mathtt{st}}\,\Big(\Vert\varphi_{D}\Vert_{1/2+\epsilon,\Gamma_{D}}+\dfrac{8}{\alpha_{\mathcal{A}^S}\alpha_{\mathcal{A}^T}^{2}}\Vert\bg\Vert_{0,3/2;\Omega}\Vert\varphi_{D}\Vert_{1/2,\Gamma_{D}}^{2}\Big)\, \leq\, \dfrac{1}{4}\,,
		\end{array}
	\end{equation}
and
\begin{equation}\label{eq:as2.conv}	\dfrac{1}{\alpha_{\mathcal{A}^S}\alpha_{\mathcal{A}^T}}\,\Vert\varphi_{D}\Vert_{1/2,\Gamma_{D}}\Vert\bg\Vert_{0,3/2;\Omega}\,+\,\mathcal{C}_{2,\mathtt{st}}\,\dfrac{1}{\alpha_{\mathcal{A}^T}}\,\Vert\varphi_{D}\Vert_{1/2,\Gamma_{D}}\,\leq\, \dfrac{1}{8\mathcal{C}_{1,\mathtt{st}}}\,.
\end{equation}
Then, there exists a constant $ \widehat{\mathcal{C}}_{\mathtt{st}} >0$, depending only on $ \mathcal{C}_{1,\mathtt{st}} $  and $ \mathcal{C}_{2,\mathtt{st}} $, and hence, independent of $ h $, such that
\begin{equation*}
	\begin{array}{c}
	\big\Vert (\bsi,\vec{\bu})-(\bsi_h,\vec{\bu}_h)\big\Vert_{\mathcal{V}}+\big\Vert (\brho,\varphi)-(\brho_h,\varphi_h)\big\Vert_{\mathcal{W}}\\[2ex]
	\, \leq\, \widehat{\mathcal{C}}_{\mathtt{st}}\, \bigg\{\dist\big((\bsi,\vec{\bu}),\mathcal{V}_h\big)\,+\,\dist\big((\brho,\varphi),\mathcal{W}_h\big) \\[2ex]
	\,+\,\Vert\bsi-\Pcalbb^h_r(\bsi)\Vert_{0,\Omega}+\Vert\bgam-\Pcalbb^h_r(\bgam)\Vert_{0,\Omega}
	\,+\,\Vert (\bu\otimes\bu)-\Pcalbb^h_r(\bu\otimes\bu)\Vert_{0,\Omega}\\[2ex]
	\,+\,\|\brho - \Pcalbf^h_r(\brho)\|_{0,\Omega}\,+\,\Vert (\bu\,\varphi)-\Pcalbf^h_r(\bu\,\varphi)\Vert_{0,\Omega}
	 \bigg\}\,.
	\end{array}
\end{equation*}
\end{theorem}
\begin{proof}
The result follows straightly from \eqref{eq:l2} after bounding $ \Vert\bu\Vert_{0,6;\Omega} $ and $ \Vert\varphi\Vert_{0,6;\Omega} $ by \eqref{eq:apr1.prob1} and \eqref{eq:apr2.prob2}, respectively, and realizing that the last term on the right-hand side can be subtracted from the second term on the left-hand side under assumption \eqref{eq:as1.conv}, while the analogous procedure can be applied to the corresponding fifth and first terms under \eqref{eq:as2.conv}.
\end{proof}
\subsection{The rates of convergence}
 \subsubsection{Interpolation estimates}
In order to define an interpolation operators in the local spaces $ \mathbf X_r(E) $ and $ \mathbb X_r(E) $, for each element $ E\in \Omega_{h} $ we denote by
$ \widetilde{\bchi}_i^{E} $ and $ \bchi_i^{E} $ the operator associated to the $ i $-th local degree of freedom, $ i=1,\cdots, \widetilde n^E_r $ and $ i=1,\cdots, n^E_r $, respectively. From the
definition of these spaces, it is easily seen that for every smooth enough function  $\bet \in \mathbf W^{1,1}(E)$ and for each $\bta \in \mathbb W^{1,1}(E)$ there exists a unique elements $ \bet_{I}^{E}\in \mathbf X_r(E)  $ and $ \bta_{I}^{E}\in \mathbb X_r(E)  $ such that
\begin{equation}\label{def-interpolation-bf-Pi-K-r-bb-Pi-K-r}
	\begin{array}{c}
		\widetilde{\bchi}_i^{E}\big(\bet-\bet_{I}^{E}\big)\,=\,0\quad\forall\, i \,=\,1,\cdots, \widetilde n^E_r \qan
		\bchi_i^{E}\big(\bta-\bta_{I}^{E}\big)\,=\,0\quad\forall\, i \,=\,1,\cdots, n^E_r \,,
	\end{array}
\end{equation}
Let $ C $ (or $ \widetilde{C} $, $ C_{1} $, $ \widetilde{C}_1 $) denote a generic constant independent of the mesh parameter, the values of which may be different at different places. Following the discussion of \cite{cg-IMANUM-2017} (also
see \cite{gs-M3AS-2021,cgs-SIAM.NA-2018}), we have the following interpolation error estimates
\begin{equation}\label{ap-Pibf-0-r}
	\|\bet - \bet_{I}^{E}\|_{0,E} \,\le\, C \, h^s_E\,|\bet|_{s,E} \qquad\forall\,\bet \in \mathbf H^s(E)\,,
\end{equation}
and
\begin{equation}\label{ap-Pibb-0-r}
	\|\bta - \bta_I^{E}\|_{0,E} \,\le\, \widetilde C\, h^s_E\,|\bta|_{s,E} \qquad\forall\,\bta \in \mathbb H^s(E)\,, 
\end{equation}

In addition, from Lemma \ref{l1} and the following commutative properties
\begin{equation}\label{eq:comm.prop}
	\begin{array}{c}
		\div(\bet_{I}^{E})\,=\, \Pcalbf_{r}^E (\div(\bet)) \quad \forall\,\bet \in \mathbf W^{1,1}(E) \qan \\[2ex]
		\bdiv(\bta_{I}^{E})\,=\, \Pcalbb_{r}^E (\bdiv(\bta))\quad \forall\,\bta \in \mathbb W^{1,1}(E)\,,
	\end{array}
\end{equation}

we deduce, for each $ \bet \in \mathbf W^{1,1}(E) $ and $ \bta \in \mathbb W^{1,1}(E) $ such that $ \div(\bet)\in \mathrm{W}^{s,t}(E) $ and $ \bdiv(\bta)\in\mathbf{W}^{s,q}(E) $, with $s\in [0,r+1]$, there holds (see, e.g., \cite[eq. (3.14)]{gs-M3AS-2021})
\begin{equation}\label{ap-div-Pibf-0-r}
	\|\div\big(\bet - \bet_I^{E}\big)\|_{0,t;E} \,\le\, C_{1}\, h^s_E\,|\div(\bet)|_{s,t;E} \,,
\end{equation}
and 
\begin{equation}\label{ap-bdiv-Pibb-0-r}
	\|\bdiv\big(\bta - \bta_I^{E}\big)\|_{0,q;E} \,\le\, \widetilde C_{1}\, h^s_E\,|\bdiv(\bta)|_{s,q;E} \,.
\end{equation}

Then, given a subspace $V_h$ of an arbitrary Banach space $\big(V,\|\cdot\|_V\big)$, we set 
from now on
\begin{equation*}\label{dist-z-Z-h}
	\mathrm{dist}(v,V_h) \,:=\, \inf_{v_h\in V_h} \|v - v_h\|_V \qquad \forall\, v\in V \,.
\end{equation*}
Hence, according to the local approximation properties provided by the estimates 
\eqref{ap-Pibf-0-r} up to \eqref{ap-bdiv-Pibb-0-r}, and Lemma \ref{l1}, we easily 
derive the following global ones:

\begin{itemize}[leftmargin=.7in]
	\item[$(\mathbf{AP}_{h}^{\bsi})$]
	for each integer $s\in [1,r+1]$ there exists a positive constant $C$, independent of $h$, such that
	\begin{equation*}
		\text{dist}(\bta,\mathbb{X}_{h}) \,\le\, \|\bta - \bta_I\|_{\bdiv_{6/5};\Omega}\,\le\,
		C \, h^{s} \, \Big\{|\bta|_{s,\Omega} + |\bdiv(\bta)|_{s,6/5;\Omega}\Big\}\,,
	\end{equation*}
	for all $\bta \in \mathbb X$ such that $\bta \in \mathbb{H}^{s}(\Omega)$ and 
	$\bdiv(\bta) \in \mathbf{W}^{s,6/5}(\Omega)$,
	
	\item[$(\mathbf{AP}_{h}^{\bu})$]
	for each integer $s\in [0,r+1]$ there exists a positive constant $C$, independent of $h$, such that
	\begin{equation*}
		\text{dist}(\boldsymbol{\bv},\mathbf{Y}_{h}) \,\le\, \|\bv - \Pcalbf_r^h(\bv)\|_{0,6;\Omega}\,\le\,
		C \, h^{s} \, |\bv|_{s,6;\Omega}\,,
	\end{equation*}
	for all $\bv \in \mathbf Y$ such that $\bv \in \mathbf W^{s,6}(\Omega)$,
	
	\item[$(\mathbf{AP}_{h}^{\bgam})$]
	for each integer $s\in [0,r+1]$ there exists a positive constant $C$, independent of $h$, such that
	\begin{equation*}
		\text{dist}(\bom,\mathbb{Z}_{h}) \,\le\, \|\bom - \Pcalbb_r^h(\bom)\|_{0,\Omega}\,\le\,
		C \, h^{s} \, |\bom|_{s,\Omega}\,,
	\end{equation*}
	for all $\bom \in \mathbb Z$ such that $\bom \in \mathbb H^{s}(\Omega)$,
	
	\item[$(\mathbf{AP}_{h}^{\brho})$]
	for each integer $s\in [1,r+1]$ there exists a positive constant $C$, independent of $h$, such that
	\begin{equation*}
		\text{dist}(\bet,\mathbf{X}_{h}) \,\le\, \|\bet - \bet_I\|_{\div_{6/5};\Omega}\,\le\,
		C \, h^{s} \, \Big\{|\bet|_{s,\Omega} + |\div(\bet)|_{s,6/5;\Omega}\Big\}\,,
	\end{equation*}
	for all $\bet \in \mathbf X$ such that $\bet \in \mathbf{H}^{s}(\Omega)$ and 
	$\div(\bet) \in \mathrm{W}^{s,6/5}(\Omega)$, and
	
	\item[ $(\mathbf{AP}_{h}^{\varphi})$]
	for each integer $s\in [0,r+1]$ there exists a positive constant $C$, independent of $h$, such that
	\begin{equation*}
		\text{dist}(\psi,\mathrm{Y}_{h}) \,\le\, \|\psi - \Pcal_r^h(\psi)\|_{0,6;\Omega}\,\le\,
		C \, h^{s} \, |\varphi|_{s,6;\Omega}\,,
	\end{equation*}
	for all $\psi \in \mathrm Y$ such that $\psi \in \mathrm W^{s,6}(\Omega)$.
\end{itemize}

We are now in a position to provide the rates of convergence of the virtual scheme \eqref{eq:dis.ful.v1a-eq:v4-a}-\eqref{eq:dis.ful.v1a-eq:v4-d}.
\begin{theorem}\label{theorem-rates-of-convergence}
	Assume that the data satisfy \eqref{eq:as1.conv} and \eqref{eq:as2.conv}, and let $ (\bsi,\vec{\bu})\in \mathbb{X}\times\mathcal{Z} $, $ (\brho,\varphi)\in \mathbf{X}\times\mathrm{Y} $ and 
	$ (\bsi_h,\vec{\bu}_h)\in \mathbb{X}_h\times\mathcal{Z}_h $, $ (\brho_h,\varphi_h)\in \mathbf{X}_h\times\mathrm{Y}_h $ be the solutions of \eqref{eq:v1a-eq:v4-a}-\eqref{eq:v1a-eq:v4-d} and \eqref{eq:dis.ful.v1a-eq:v4-a}-\eqref{eq:dis.ful.v1a-eq:v4-d}, respectively, with $ (\bu,\varphi)\in \mathbf{W}_r $ and $ (\bu_h,\varphi_h)\in \mathbf{W}_{r_{\mathtt{d}}} $, whose existences are guaranteed by Theorems \ref{t:exis} and \ref{t:dis.exis}, respectively.
	Furthermore, given an integer $ r\geq 0 $, assume that there exists $k \in [0,r+1]$ and $s \in [1,r+1]$ such that
	\[
	\bsi \in \mathbb H^s(\Omega),\quad \bdiv(\bsi) \in \mathbf W^{s,6/5}(\Omega), \quad  \bu \in \mathbf W^{k,6}(\Omega),\quad\bu \otimes \bu \in \mathbb H^k(\Omega),\quad
	\bgam\in \mathbb H^k(\Omega)\,,
	\]
	and
	\[
	\brho \in \mathbf H^s(\Omega), \quad\div(\brho) \in \mathrm W^{s,6/5}(\Omega), \quad\varphi \in \mathrm W^{k,6}(\Omega), \quad\varphi\,\bu \in \mathbf H^k(\Omega)\,. 
	\]
	Then, there exists a positive constant $C$, independent of $h$, such that
	\begin{equation*}\label{eq-rates-of-convergence}
		\begin{array}{c}
			\disp
			\|(\bsi,\vec{\bu}) - (\bsi_h,\vec{\bu}_h)\|_{\mathcal V} + \|(\brho,\varphi) - (\brho_h , \varphi_h)\|_{\mathcal W}
			\,\le\, C\,h^{\min\{k,s\}} \,\Big\{ |\bsi|_{s,\Omega} \,+\, |\bdiv(\bsi)|_{s,6/5;\Omega} 
			\,+\,  |\bu|_{k,6;\Omega}\,+\, |\bgam|_{k,\Omega} \\[2ex]
			\disp
			+ \,\, |\brho|_{s,\Omega} \,+ \, |\div(\brho)|_{s,6/5;\Omega} \,+\, |\varphi|_{s,6;\Omega} 
			\,+\, |\bu \otimes \bu|_{k,\Omega} \,+\, |\varphi\,\bu|_{k,\Omega}\Big\} \,.
		\end{array}
	\end{equation*} 
\end{theorem}
\begin{proof}
First, an application of approximation properties given by \eqref{eq_Q0} (cf. Lemma \ref{l1}) with taking $m = 0$, gives
	\begin{equation}\label{varias}
		\begin{array}{rcl}
			\disp
			\|\bsi - \Pcalbb^h_r(\bsi)\|_{0,\Omega} &\le& C_r\, h^s\,|\bsi|_{s,\Omega} \,, \\[1ex]
			\disp
			\|\bgam - \Pcalbb^h_r(\bgam)\|_{0,\Omega} &\le& C_r\, h^k\,|\bgam|_{k,\Omega} \,, \\[1ex]
			\disp
			\|(\bu \otimes \bu) - \Pcalbb^h_r(\bu \otimes \bu)\|_{0,\Omega}  &\le& C_r\, h^k\, |\bu\otimes\bu|_{k,\Omega} \,, \\[1ex]
			\disp
			\|\brho - \Pcalbf^h_r(\brho)\|_{0,\Omega} &\le& C_r\,h^s\, |\brho|_{s,\Omega} \,, \qanf \\[1ex]
			\disp
			\|(\varphi \, \bu) - \Pcalbf^h_r(\varphi \, \bu)\|_{0,\Omega} &\le&  C_r\,h^k\, |\varphi\,\bu|_{k,\Omega}\,.
		\end{array}
	\end{equation}
Thus, the proof proceeds directly from Theorem \ref{theorem-Cea} by involving the approximation properties of ({\bf AP}$^\bsi_h$), ({\bf AP}$^\bu_h$), ({\bf AP}$^\bgam_h$),
	({\bf AP}$^\brho_h$), and ({\bf AP}$^\varphi_h$) (cf. Section \ref{section32}), along with the estimates provided by \eqref{varias}.
\end{proof}
		\newpage					
	\section{Numerical Results}\label{sec.7}
	In this section, we present several numerical experiments in order to verify the theoretical
	results and illustrate the performance of the proposed method.
	The implementation
	of the numerical method produced by an
	in-house MATLAB code, using sparse factorisation as linear solver. A Newton-Raphson algorithm with a
	fixed tolerance tol 1e-6, is used for the resolution of the nonlinear problem \eqref{eq:dis.ful.v1a-eq:v4-a}-\eqref{eq:dis.ful.v1a-eq:v4-d}. As usual, the
	iterative method is finished when the $\ell^2$-norm of the global incremental discrete 
	solutions drop below that value.
	In addition, errors between exact and approximate solutions relevant to the norms used in the analysis of
	Section \ref{sec.6} are denoted as

\begin{equation*}
		\begin{array}{c}
		{\tt e}(\bsi) \,:=\, \Vert\bsi-\widehat{\bsi}_h
		\Vert_{\bdiv_{6/5};\Omega}\,, \,\,\,
		{\tt e}(\bu) \,:=\, \Vert\bu-\bu_h\Vert_{0,6;\Omega}\,,\,\,\,
		{\tt e}(\bgam) \,:=\, \Vert \bgam-\bgam_h\Vert_{0,\Omega}\,,
		 \,\,\,
		{\tt e}(p) \,:=\, \Vert p-\widehat{p}_h\Vert_{0,\Omega}, \\[2ex]	
		{\tt e}(\brho) \,:=\, \Vert\brho-\widehat{\brho}_h\Vert_{\div_{6/5};\Omega}\,, \qan {\tt e}(\varphi) \,:=\, \Vert\varphi-\varphi_h\Vert_{0,6;\Omega}
		\,.
	\end{array}
\end{equation*}

	In turn, for all $ \star\in \{\boldsymbol{\sigma},\mathbf{u},\bgam, p,\boldsymbol{\rho},\varphi\} $, 
	we let ${\tt r}(\star)~:=~\dfrac{\text{log}({\tt e}(\star)/{\tt e}'(\star))}{\text{log}(h/h')}$ be the experimental 
	rates of convergence, where $ h $ and $ h^{'} $ denote two consecutive mesh sizes with 
	errors ${\tt e}(\star)$ and ${\tt e}^{'}(\star)$, respectively. 
	
	\subsection{Example 1 (\cite{ag-CMAM-2020}): smooth exact solution}
\begin{figure}[t!]
	\begin{tabular}{ccc}
		\centering
		\hspace{-1.6cm}\includegraphics[width=.45\textwidth]{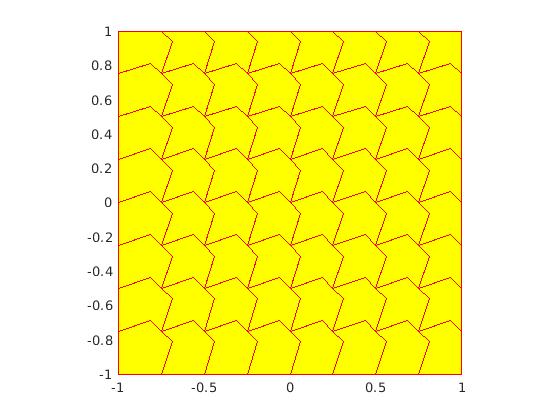} 
		\hspace{-1.4cm}\includegraphics[width=.45\textwidth]{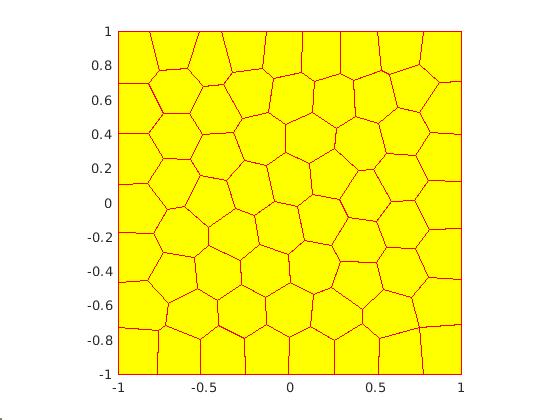} 
		\hspace{-1.4cm}\includegraphics[width=.45\textwidth]{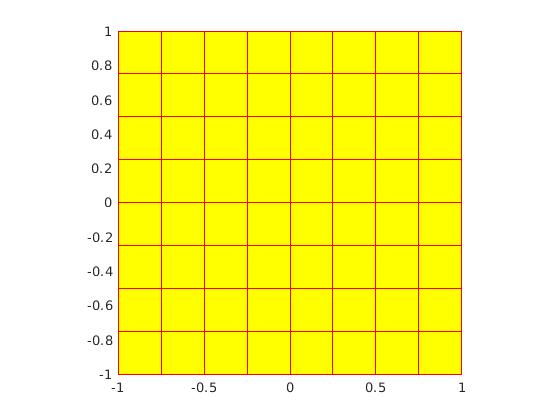}
	\end{tabular}
	\caption{Example 1, illustration of the meshes used: non-convex mesh (left panel), hexagon mesh (central panel)
		and  quadrilateral mesh (right panel).}\label{figure-1}
\end{figure}
In this test, we corroborate the rates of convergence for our problem with a smooth exact solution. For this purpose, we choose the viscosity, thermal conductivity, and body force given by
	\[
	\mu(\varphi)\, =\, \exp(-0.25 \,\varphi)\,, \quad \kappa(\varphi)\, =\, \exp(0.25\,\varphi)\,\qan \bg \, =\, (0,1)^{t}\,,
	\]
	
	and consider the following closed-form exact solution to the problem \eqref{eq:prob1}-\eqref{eq:prob5a} 
	\[
	\mathbf{u}(\textbf{x}) \,:=\, 
	\begin{pmatrix} 
		2y\sin(\pi x)\sin(\pi y)(x^2-1)+\pi \sin(\pi x)\cos(\pi y)(x^2-1)(y^2-1)
		\\[2mm]
		-2x\sin(\pi x)\sin(\pi y)(y^2-1)-\pi \sin(\pi y)\cos(\pi x)(x^2-1)(y^2-1) \end{pmatrix} \,, \quad
	p(\textbf{x}) \,:=\,  x^{2}-y^{2} \,,
	\]
	and
	\[
	\varphi(\textbf{x})\, = \, (x^2-1)(y^2-1)\,,
	\]
	defined over the computational domain $ \textbf{x} \,:=\, (x , y)^{\mathtt{t}}\in\Omega:=(-1,1)^{2} $.
	Table \ref{Tab1} shows
	the convergence history for sequences of successively refined meshes made of non-convex, hexagons, quadrilaterals elements (see Fig. \ref{figure-1}). The experiments confirm the theoretical rate of convergence
	$O(h^{r+1})$ for $r\in \big\{0,1\big\}$, provided by Theorem \ref{theorem-rates-of-convergence} for all the unknowns. 
One can be seen from Table \ref{Tab1} that errors of all unknowns for the second mesh (i.e., hexagon case) are more accurate than other cases.
	Furthermore, in order to illustrate the accurateness of the discrete scheme, in Figure \ref{fig1} we 
	display some components of the approximate solution obtained with the polynomial degree 
	$r = 0$.
		\begin{sidewaystable}
			\caption{Example 1, history of convergence using  non-convex, hexagons, and quadrilaterals meshes denoted by $ \mathcal{T}_h $, $ \mathcal{K}_h $ and $ \mathcal{R}_h$, respectively.}\label{Tab1}
			\bigskip	
			\centerline{
				{\footnotesize\begin{tabular}{|cccccccccccccc|}
						\hline\hline
						$r$ & $h$ & $ \mathtt{e}(\bsi) $ & $ {\tt r}(\bsi) $ & $\mathtt{e}(\bu)$ & ${\tt r}(\bu)$ & 
						$\mathtt{e}(\bgam)$ & ${\tt r}(\bgam)$ & $\mathtt{e}(p)$ & ${\tt r}(p) $ &
						$\mathtt{e}(\brho)$ & ${\tt r}(\brho)$ &
						$\mathtt{e}(\varphi)$ & ${\tt r}(\varphi)$ \\
						\hline\hline
						&\cellcolor{purple} $\mathcal{T}_{h} $& & &	&&&&&&&&& \\[1mm]
						\hline
						{}& {2.500e-01}& {6.972e-01}& {-}& {7.601e-01}& {-}& {8.174e-01}& {-}& {2.510e+00}& {-}& {5.519e-01}& {-}& {4.190e-01}& {-}\\[1mm]
						
						{}& {1.250e-01}& {4.066e-01}& {0.777}& {4.238e-01}& {0.842}& {4.832e-01}& {\cellcolor{yellow}0.758}& {1.353e+00}& {0.891}& {3.032e-01}& {0.863}& {1.768e-01}& {1.244}\\[1mm]	
						
						{0}& {6.250e-02}& {2.054e-01}& {0.985}& {2.141e-01}& {0.984}& {2.367e-01}& {1.029}& {6.326e-01}& {1.096}& {1.465e-01}& {1.049}& {8.478e-02}& {1.060}\\[1mm]	
						
						{}& {3.125e-02}& {1.001e-01}& {1.037}& {1.061e-01}& {1.013}& {1.148e-01}& {1.043}& {2.299e-01}& {1.460}& {7.092e-02}& {1.046}& {4.183e-02}& {1.019}\\[1mm]					
						\hline
						
						{}& {2.500e-01}& {3.928e-01}& {-}& {3.940e-01}& {-}& {4.286e-01}& {-}& {2.843e+00}& {-}& {2.544e-01}& {-}& {7.559e-02}& {-}\\[1mm]
						
						{}& {1.250e-01}& {1.166e-01}& {1.751}& {1.028e-01}& {1.938}& {1.183e-01}& {1.857}& {4.405e-01}& {2.690}& {6.928e-02}& {1.876}& {1.749e-02}& {2.111}\\[1mm]
						
						{1}& {6.250e-02}& {3.106e-02}& {1.908}& {2.624e-02}& {1.969}& {3.249e-02}& {1.864}& {1.071e-01}& {2.040}& {1.872e-02}& {1.888}& {4.180e-03}& {2.064}\\[1mm]
						
						{}& {3.125e-02}& {7.924e-03}& {1.970}& {6.589e-03}& {1.993}& {8.601e-03}& {\cellcolor{cyan}1.917}& {2.685e-02}& {1.996}& {4.805e-03}& {1.961}& {1.030e-03}& {2.020}\\[1mm]
						\hline
						&\cellcolor{purple} $\mathcal{K}_{h} $& & &	&&&&&&&&& \\[1mm]
						\hline
						{}& {2.500e-01}& {\cellcolor{orange}6.912e-01}& {-}& {7.733e-01}& {-}& {8.021e-01}& {-}& {1.826e+00}& {-}& {5.205e-01}& {-}& {4.151e-01}& {-}\\[1mm]
						
						{}& {1.250e-01}& {\cellcolor{orange}3.816e-01}& {0.856}& {4.119e-01}& {0.908}& {4.879e-01}& {\cellcolor{yellow}0.717}& {1.056e+00}& {0.790}& {2.707e-01}& {0.942}& {1.702e-01}& {1.286}\\[1mm]	
						
						{0}& {6.250e-02}& {\cellcolor{orange}1.919e-01}& {0.992}& {2.048e-01}& {1.008}& {2.265e-01}& {1.107}& {4.460e-01}& {1.243}& {1.357e-01}& {0.997}& {8.148e-02}& {1.062}\\[1mm]	
						
						{}& {3.125e-02}& {\cellcolor{orange}9.403e-02}& {1.028}& {1.012e-01}& {1.017}& {1.104e-01}& {1.037}& {1.527e-01}& {1.546}& {6.733e-02}& {1.010}& {3.994e-02}& {1.028}\\[1mm]
						
						\hline
						{}& {2.500e-01}& {3.505e-01}& {-}& {3.679e-01}& {-}& {3.956e-01}& {-}& {2.449e+00}& {-}& {2.262e-01}& {-}& {7.799e-02}& {-}\\[1mm]
						
						{}& {1.250e-01}& {9.628e-02}& {1.863}& {8.854e-02}& {2.054}& {9.571e-02}& {2.047}& {2.603e-01}& {3.234}& {5.409e-02}& {2.064}& {1.525e-02}& {2.354}\\[1mm]
						
						{1}& {6.250e-02}& {2.454e-02}& {1.972}& {2.195e-02}& {2.011}& {2.363e-02}& {2.017}& {4.025e-02}& {2.693}& {1.368e-02}& {1.983}& {3.509e-03}& {2.119}\\[1mm]
						
						{}& {3.125e-02}& {6.013e-03}& {2.029}& {5.389e-03}& {2.026}& {5.745e-03}& {\cellcolor{cyan}2.040}& {5.895e-03}& {2.771}& {3.371e-03}& {2.020}& {8.423e-04}& {2.058}\\[1mm]
						
						\hline
						&\cellcolor{purple} $\mathcal{R}_{h} $& & &	&&&&&&&&& \\[1mm]
						\hline
						{}& {2.500e-01}& {8.811e-01}& {-}& {1.327e+00}& {-}& {2.840e+00}& {-}& {1.531e+01}& {-}& {5.328e-01}& {-}& {5.203e-01}& {-}\\[1mm]
						
						{}& {1.250e-01}& {4.496e-01}& {0.970}& {5.498e-01}& {1.270}& {6.958e-01}& {\cellcolor{yellow}2.029}& {4.507e+00}& {1.764}& {2.798e-01}& {0.929}& {1.941e-01}& {1.422}\\[1mm]	
						
						{0}& {6.250e-02}& {2.053e-01}& {1.130}& {2.297e-01}& {1.259}& {2.568e-01}& {1.437}& {1.288e+00}& {1.807}& {1.381e-01}& {1.018}& {8.479e-02}& {1.194}\\[1mm]	
						
						{}& {3.125e-02}& {9.735e-02}& {1.076}& {1.059e-01}& {1.117}& {1.145e-01}& {1.165}& {3.570e-01}& {1.850}& {6.811e-02}& {1.020}& {4.086e-02}& {1.053}\\[1mm]
						
						\hline
						{}& {2.500e-01}& {3.633e-01}& {-}& {3.845e-01}& {-}& {4.250e-01}& {-}& {2.514e+00}& {-}& {2.381e-01}& {-}& {1.005e-01}& {-}\\[1mm]
						
						{}& {1.250e-01}& {1.006e-01}& {1.853}& {9.814e-02}& {1.970}& {1.164e-01}& {1.868}& {2.630e-01}& {3.256}& {5.887e-02}& {2.016}& {1.951e-02}& {2.364}\\[1mm]
						
						{1}& {6.250e-02}& {2.579e-02}& {1.962}& {2.429e-02}& {2.014}& {2.853e-02}& {2.028}& {5.183e-02}& {2.343}& {1.516e-02}& {1.957}& {4.314e-03}& {2.176}\\[1mm]
						
						{}& {3.125e-02}& {6.502e-03}& {1.987}& {5.987e-03}& {2.020}& {6.780e-03}& {\cellcolor{cyan}2.073}& {1.104e-02}& {2.231}& {3.758e-03}& {2.011}& {9.860e-04}& {2.129}\\[1mm]
						
						\hline
			\end{tabular}}}
		\end{sidewaystable}
	\begin{figure}[t!]
		\centering
		\begin{tabular}{ccc}
			\hspace{-1.5cm}	\includegraphics[width=.41\textwidth]{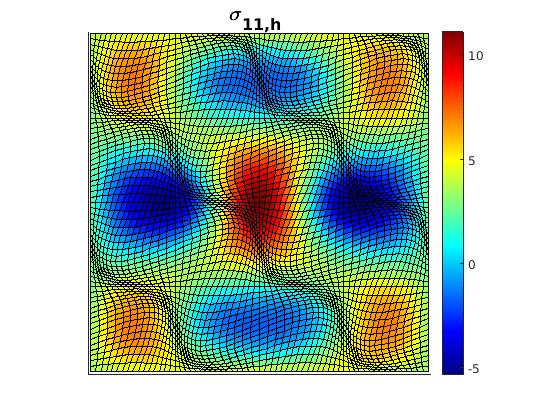}
			\hspace{-1.1cm}	\includegraphics[width=.41\textwidth]{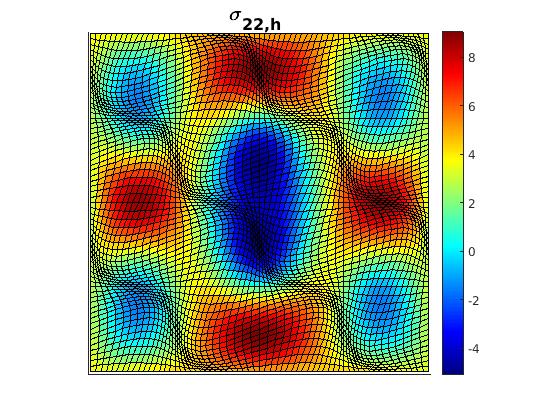}
			\hspace{-1.1cm}	\includegraphics[width=.41\textwidth]{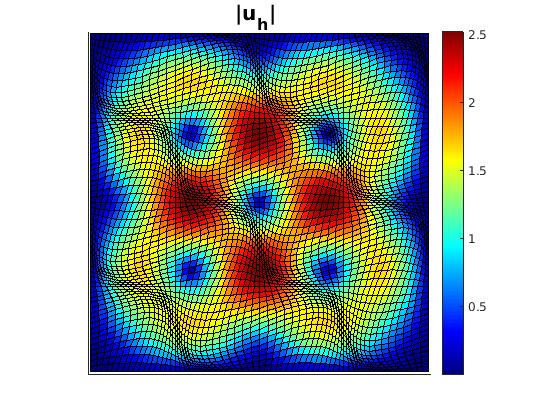}\\
			\hspace{-1.5cm} \includegraphics[width=.41\textwidth]{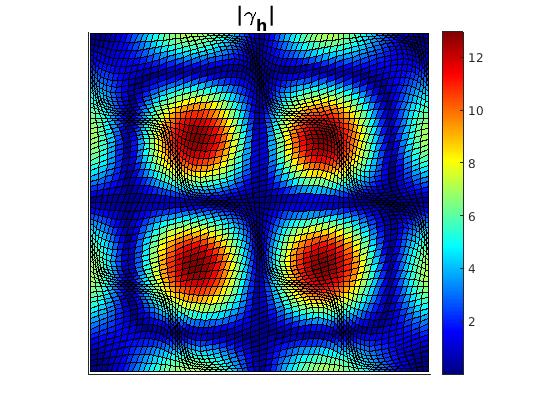}
			\hspace{-1.1cm}	\includegraphics[width=.41\textwidth]{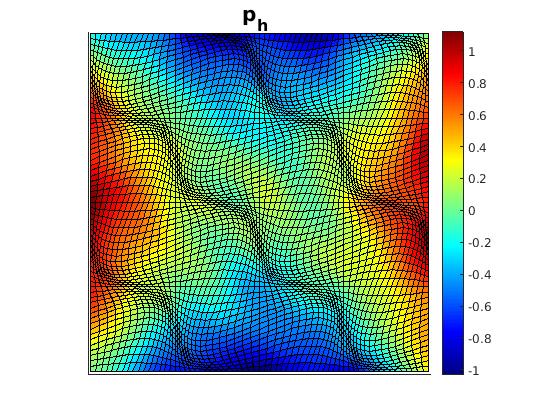}
			\hspace{-1.1cm}	\includegraphics[width=.41\textwidth]{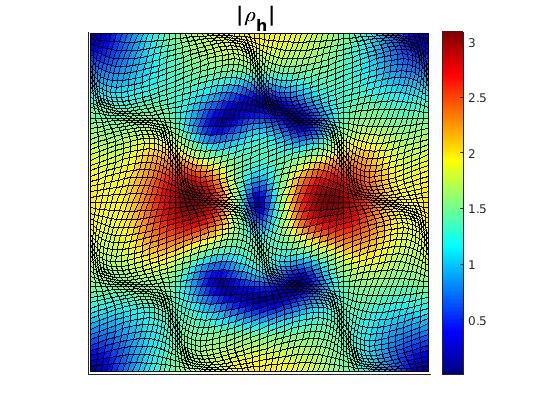}\\
			\hspace{-1.5cm}	\includegraphics[width=.41\textwidth]{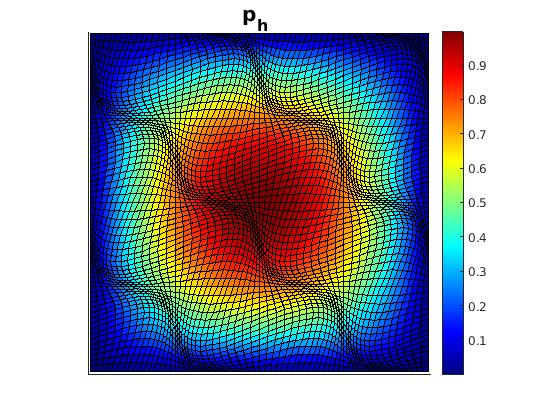}
		\end{tabular}
		\vspace{-.6cm}
		\caption{Example 1, snapshots of the first and fourth components of numerical stress, and velocity magnitude 
			(first row, left to right), vorticity magnitude,  pressure and pseudoheat magnitude (second row, left to right), 
			and temperature (third row), computed with $r = 0$
			in the mesh made of distortion quadrilaterals  with $h = 1/32$. }\label{fig1}
	\end{figure}

\subsection{Example 2: natural convection in a square cavity}
In the second example, we study the natural convection of a fluid within a square cavity featuring differentially heated walls. This phenomenon has been extensively investigated under various boundary conditions (see, e.g. \cite{bma-INMF-1994,dd-NHT-2006,d-INMF-1983}).
Here, we recall from \cite{dd-NHT-2006} the modified problem with dimensionless numbers: seek $ (\mathbf{u},p,\varphi) $ such that
\begin{equation}\label{eq:NC}
	\begin{array}{rcll}
		-\mathrm{Pr}\, \operatorname{\textbf{div}} (2\mu(\varphi)\textbf{e}(\mathbf{u}))+ (\boldsymbol{\nabla}\mathbf{u})\mathbf{u}+\nabla p -\mathrm{Ra}~\mathrm{Pr}~ \varphi\, \mathbf{g}&\,=\,&\mathbf{0} &\quad \qin\Omega\,,\\[1ex]
		\operatorname{div}\mathbf{u}&\,=\,&0 &\quad \qin\Omega\,,\\[1ex]
		- \operatorname{div}(\kappa(\varphi)\nabla\varphi)+\mathbf{u}\cdot\nabla\varphi &\,=\,&0 &\quad \qin\Omega\,,
	\end{array}
\end{equation}

where $\mathrm{Pr} $ represents the Prandtl number, denoting the ratio of momentum diffusivity to thermal diffusivity, while $\mathrm{Ra} $ signifies the Rayleigh number, which is defined as the ratio of buoyancy forces to viscosity forces multiplied by the Prandtl number. Therefore, we determine the cavity as $ \Omega =(0,1)^2 $ and consider the Prandtl and Rayleigh numbers as $\mathrm{Pr}=0.5 $ and $\mathrm{Ra}=4000 $, respectively.
Furthermore, the viscosity, thermal conductivity, and body force will be provided as follows:
\[
\mu(\varphi)\,=\,\exp(-\varphi),\qquad\kappa(\varphi)\,=\,\exp(\varphi)\,\qquad\mathbf{g}\,=\,(0,1)^{\mathtt{t}} \,,
\]
and the boundary conditions will be adopted as described in \cite{dd-NHT-2006}, that is
\[
\mathbf{u}_{D}(\mathbf{x})\,=\,\mathbf{0}\qan \varphi_{D}(\mathbf{x})\,=\,\dfrac{1}{2}(1-\cos(2\pi x_{1}))(1-x_{2})\,.
\]
The numerical experiments are performed with $r = 0$. In Fig. \ref{fig2} we have portrayed the approximate solutions, generated with the first-order mixed virtual
element family. All plots are in concordance with those obtained in \cite{ago-Calcolo-2018} and what is expected to be
observed from the physical point of view, in accordance to \cite{dd-NHT-2006}.

\begin{figure}[t!]
	\centering
	\begin{tabular}{ccc}
		\hspace{-1.5cm}	\includegraphics[width=.42\textwidth]{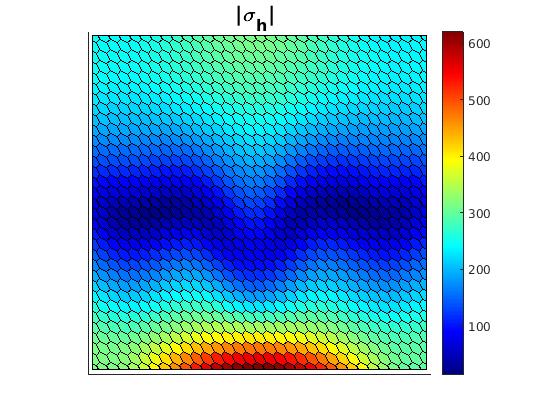}
		\hspace{-1.3cm}	\includegraphics[width=.42\textwidth]{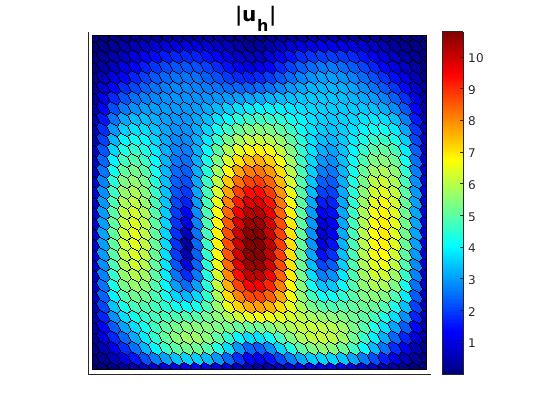}
		\hspace{-1.3cm}	\includegraphics[width=.42\textwidth]{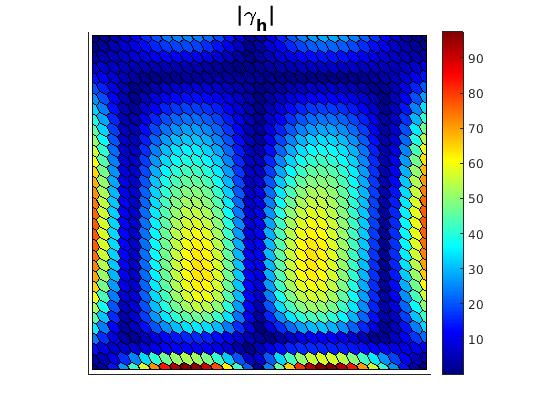}\\
		\hspace{-1.5cm}
		\includegraphics[width=.42\textwidth]{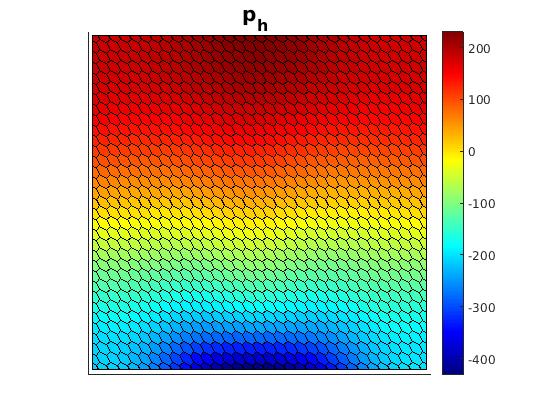}
		\hspace{-1.3cm}	\includegraphics[width=.42\textwidth]{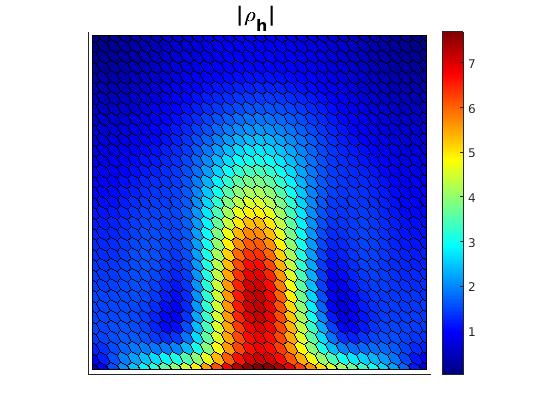}
		\hspace{-1.3cm}	\includegraphics[width=.42\textwidth]{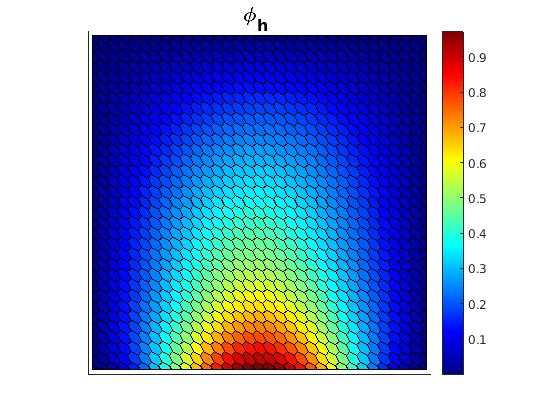}
	\end{tabular}
	\vspace{-.6cm}
	\caption{Example 2, snapshots of the stress, velocity, and vorticity magnitudes (first row, left to right) and the pressure and heat-flux vector magnitudes, and  temperature (second row, left to right) computed with $r = 0$
		in a mesh made of polygonals with $h = 3.030$e-2.}\label{fig2}
\end{figure}

\subsection{Example 3: natural convection driven phase change in non-convex disk-shaped geometry}
Here, we consider the steady regime of the phase change of a material adopting a 2D slice of a shell-and-tube geometry configuration, which is commonly used in thermal energy storage
systems. More precisely, the numerical experiments
are performed with $ r = 0 $ for the disk-shaped geometry domain $ \Omega $,
considering the problem \eqref{eq:NC} equipped to the mixed
boundary conditions for the energy equation that are described next. More precisely,
the domain includes four circles cavities with radius $ \frac{1}{8} $. The inner tubes are kept
hot with $ \varphi_{D}=1 $ and the outer shell is kept cold $ \varphi_{D}=-0.01 $.
 For the flow equations, all boundaries are equipped with no-slip velocity conditions. Fig. \ref{fig4} show images of the stress magnitude, velocity and temperature for parameters $ \mathrm{Pr}=1 $ and $ \mathrm{Ra}=100 $ on the refined polygonal mesh. No closed-form solution is available for this problem but all fields exhibit a well resolved behavior.

\begin{figure}[t!]
\centering
\begin{tabular}{cc}
\includegraphics[width=.42\textwidth]{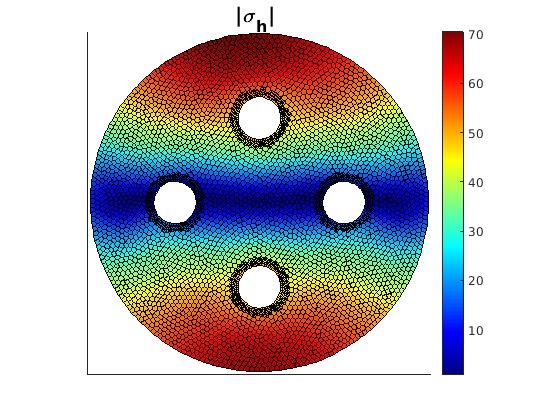}
\includegraphics[width=.42\textwidth]{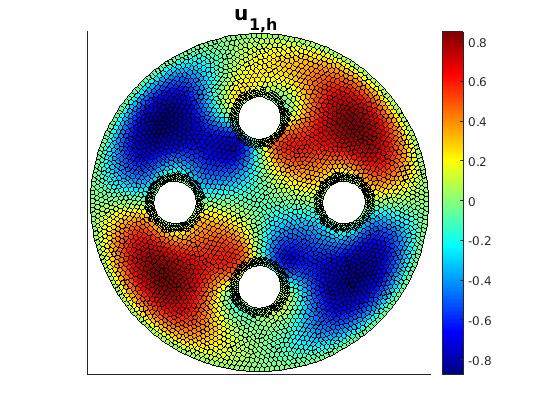}\\
\includegraphics[width=.42\textwidth]{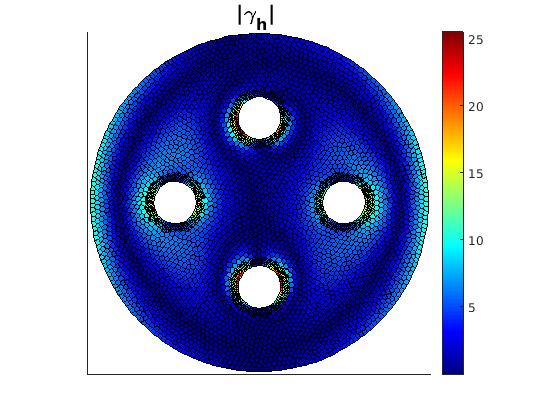}
\includegraphics[width=.42\textwidth]{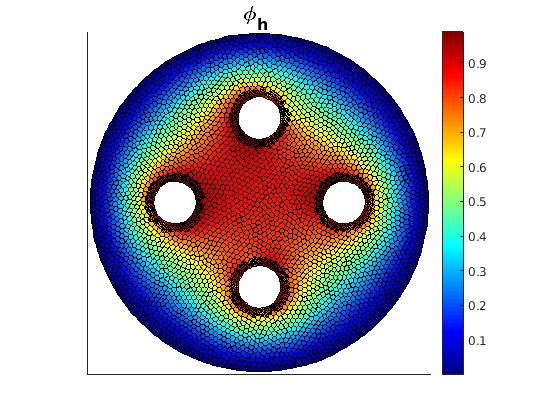}
\end{tabular}
\caption{Example 3, snapshots of stress, first component of velocity, vorticity and  
temperature for values $\mathrm{Ra}=1$e2, computed with $r = 0$ in a polgonal mesh with $h = 1.39$e-2.}\label{fig4}
\end{figure}

\bigskip

%
%
%
	

\end{document}